\pdfoutput=1
\documentclass[10pt]{article}
\usepackage{arxiv-setting}

\DeclareMathAlphabet{\mathcalligra}{T1}{calligra}{m}{n}
\DeclareMathAlphabet{\mathpzc}{OT1}{pzc}{m}{it}

\makeatletter
\def\blfootnote{\gdef\@thefnmark{}\@footnotetext}
\makeatother

\renewcommand{\epsilon}{\varepsilon}

\newcommand{\light}{\mathcal{L}}
\newcommand{\heavy}{\mathcal{H}}

\def\LM{{\textnormal{L}}} 

\title{Partial recovery and weak consistency in the non-uniform Hypergraph Stochastic Block Model\blfootnote{Author names are listed in alphabetical order.}}

\author{Ioana Dumitriu\thanks{Department of Mathematics, University of California, San Diego, La Jolla, CA 92093. Email: \texttt{idumitriu@ucsd.edu}}
\qquad
Hai-Xiao Wang\thanks{Department of Mathematics, University of California, San Diego, La Jolla, CA 92093. Email: \texttt{h9wang@ucsd.edu}}
\qquad
Yizhe Zhu\thanks{Department of Mathematics, University of California, Irvine, Irvine, CA 92697. Email: \texttt{yizhe.zhu@uci.edu}}
}

\date{This version: June 2026}

\begin{document}

\maketitle

\begin{abstract}
We consider the community detection problem in sparse random hypergraphs under the non-uniform \emph{Hypergraph Stochastic Block Model} (HSBM), a  general model of random networks with community structure and  higher-order interactions. When the random hypergraph has bounded expected degrees, we provide a  spectral algorithm that outputs a partition with at least a $\gamma$ fraction of the vertices  classified correctly, where $\gamma\in (0.5,1)$ depends on the \emph{Signal-to-Noise Ratio} (SNR) of the model. When the SNR  grows slowly as the number of vertices goes to infinity, our algorithm achieves weak consistency, which improves the previous  results in \cite{Ghoshdastidar2017ConsistencyOS} for non-uniform HSBMs.

Our spectral algorithm consists of  three major steps: (1) Hyperedge selection: select hyperedges of certain sizes to provide the maximal signal-to-noise ratio for the induced sub-hypergraph; (2) Spectral partition: construct a regularized adjacency matrix and obtain an approximate partition based on singular vectors; (3) Correction and merging: incorporate the hyperedge information from adjacency tensors to upgrade the error rate guarantee. The theoretical analysis of our algorithm relies on the concentration and regularization of the adjacency matrix for sparse non-uniform random hypergraphs, which can be of independent interest.
\end{abstract}

\tableofcontents

\section{Introduction}
Clustering is one of the central problems in network analysis and machine learning \cite{Newman2002RandomGM, Shi2000NormalizedCA, Ng2002SpectralCA}.  Many clustering algorithms make use of graph models, which represent pairwise relationships among data.  A well-studied probabilistic model is the \emph{Stochastic Block Model} (SBM), which was first introduced in \cite{Holland1983StochasticBM} as a random graph model that generates community structure with given ground truth for clusters so that one can study algorithm accuracy. The past decades have brought many notable results in the analysis of different algorithms and fundamental limits for community detection in SBMs in different settings \cite{Coja-Oghlan2010GraphPV, Vu2018ASimple, Guedon2016CommunityDC,Montanari2016SemidefinitePO}.  A major breakthrough was the proof of phase transition behaviors of community detection algorithms in various connectivity regimes \cite{Massoulie2014CommunityDT, Bordenave2018NonbacktrackingSO, Mossel2015ReconstructionAE, Mossel2018ProofOT, Mossel2016ConsistencyTF, Abbe2016ExactRI, Abbe2018ProofOT}. See the survey \cite{Abbe2018CommunityDA} for more references.

Hypergraphs can represent more complex relationships among data \cite{Benson2016HigherOO,Battiston2020NetworksBP},  including recommendation systems \cite{Bu2010MusicRB,Li2013NewsRV}, computer vision \cite{Govindu2005TensorDF,Wen2019LearningNH}, and biological networks \cite{Michoel2012AlignmentAI,Tian2009HypergraphLA}, and they have been shown empirically to have advantages over graphs \cite{Zhou2007LearningWH}.  Besides community detection problems, sparse hypergraphs and their spectral theory have also found applications in data science \cite{Jain2014ProvableTF,Zhou2021SparseRT,Harris2021DeterministicTC}, combinatorics \cite{Dumitriu2021SpectraOR,Friedman1995SecondEO,Soma2019SpectralSO}, and statistical physics \cite{Caceres2021SparseSYK, Sen2018OptimizationOS}.  

With the motivation from a broad set of applications,  many efforts have been made in recent years to study community detection on random hypergraphs. The \emph{Hypergraph Stochastic Block Model} (HSBM), as a generalization of graph SBM, was first introduced and studied in \cite{Ghoshdastidar2014ConsistencyOS}. In this model, we observe a random uniform hypergraph where each hyperedge appears independently with some given probability depending on the community structure of the vertices in the hyperedge. 

Succinctly put, the HSBM recovery problem is to find the ground truth clusters either approximately or exactly, given a sample hypergraph and estimates of model parameters. We may ask the following questions about the quality of the solutions (see \cite{Abbe2018CommunityDA} for further details in the graph case).
\begin{enumerate}
	\item \textbf{Exact recovery (strong consistency):} With high probability, recover all the clusters correctly. (up to permutation).
	\item \textbf{Almost exact recovery (weak consistency):} With high probability, find a partition of the vertex set such that at most $o(N)$ vertices are misclassified.
	\item \textbf{Partial recovery:} Given a fixed $\gamma \in (0.5,1)$, with high probability, find a partition of the vertex set such that at least a fraction $\gamma$ of the vertices are clustered correctly.
	\item \textbf{Weak recovery (detection):} With high probability, find a partition correlated with the true partition, better than a random guess.
\end{enumerate}

A faithful representation of a hypergraph is its adjacency tensor. However, most of the computations involving tensors are NP-hard \cite{Hillar2013MostTP}. Instead,
they considered spectral algorithms for exact recovery using hypergraph Laplacians. Subsequently, the analysis was extended to sparse, non-uniform hypergraphs \cite{Ghoshdastidar2015ProvableGT, Ghoshdastidar2015SpectralCU, Ghoshdastidar2017ConsistencyOS}. 

For exact recovery of uniform HSBMs, it was shown that the phase transition occurs in the regime of logarithmic expected degrees in \cite{Lin2017FundamentalSL, Chien2018CommunityDI, Chien2019MinimaxMR}. The thresholds are given for binary \cite{Kim2018StochasticBM, Gaudio2023CommunityDI} and multiple \cite{Zhang2023ExactRI} community cases, by generalizing the techniques in \cite{Abbe2016ExactRI, Abbe2015CommunityDI, Abbe2020EntrywiseEA}. After our work appeared on arXiv, thresholds for exact recovery on non-uniform HSBMs were given by \cite{Dumitriu2023OptimalAE, Wang2023ITLimits}. Strong consistency on the degree-corrected non-uniform HSBM was studied in \cite{Deng2023StrongCO}.

In \cite{Ghoshdastidar2014ConsistencyOS}, the authors used spectral clustering based on the hypergraph's Laplacian to recover HSBMs that are dense and uniform. Subsequently, they extended their results to sparse, non-uniform hypergraphs  in \cite{Ghoshdastidar2015ProvableGT,Ghoshdastidar2015SpectralCU,Ghoshdastidar2017ConsistencyOS}.

Spectral methods were considered in \cite{Chien2018CommunityDI, Ahn2016CommunityRI, Cole2020ExactRI, Zhang2023ExactRI, Yuan2021CommunityDI, Gaudio2023CommunityDI}, while \emph{semidefinite programming} (SDP) methods were analyzed in \cite{Kim2018StochasticBM,Lee2020RobustHC, Alaluusua2023MultilayerHC}.  Weak consistency for HSBMs was studied in \cite{Chien2018CommunityDI, Chien2019MinimaxMR, Ghoshdastidar2017ConsistencyOS, Ghoshdastidar2017UniformHP, Ke2020CommunityDF}.
 
For detection of the HSBM, the authors of \cite{Angelini2015SpectralDO} proposed a conjecture that the phase transition occurs in the regime of constant expected degrees. The positive part of the conjecture for the binary and multi-block case was solved in \cite{Pal2021ComunityDI} and \cite{Stephan2022SparseRH}, respectively. Their algorithms can output a partition better than a random guess when above the \emph{Kesten-Stigum} threshold, but can not guarantee the correctness ratio. \cite{Gu2023WeakRT, Gu2024CommunityDI} proved that detection is impossible and the \emph{Kesten-Stigum} threshold is tight for $\ell$-uniform hypergraphs with binary communities when $\ell =3, 4$, while KS threshold is not tight when $\ell \geq 7$, and some regimes remain unknown.

\subsection{Non-uniform hypergraph stochastic block model}
The non-uniform HSBM  was first studied in \cite{Ghoshdastidar2017ConsistencyOS}, which removed the uniform hypergraph assumption in previous works, and it is a more realistic model to study higher-order interaction on networks \cite{Lung2018HypergraphM,Wen2019LearningNH}. It can be seen as a superposition of several uniform HSBMs with different model parameters.  We first define the uniform HSBM in our setting and extend it to non-uniform hypergraphs. 

\begin{definition}[Uniform HSBM]
Let $\gV = [N]$ be the vertex set and $\ell \in \N^{+}$ be some fixed integer. Let $\gV_{1}, \ldots, \gV_{K}$ be a partition of $\gV$ into $K$ disjoint communities, each of size $N/K $ (assuming $N$ is divisible by $K$). For any set of $\ell$ distinct vertices $i_1,\dots i_{\ell}$, a hyperedge $\{i_1,\dots i_{\ell}\}$ is generated with probability $a_{\ell}/\binom{N - 1}{\ell-1} $ if all the vertices $i_1,\dots i_{\ell}$ are from the same community; otherwise with probability $b_{\ell} / \binom{N - 1}{\ell-1}$. We denote this distribution on the set of $\ell$-uniform hypergraphs as 
\begin{equation}
 \gH_{\ell}\sim \textnormal{HSBM}_{\ell} \Bigg( N, K, \frac{a_{\ell}}{\binom{N - 1}{\ell-1}}, \frac{b_{\ell}}{\binom{N - 1}{\ell-1}} \Bigg)\,.\label{eqn:uniform_HSBM_two_parameters}
\end{equation}
\end{definition}

\begin{definition}[Non-uniform HSBM]\label{def:sym_non_uni_HSBM}
Let $\sL = \{\ell: \ell \geq 2, \ell \in \N\}$ be a set of integers with finite cardinality. 
Let $\gH = (\gV, \gE)$ be a non-uniform random hypergraph, which can be
considered as a collection of $\ell$-uniform hypergraphs, i.e., $\gH = \cup_{\ell \in \sL} \gH_{\ell}$ with each $\gH_{\ell}$ sampled from \eqref{eqn:uniform_HSBM_two_parameters}.
\end{definition}
Examples of $2$-uniform and $3$-uniform \textnormal{HSBM}, and an example of non-uniform \textnormal{HSBM} with $\sL = \{2, 3\}$ and $K =3$ is displayed in \Cref{fig:non_uniform_partial} (a), \Cref{fig:non_uniform_partial} (b) and \Cref{fig:non_uniform_partial} (c) respectively.
\begin{figure}
     \centering
     \begin{subfigure}{0.3\textwidth}
         \centering
        \includegraphics[width=\textwidth]{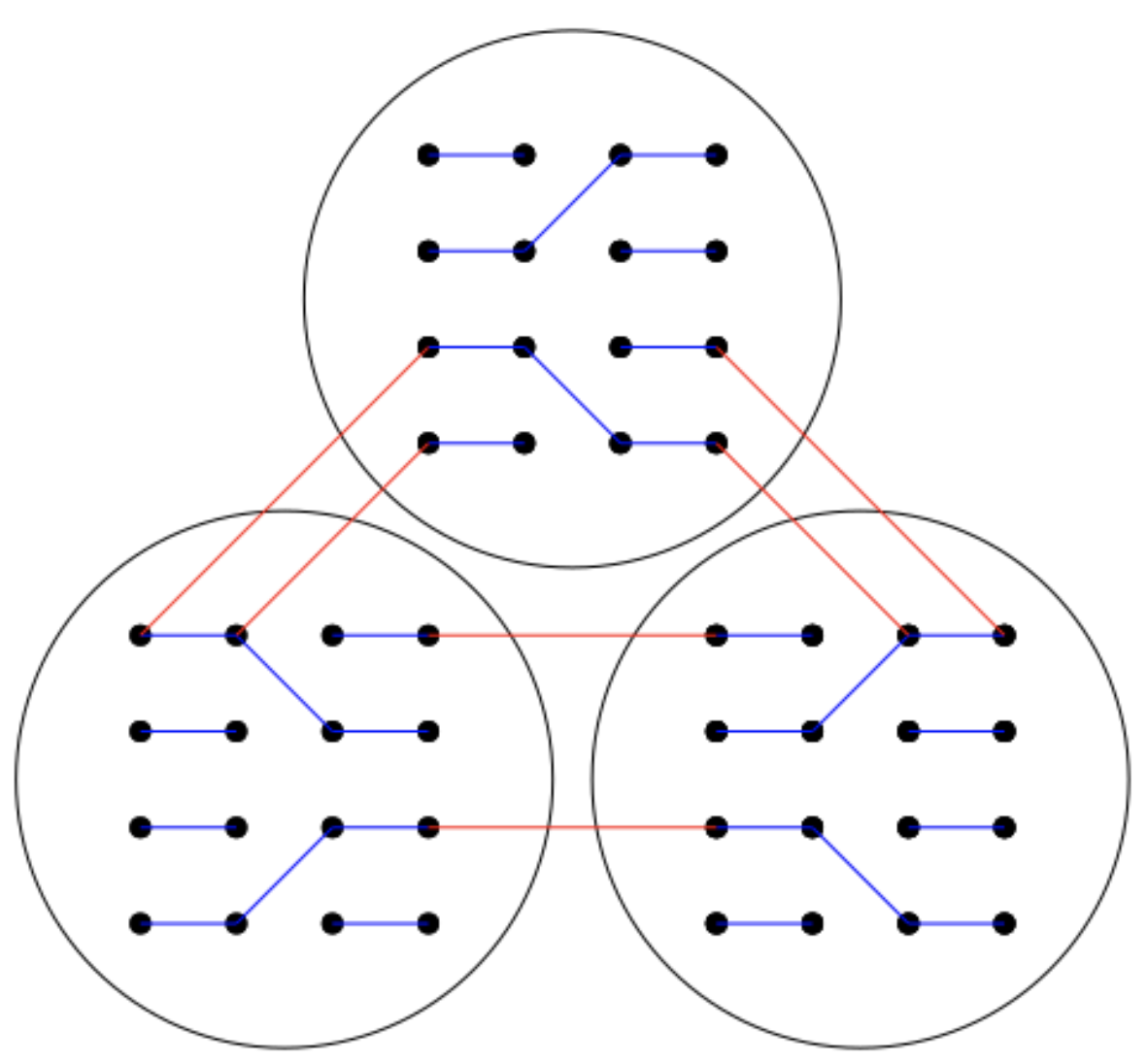}
        \subcaption{2-uniform HSBM}
     \end{subfigure}
     \begin{subfigure}{0.3\textwidth}
         \centering
        \includegraphics[width=\textwidth]{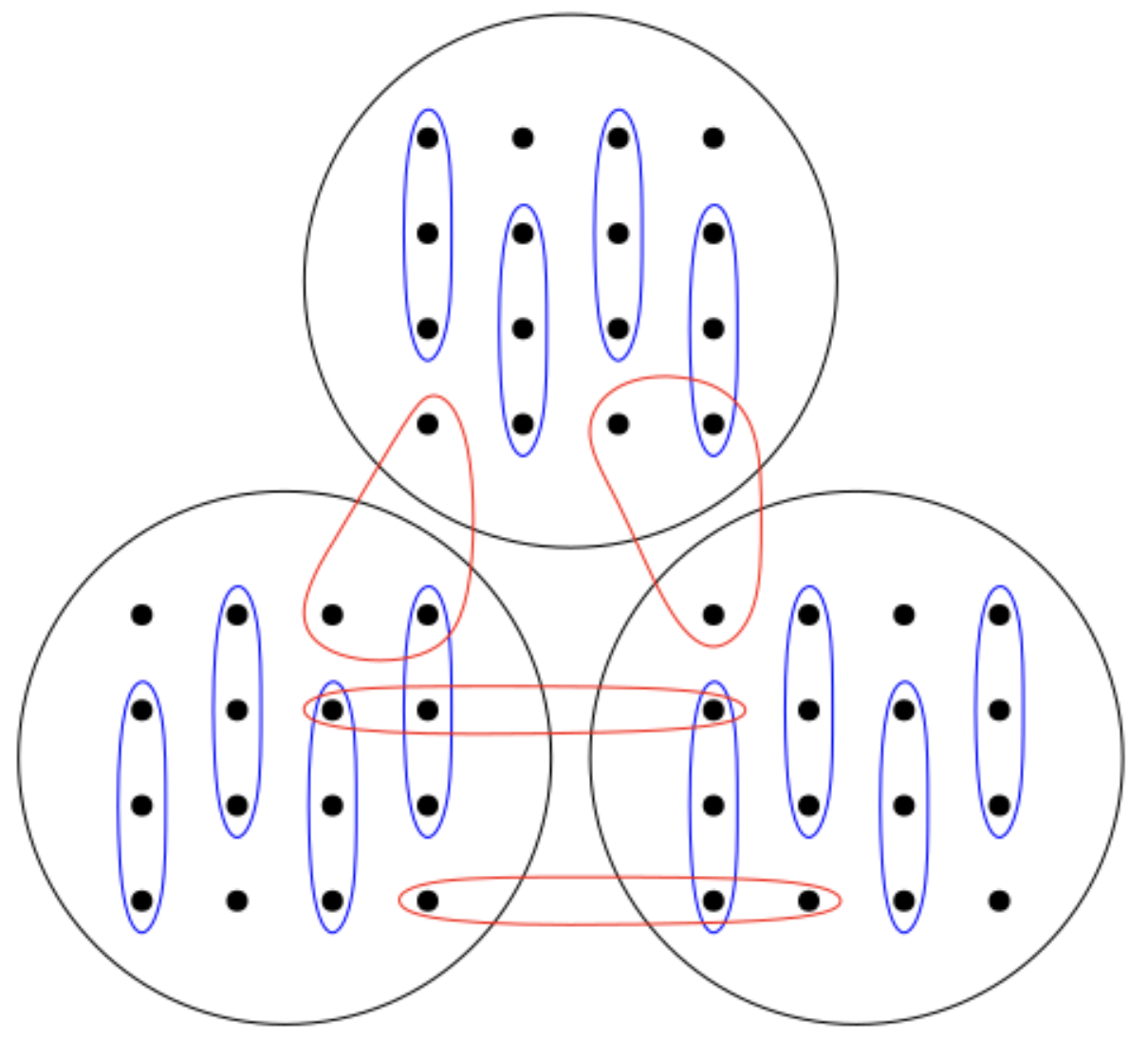}
        \subcaption{3-uniform HSBM}
     \end{subfigure}
     \begin{subfigure}{0.3\textwidth}
         \centering
        \includegraphics[width=\textwidth]{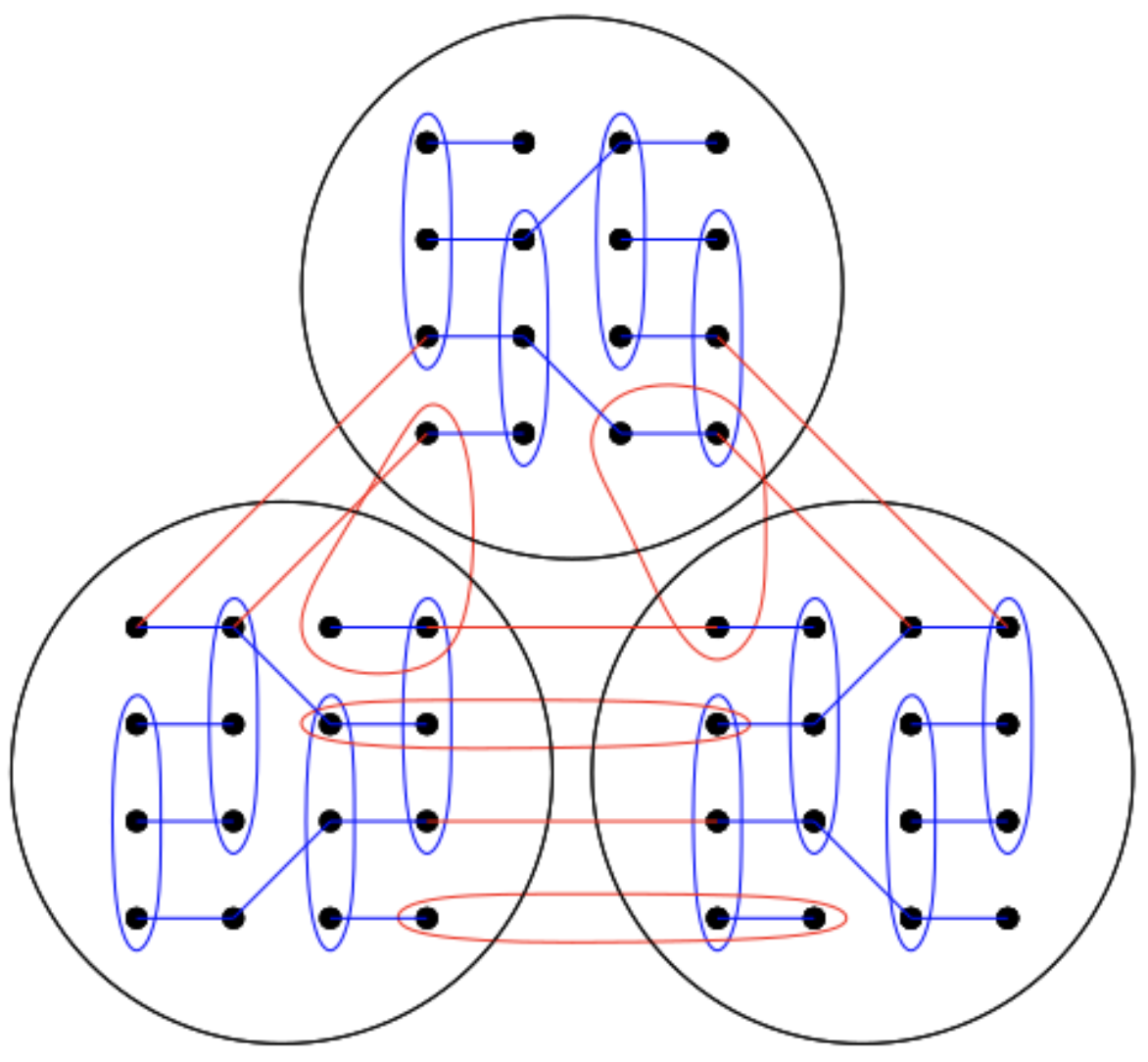}
        \subcaption{Non-uniform HSBM.}
     \end{subfigure}
     \caption{Uniform \textnormal{HSBM} and non-uniform \textnormal{HSBM}.}
     \label{fig:non_uniform_partial}
\end{figure}

\subsection{Main results}
To illustrate our main results, we first introduce the concepts \emph{$\gamma$-correctness} and \emph{Signal-to-Noise Ratio} (SNR) to measure the accuracy of the obtained partitions.
\begin{definition}[$\gamma$-correctness]
Let $\gV_1, \ldots, \gV_{K}$ be a partition of the vertex set $\gV$. A collection of subsets $\widehat{\gV}_{1},\ldots, \widehat{\gV}_{K}$ of $\gV$ is $\gamma$-correct if $|\gV_{k}\cap \widehat{\gV}_{k}|\geq \gamma |\gV_{k}|$ for all $k \in [K]$.
\end{definition}

\begin{definition}
    For model \ref{def:sym_non_uni_HSBM} under \Cref{ass:regimes_partial}, we define the \emph{signal-to-noise ratio}  as 
\begin{align}
    \mathrm{SNR}_{\sL}(K) \coloneqq\frac{ \Big[\sum_{\ell \in \sL} (\ell - 1)(a_{\ell} - b_{\ell})K^{-\ell + 1} \Big]^2 }{\sum_{\ell \in \sL} (\ell - 1)\big[(a_{\ell} - b_{\ell})K^{-\ell + 1} + b_{\ell} \big]} \,. \label{eqn:SNRk}
\end{align}
Let $\LM $ denote the maximum element in the set $\sL$. The following constant $\const_{\sL}(K, \nu)$ is used to characterize the accuracy of the clustering result,
\begin{align}
    \const_{\sL}(K, \nu)\coloneqq\frac{[\nu^{\LM -1} - (1-\nu)^{\LM -1}]^2 }{2^{3}\cdot (\LM  - 1)^2} \cdot \bigg( \indi{K = 2} +  \frac{1}{4^{\LM }} \cdot \indi{K \geq 3} \bigg)\label{eqn:CMk}
\end{align}
\end{definition}

Note that a non-uniform HSBM can be seen as a collection of observations for the same underlying community structure through several uniform HSBMs of different orders. A possible issue is that some uniform hypergraphs with small SNR might not be informative (if we observe an $\ell$-uniform hypergraph with parameters $a_{\ell}=b_{\ell}$,  including hyperedge information from it ultimately increases the noise). To improve our error rate guarantees, we start by adding a pre-processing step (\Cref{alg:parameter_preprocessing}) for hyperedge selection according to SNR and then apply the algorithm on the sub-hypergraph with maximal SNR.

We state the following assumption that will be used in our analysis of Algorithms~\ref{alg:binary_partition_partial} ($K =2$) and \ref{alg:multiple_partition_partial} ($K \geq 3$).

\begin{assumption}\label{ass:regimes_partial}
For each $\ell \in \sL$, assume $a_{\ell}, b_{\ell}$ are constants independent of $N$, and $a_\ell \geq b_{\ell}$. Let $\LM $ denote the maximum element of the set $\sL$. Given $\nu\in (1/K, 1)$, assume that there exists a universal constant $\const_{\eqref{eqn:degree_constant}}$ and some $\nu$-dependent constant $\const_{\nu} > 0$, such that 
\begin{subequations}
\begin{align}
    d\coloneqq \sum_{\ell \in \sL} (\ell - 1)a_{\ell} &\geq \const_{\eqref{eqn:degree_constant}}\,,\label{eqn:degree_constant}\\
    \sum_{\ell \in \sL} (\ell - 1)(a_{\ell}-b_{\ell}) &\geq \const_{\nu} \sqrt{d}\cdot K^{\LM -1} \cdot \Big( 2^{3} \cdot \indi{K = 2} + \sqrt{\log\big(K/(1 -\nu)\big)} \cdot \indi{K \geq 3} \Big)\,.\label{eqn:SNR_condition} 
  \end{align}
\end{subequations}
\end{assumption}

The constant $\const_{\eqref{eqn:degree_constant}}$ is not necessarily large; for example, $\const_{\eqref{eqn:degree_constant}} = (2^{1/\LM } - 1)^{-1/3}$ should suffice, but smaller $\const_{\eqref{eqn:degree_constant}}$ may work as well. Both two inequalities above prevent the hypergraph from being too sparse, while \eqref{eqn:SNR_condition} also requires that the difference between \emph{in-block} and \emph{across-blocks} densities is large enough. The choices of $\const_{\eqref{eqn:degree_constant}}$, $\const_{\nu}$ and their relationship will be discussed in \Cref{rem:two_constants}.

\subsubsection{The binary-community case}

We start with \Cref{alg:binary_partition_partial}, which outputs a $\gamma$-correct partition when the non-uniform HSBM $\gH$ is sampled from model \ref{def:sym_non_uni_HSBM} with only $2$ communities. Inspired by the innovative graph algorithm in \cite{Chin2015StochasticBM}, we generalize it to non-uniform hypergraphs while we provide a complete and detailed analysis at the same time.

\begin{algorithm}[h]
\caption{\textbf{Binary Partition}}\label{alg:binary_partition_partial}

    \KwData{The adjacency tensors $\{\tA^{(\ell)}\}_{\ell \in \sL}$, the number of communities $K$ and $\{a_{\ell}\}_{\ell \in \sL}$, $\{b_{\ell}\}_{\ell \in \sL}$.}

    {Run \textbf{\Cref{alg:parameter_preprocessing} Pre-processing} to obtain subset $\gS$ which achieves maximal $\mathrm{SNR}$.}\,
    
    {Randomly color all hyperedges red or blue with equal probability.}\,

    {Run \textbf{\Cref{alg:binary_block_spectral_partition} Spectral Partition} on the red hypergraph and output $\widehat{\gV}_{1}^{(0)}, \widehat{\gV}_{2}^{(0)}$.}\,\\
    {Run \textbf{\Cref{alg:binary_block_correction} Correction} on the blue hypergraph and output $\widehat{\gV}_{1}, \widehat{\gV}_2$.}\,
        
    \KwResult{The estimated partition $\widehat{\gV}_{1}, \widehat{\gV}_2$}
\end{algorithm}

\begin{theorem}[$K =2$]\label{thm:partial_binary}
Let $\nu\in(0.5, 1)$ and $\rho = 2\exp( -\const_{\sL}(2, \nu) \cdot \mathrm{SNR}_{\sL}(2))$ with $\mathrm{SNR}_{\sL}(K)$, $\const_{\sL}(K, \nu)$ defined in \eqref{eqn:SNRk},\eqref{eqn:CMk}, and let $\gamma = \max\{\nu,\, 1 - 2\rho\}$. Then under \Cref{ass:regimes_partial}, \Cref{alg:binary_partition_partial} outputs a $\gamma$-correct partition for sufficiently large $N$ with probability at least $1 - O(N^{-2})$.
\end{theorem}

\subsubsection{The multi-community case}
For the multi-community case ($K \geq 3$), another algorithm with more subroutines is developed in \Cref{alg:multiple_partition_partial}, which outputs a $\gamma$-correct partition with high probability. We state the result as follows.

\begin{algorithm}[h]
\caption{\textbf{General Partition}}\label{alg:multiple_partition_partial}

    \KwData{The adjacency tensors $\{\tA^{(\ell)}\}_{\ell \in \sL}$, the number of communities $K$ and $\{a_{\ell}\}_{\ell \in \sL}$, $\{b_{\ell}\}_{\ell \in \sL}$.}

    {Run \textbf{\Cref{alg:parameter_preprocessing} Pre-processing} to obtain subset $\gS$ which achieves maximal $\mathrm{SNR}$.}\,
    
    {Randomly color all hyperedges red or blue with equal probability.}\,

    {Randomly partition the vertex $\gV$ into two disjoint subsets $\gY$ and $\gZ$ by assigning $+1$ or $-1$ to each vertex with equal probability.}\,

    {Let $\rmB$ denote the adjacency matrix of the bipartite hypergraph between $\gY$ and $\gZ$ consisting only of the red hyperedges, with rows indexed by $\gZ$ and columns indexed by $\gY$.}\,
    
    {Run \textbf{\Cref{alg:multi_block_spectral_partition} Spectral Partition} on the red hypergraph and output $\widehat{\gU}_{1}^{(0)}, \cdots, \widehat{\gU}_{K}^{(0)}$.  \newline \textit{This step only uses the red hyperedges between vertices in $\gY$ and $\gZ$ and outputs approximate clusters for $\gU_{k} \coloneqq \gV_{k}\cap \gZ$ for each $k=1,\ldots, K$}}
    
    {Run \textbf{\Cref{alg:multi_block_correction} Correction} on the red hypergraph and output $\widehat{\gU}_{1}, \ldots, \widehat{\gU}_{K}$.}
    
    {Run \textbf{\Cref{alg:multi_block_merging} Merging} on the blue hypergraph and output $\widehat{\gV}_{1}, \ldots, \widehat{\gV}_{K}$. \newline\textit{This step only uses the blue hyperedges  between vertices in $\gY$ and $\gZ$ and assigns the vertices in $\gY$ to an appropriate approximate cluster.}}

    \KwResult{The estimated partition $\widehat{\gV}_{1}, \cdots, \widehat{\gV}_{K}$.}
\end{algorithm}

\begin{theorem}[$K \geq 3$]\label{thm:partial_multiple}
Let $\nu\in (1/K, 1)$ and $\rho = \exp( -\const_{\sL}(K, \nu) \cdot \mathrm{SNR}_{\sL}(K) \,)$ with $\mathrm{SNR}_{\sL}(K)$, $\const_{\sL}(K, \nu)$ defined in \eqref{eqn:SNRk},\eqref{eqn:CMk}, and let $\gamma = \max\{\nu,\, 1 - K\rho\}$. Then under \Cref{ass:regimes_partial}, \Cref{alg:multiple_partition_partial} outputs a $\gamma$-correct partition for sufficiently large $N$ with probability at least $1 - O(N^{-2})$.
\end{theorem}

The time complexities of Algorithms \ref{alg:binary_partition_partial} and \ref{alg:multiple_partition_partial} are $O(N^{3})$, with the bulk of time spent in Stage 1 by the spectral method.

To the best of our knowledge, Theorems~\ref{thm:partial_binary} and ~\ref{thm:partial_multiple} are the first results for partial recovery of non-uniform HSBMs.
When the number of blocks is $2$, \Cref{alg:binary_partition_partial} guarantees a better error rate for partial recovery as in \Cref{thm:partial_binary}. This happens because \Cref{alg:binary_partition_partial} does not need the merging routine in \Cref{alg:multiple_partition_partial}: if one of the communities is obtained, then the other one is also obtained via the complement.

\begin{remark}
Taking $\sL = \{2 \}$, \Cref{thm:partial_multiple} can  be reduced to  \cite[Lemma 9]{Chin2015StochasticBM} for the graph case. The failure probability $O(N^{-2})$ can be decreased to $O(N^{-p})$ for any $p>0$, as long as one is willing to pay the price by increasing the constants $\const_{\eqref{eqn:degree_constant}}$, $\const_v$ in \eqref{eqn:degree_constant}, \eqref{eqn:SNR_condition} respectively.
\end{remark}

Our Algorithms \ref{alg:binary_partition_partial} and \ref{alg:multiple_partition_partial} can be summarized in 3 steps:
\begin{enumerate}
    \item \textbf{Hyperedge selection}: select hyperedges of certain sizes to provide the maximal signal-to-noise ratio (SNR) for the induced sub-hypergraph.
    \item \textbf{Spectral partition}: construct a regularized adjacency matrix and obtain an approximate partition based on singular vectors (first approximation). 
    \item \textbf{Correction and merging}: incorporate the hyperedge information from adjacency tensors to upgrade the error rate guarantee (second, better approximation). 
\end{enumerate}
The algorithm requires the input of model parameters $a_{\ell},b_{\ell}$, which can be estimated by counting cycles in hypergraphs as shown in \cite{Mossel2015ReconstructionAE, Yuan2022TestingCS}. Estimation of the number of blocks can be done by counting the outliers in the spectrum of the non-backtracking operator, e.g., as shown (for different regimes and different parameters) in \cite{Saade2014SpectralCO, Angelini2015SpectralDO, Le2022EstimatingTN, Stephan2022SparseRH}. 

\subsubsection{Weak consistency}
Throughout the proofs for \Cref{thm:partial_binary} and \Cref{thm:partial_multiple}, we make only one assumption on the growth or finiteness of $d$ and $\mathrm{SNR}_{\sL}(K)$, and it happens in estimating the failure probability as noted in \Cref{rem:SNRlogn}. Consequently, the corollary below follows, which covers the case when $d$ and $\mathrm{SNR}_{\sL}(K)$ grow with $N$.

\begin{corollary}[Weak consistency]\label{cor:weak_consistency}
Let $\sL$ be a set of finite integers and $K$ be some fixed integer. If $\mathrm{SNR}_{\sL}(K)$ defined in $\eqref{eqn:SNRk}$ goes to infinity as $N \to \infty$ and $\mathrm{SNR}_{\sL}(K) = o(\log(N))$, then with probability $1- O(N^{-2})$, Algorithms \ref{alg:binary_partition_partial} and \ref{alg:multiple_partition_partial} output a partition with only $o(N)$  misclassified vertices.
 \end{corollary}
The paper 
\cite{Ghoshdastidar2017ConsistencyOS} also proves weak consistency for non-uniform HSBMs, but in a much denser regime than we do here ($d = \Omega(\log^2(N))$, instead of $d=\omega(1)$ , as in \Cref{cor:weak_consistency}). In fact, we now know  that strong consistency should be achievable in this denser regime, as \cite{Dumitriu2023OptimalAE} shows. When restricting to the uniform HSBM case, \Cref{cor:weak_consistency} achieves weak consistency under the same sparsity condition as in \cite{Ahn2018HypergraphSC}.

\begin{remark}\label{rem:SNRlogn}
To be precise, Algorithms \ref{alg:binary_partition_partial} and \ref{alg:multiple_partition_partial} work optimally in the $\mathrm{SNR}_{\sL} = o (\log(N))$ regime. When $\mathrm{SNR}_{\sL}(K) = \Omega(\log(N))$, it implies that $\rho = N^{-\Omega(1)}$, and one may have $e^{-N\rho} = \Omega(1)$ in \eqref{eqn:wrong_prob}, which may not decrease to $0$ as $N \to\infty$. Therefore the theoretical guarantees of Algorithms \ref{alg:multi_block_correction} and \ref{alg:multi_block_merging} may not remain valid. This, however, should not matter: in the regime when $\mathrm{SNR}_{\sL}(K) = \Omega(\log(N))$, strong (rather than weak) consistency is expected, as per  \cite{Dumitriu2023OptimalAE}. Therefore, the regime of interest for weak consistency is $\mathrm{SNR}_{\sL} = o(\log(N))$.

\end{remark}

\subsection{Comparison with existing results} Although many algorithms and theoretical results have been developed for hypergraph community detection, most of them are restricted to uniform hypergraphs, and few results are known for non-uniform ones. We will discuss the most relevant results. 



In \cite{Ke2020CommunityDF}, the authors studied the degree-corrected HSBM with general connection probability parameters by using a tensor power iteration algorithm and Tucker decomposition. Their algorithm achieves weak consistency for uniform hypergraphs when the average degree is $\omega(\log^2(N))$, which is the regime complementary to the regime we studied here. They discussed a way to generalize the algorithm to non-uniform hypergraphs, but the theoretical analysis remains open.  The recent paper \cite{Zhen2021CommunityDI} analyzed non-uniform hypergraph community detection by using hypergraph embedding and optimization algorithms and obtained weak consistency when the expected degrees are of $\omega(\log(N))$, again a complementary regime to ours. Results on spectral norm concentration of sparse random tensors were obtained in \cite{Cooper2020AdjacencySO, Nguyen2015TensorSV, Jain2014ProvableTF, Lei2020ConsistentCD, Zhou2021SparseRT}, but no provable tensor algorithm in the bounded expected degree case is known. Testing for the community structure for non-uniform hypergraphs was studied in \cite{Yuan2022TestingCS, Jin2021SharpIR}, which is a problem different from community detection.

In our approach, we relied on knowing the tensors for each uniform hypergraph. However, in computations, we only ran the spectral algorithm on the adjacency matrix of the entire hypergraph since the stability of tensor algorithms does not yet come with guarantees due to the lack of concentration, and for non-uniform hypergraphs, $\ell - 1$ adjacency tensors would be needed. This approach presented the challenge that, unlike for graphs, the adjacency matrix of a random non-uniform hypergraph has dependent entries, and the concentration properties of such a random matrix were previously unknown. We overcame this issue and proved concentration bounds from scratch down to the bounded degree regime. Similar to \cite{Feige2005SpectralTA, Le2017ConcentrationAR}, we provided here a regularization analysis by removing rows in the adjacency matrix with large row sums (suggestive of large degree vertices) and proving a concentration result for the regularized matrix down to the bounded expected degree regime (see \Cref{thm:regularization_concentration_partial}).


In terms of partial recovery for hypergraphs, our results are new, even in the uniform case. In \cite[Theorem 1]{Ahn2018HypergraphSC}, for uniform hypergraphs, the authors showed detection (not partial recovery) is possible when the average degree is $\Omega(1)$; in addition, the error rate is not exponential in the model parameters, but only polynomial.  Here, we mention two more results for the graph case. In the arbitrarily slowly growing degrees regime, it was shown in \cite{Zhang2016MinimaxRO,Fei2020AchievingTB} that the error rate in \eqref{eqn:SNRk} is optimal up to a constant in the exponent. In the bounded expected degrees regime, the authors in \cite{Mossel2016BeliefRR, Chin2021OptimalRO} provided algorithms that can asymptotically recover the optimal fraction of vertices, when the signal-to-noise ratio is large enough. It's an open problem to extend their analysis to obtain a minimax error rate for hypergraphs. 

In \cite{Ghoshdastidar2017ConsistencyOS}, the authors considered weak consistency in a non-uniform HSBM model with a spectral algorithm based on the hypergraph Laplacian matrix, and  showed that  weak consistency is achievable if the expected degree is of $\Omega(\log^2(N))$  with high probability \cite[Theorem 4.2]{Ghoshdastidar2014ConsistencyOS}. Their algorithm can't be applied to sparse regimes straightforwardly since the normalized Laplacian is not well-defined due to the existence of isolated vertices in the bounded degree case. In addition, our weak consistency results obtained here are valid as long as the expected degree is $\omega(1)$ and $o(\log(N))$, which is the entire set of problems on which weak consistency is expected. By contrast, in \cite{Ghoshdastidar2017ConsistencyOS}, weak consistency is shown only when the minimum expected degree is $\Omega(\log^2(N))$, which is a regime complementary to ours and where exact recovery should (in principle) be possible: for example, this is known to be an exact recovery regime in the uniform case \cite{Chien2019MinimaxMR,Kim2018StochasticBM,Lee2020RobustHC,Zhang2023ExactRI}.

Finally, although our analysis does not provide a way to identify the exact threshold for the phase transition of detection in the non-uniform hypergraph, we conjecture that (similar to \cite{Angelini2015SpectralDO} in the uniform case) $\mathrm{SNR}_{\sL}(K) =1$ in \eqref{eqn:SNRk} is the exact threshold for detection.

In subsequent works \cite{Dumitriu2023OptimalAE, Wang2023ITLimits} we proposed algorithms to achieve weak consistency. However, their methods can not cover the regime when the expected degree is $\Omega(1)$ due to the lack of concentration. Additionally, \cite{Wang2023ProjectedTP} proposed \emph{Projected Tensor Power Method} as the refinement stage to achieve strong consistency, as long as the first stage partition is partially correct, as ours.

\subsection{Organization of the paper}
In \Cref{sec:partial_prelim}, we include the definitions of adjacency matrices of hypergraphs. The concentration results for the adjacency matrices are provided in \Cref{sec:concentration_partial}. The algorithms for partial recovery are presented in Section \ref{sec:algorithm_blocks}. The proof for the correctness of our algorithms for Theorem \ref{thm:partial_multiple} and \Cref{cor:weak_consistency} are given in Section \ref{sec:analysis_multiple}. The proof of Theorem \ref{thm:partial_binary}, as well as the proofs of many auxiliary lemmas and useful lemmas in the literature, are provided in the supplemental materials.


\section{Preliminaries}\label{sec:partial_prelim}

\begin{definition}[Adjacency tensor]
An $\ell$-uniform hypergraph $\gH_{\ell}=(\gV, \gE_{\ell})$ can be associated to an order-$\ell$ adjacency tensor $\tA^{(\ell)}$. For each $\ell$-hyperedge $e = \{ i_1, \dots, i_{\ell} \}$, define the corresponding entry by
\begin{equation}
    \etA_{e}^{(\ell)} \coloneqq \etA_{[i_1,\dots, i_{\ell}]}^{(\ell)} = \indi{e \in \gE_{\ell}}\,. 
\end{equation}
\end{definition}

\begin{definition}[Adjacency matrix]
For some $\ell \in \N^{+}$, let $\tA^{(\ell)}$ denote the adjacency tensor associated to the uniform \emph{HSBM} $\gH_{\ell} = (\gV, \gE_{\ell})$ sampled from model \eqref{eqn:uniform_HSBM_two_parameters}. The adjacency matrix $\rmA^{(\ell)}\coloneqq [\ermA_{ij}^{(\ell)}]_{1\leq i, j\leq N}$ corresponding to $\gH_{\ell}$ is defined by
\begin{equation}
    \ermA_{ij}^{(\ell)} = \indi{i \neq j} \cdot \sum_{\substack{e\in \gE_{\ell}\\ \{i,j\}\subset e} } \etA_{e}^{(\ell)}\,. 
\end{equation}
For the non-uniform hypergraph $\gH$ sampled from model \ref{def:sym_non_uni_HSBM} with $\sL$ be a finte set of integers, its corresponding adjacency matrix $\rmA \coloneqq [\ermA_{ij}]_{1\leq i, j\leq N}$ is defined entrywisely as
\begin{equation}\label{eqn:adjacency_matrix_entries}
    \ermA_{ij} = \sum_{\ell \in \sL}\ermA_{ij}^{(\ell)}.
\end{equation}
\end{definition}

We compute the expectation of $\rmA$ first. In each $\ell$-uniform hypergraph $\gH_{\ell}$, two distinct vertices $i, j\in \gV$ with $i \neq j$ are picked arbitrarily since our model does not allow for loops. Assume for a moment $(N/K ) \in \N$, then the expected number of $\ell$-hyperedge containing $i$ and $j$ can be computed as follows. 
\begin{itemize}
    \item If $i$ and $j$ are from the same block, the $\ell$-hyperedge is sampled with probability $a_{\ell} /\binom{N - 1}{\ell-1}$ when the other $\ell - 2$ vertices are from the same block as $i$, $j$, otherwise with probability $b_{\ell} /\binom{N - 1}{\ell-1}$. Then
    \begin{align}
        \alpha_{\ell}\coloneqq\E\ermA_{ij}^{(\ell)} = \binom{\frac{N}{K} -2}{\ell - 2}  \frac{a_{\ell}}{\binom{N - 1}{\ell-1}} + \left[ \binom{N - 2}{\ell - 2} - \binom{\frac{N}{K} - 2}{\ell - 2} \right]\frac{b_{\ell}}{\binom{N - 1}{\ell-1}} \,.\notag
    \end{align}
    \item If $i$ and $j$ are not from the same block, we  sample the $\ell$-hyperedge with probability $b_{\ell} /\binom{N - 1}{\ell-1}$, and
   \begin{align}
         \beta_{\ell} \coloneqq\E\ermA_{ij}^{(\ell)} = \binom{N - 2}{\ell - 2} \frac{b_{\ell}}{\binom{N - 1}{\ell-1}}\,.\notag
    \end{align}
\end{itemize}
By assumption $a_\ell \geq b_{\ell}$,  then $\alpha_\ell \geq \beta_{\ell}$ for each $\ell \in \sL$. Summing over $\ell$, the \textit{expected adjacency} matrix under the $K$-block non-uniform $\textnormal{HSBM}$ can be written as 
\begin{align}\label{eqn:expected_adjacency_multiple}
  \E \rmA = \begin{bmatrix}
     \alpha \rmJ_{\frac{N}{K}}  &  \beta \rmJ_{\frac{N}{K}} & \cdots & \beta \rmJ_{\frac{N}{K}} \\
    \beta \rmJ_{\frac{N}{K}}  & \alpha \rmJ_{\frac{N}{K}} & \cdots & \beta \rmJ_{\frac{N}{K}} \\
    \vdots  &  \vdots & \ddots & \vdots \\
    \beta  \rmJ_{\frac{N}{K}} & \beta \rmJ_{\frac{N}{K}} & \cdots & \alpha \rmJ_{\frac{N}{K}}
  \end{bmatrix} - \alpha \rmI_{N}\,,
\end{align}
where $\rmJ_{\frac{N}{K}}\in \R^{\frac{N}{K} \times \frac{N}{K}}$ denotes the all-one matrix and
\begin{align}\label{eqn:alpha_beta}
    \alpha\coloneqq \sum_{\ell \in \sL} \alpha_{\ell}\,, \quad \beta\coloneqq \sum_{\ell \in \sL} \beta_{\ell}\,.
\end{align}

\begin{lemma}\label{lem:eigenvalues_HSBM}
The eigenvalues of $\E \rmA$ are given below:
    \begin{align*}
        \lambda_{1}(\E \rmA) =&\, \frac{N}{K}(\alpha + (K-1)\beta) - \alpha\,,\\
        \lambda_{j}(\E \rmA) =&\, \frac{N}{K}(\alpha - \beta) -\alpha \,,  \quad 2\leq j \leq K\,,\\
        \lambda_{j}(\E \rmA) =&\, - \alpha\,,  ~\quad \quad \quad K+1\leq j \leq N\,.
    \end{align*}
\end{lemma}
\Cref{lem:eigenvalues_HSBM} can be verified via direct computation. \Cref{lem:general_eigen_value_approximation} is used for approximately equi-partitions, meaning that eigenvalues of $\widetilde{\E \rmA}$ can be approximated by eigenvalues of $\E \rmA$ when $N$ is sufficiently large.
\begin{lemma}\label{lem:general_eigen_value_approximation}
For any partition $(\gV_1, \dots, \gV_{K})$ of $\gV$ where $N_i \coloneqq |\gV_i|$, consider the following matrix
\begin{align*}
\widetilde{\E \rmA} = \begin{bmatrix}
     \alpha \rmJ_{N_{1}} &  \beta \rmJ_{N_{1} \times N_{2}} & \cdots & \beta \rmJ_{N_{1} \times N_{K-1}} & \beta \rmJ_{N_{1} \times N_{K}} \\
    \beta \rmJ_{N_{2} \times N_{1}}  & \alpha \rmJ_{N_{2}} & \cdots & \beta \rmJ_{N_{2} \times N_{K-1}} & \beta \rmJ_{N_{2} \times N_{K}} \\
    \vdots  &  \vdots & \ddots & \vdots & \vdots\\
    \beta \rmJ_{N_{K-1} \times N_{1}}  & \beta \rmJ_{N_{K-1} \times N_{2}} & \cdots & \alpha \rmJ_{N_{K-1}} & \beta \rmJ_{N_{K-1} \times N_{K}} \\
    \beta  \rmJ_{N_{K} \times N_{1}} & \beta \rmJ_{N_{K} \times N_{2}} & \cdots & \beta \rmJ_{N_{K} \times N_{K-1}} & \alpha \rmJ_{N_{K}}
  \end{bmatrix} - \alpha \rmI_{N}\,.
\end{align*}
Assume that $N_i = \frac{N}{K} + O(\sqrt{N} \log(N))$ for all $k \in [K]$. Then, for all $1\leq i \leq K$, 
\begin{align*}
    \frac{|\lambda_i( \widetilde{\E \rmA} ) - \lambda_i(\E \rmA)|}{|\lambda_i(\E \rmA)| } = O\Big(N^{-\frac{1}{4}}\log^{\frac{1}{2}}(N)\Big)\,.
\end{align*}
\end{lemma}
Note that both $(\,\widetilde{\E \rmA} + \alpha \rmI_{N})$ and $(\,\E \rmA + \alpha \rmI_{N})$ are rank $K$ matrices, then $\lambda_i(\widetilde{\E \rmA}) = \lambda_i(\E \rmA) = -\alpha$ for all $(K+1) \leq i \leq N$. At the same time, $\mathrm{SNR}$ in \eqref{eqn:SNRk} is related to the following quantity
\begin{equation}
    \begin{aligned}
    &\, \frac{[\lambda_2(\E\rmA)]^{2}}{\lambda_1(\E\rmA)} = \frac{[(N - K)\alpha - N \beta]^2}{K[(N - K)\alpha + N(K-1)\beta]}
    = (1+o(1)) \cdot \frac{ \left[\sum_{\ell \in \sL} (\ell - 1)(a_{\ell} - b_{\ell})K^{-\ell + 1} \right]^2 }{\sum_{\ell \in \sL} (\ell - 1)\big((a_{\ell} - b_{\ell})K^{-\ell + 1} + b_{\ell} \big)}\notag.
    \end{aligned}
\end{equation}
When $\sL = \{2\}$ and $K$ is fixed, $\mathrm{SNR}$ in \eqref{eqn:SNRk} is equal to $\frac{(a - b)^2}{K[a + (K-1)b]}$, which corresponds to the $\mathrm{SNR}$ for the undirected graph in \cite{Chin2015StochasticBM}, see also \cite[Section 6]{Abbe2018CommunityDA}.

\section{Spectral norm concentration}\label{sec:concentration_partial}
The correctness of \Cref{alg:multiple_partition_partial} and \Cref{alg:binary_partition_partial} relies on the concentration of the adjacency matrix of $\gH$. The following two concentration results for general random hypergraphs are included, which are independent of HSBM model. The proofs are deferred to \Cref{sec:proof_concentration_partial}.

\subsubsection{Inhomogeneous Erd\H{o}s-R\'{e}nyi random hypergraphs}
\begin{definition}[Inhomogeneous Erd\H{o}s-R\`enyi hypergraph]\label{def:non_unifom_ER_hypergraph}
    Let $\tQ^{(\ell)} \in ([0, 1]^{N})^{\otimes \ell}$ be a symmetric probability tensor, i.e., $\etQ_{i_1, \ldots, i_{\ell}} = \etQ_{i_{\pi(1)}, \ldots, i_{\pi(\ell)}}$ for any permutation $\pi$ on $[\ell]$, where $\ell \geq 2$ is some finite integer. Let $\gH_{\ell} = (\gV, \gE_{\ell})$ denote inhomogeneous $\ell$-uniform Erd\H{o}s-R\'{e}nyi hypergraph associated with the probability tensor $\tQ^{(\ell)}$. Let $\tA^{(\ell)}$ denote the adjacency tensor of $\gH_{\ell}$, where each $\ell$-hyperedge $e = \{i_1, \ldots, i_{\ell}\}\subset \gV$ appears with probability $\P(\etA^{(\ell)}_{e} = 1) = \etQ^{(\ell)}_{i_1,\dots, i_{\ell}}$. Let $\sL = \{\ell \mid \ell \geq 2, \ell \in \N \}$ be a finite set of integers, where $\LM$ denotes its maximum element. An inhomogeneous non-uniform Erd\H{o}s-R\'{e}nyi hypergraph is the union of uniform ones, i.e., $\gH = \cup_{\ell \in \sL} \gH_{\ell}$.
\end{definition}

\begin{theorem}\label{thm:concentration_partial}
Let $\gH = \cup_{\ell \in \sL} \gH_{\ell}$ be an inhomogeneous non-uniform Erd\H{o}s-R\'{e}nyi hypergraph sampled from model \ref{def:non_unifom_ER_hypergraph}. Let $\{\tQ^{(\ell)}\}_{\ell\in \sL}$ denote the probability tensors such that $\etQ^{(\ell)}_{i_1,\dots, i_{\ell}} = d_{i_1,\dots,i_{\ell}} /\binom{N - 1}{\ell-1}$ for each $\ell \in sL$, and denote $d_{\ell} = \max d_{[i_1,\dots,i_{\ell}]}$. Let $\rmA$ be the adjacency matrix of $\gH$ constructed from the adjacency tensors $\{\tA^{(\ell)}\}_{\ell\in \sL}$. For some constant $\const_{\eqref{eqn:assumption_d_partial}} > 0$, suppsoe that the following assumption hold
\begin{align}
    d\coloneqq\sum_{\ell\in \sL}(\ell - 1)\cdot d_{\ell} \geq \const_{\eqref{eqn:assumption_d_partial}}\log(N)\,.\label{eqn:assumption_d_partial}
\end{align}
 Then for any $\theta >0$, there exists a constant 
 $
    \const_{\eqref{eqn:concentration_partial}} = 512 \LM (\ell - 1)(\theta + 6)\left[ 2 + (\ell - 1)(1 + \theta )/\const_{\eqref{eqn:assumption_d_partial}} \right]
 $
such that with probability at least $1-2N^{-\theta}-2e^{-N}$, the adjacency matrix $\rmA$ of $\gH$ satisfies  
\begin{align}
    \|\rmA - \E \rmA\|\leq \const_{\eqref{eqn:concentration_partial}} \sqrt{d}\,. \label{eqn:concentration_partial}
\end{align}
\end{theorem}

The inequality \eqref{eqn:concentration_partial} can be reduced to the result for graph case obtained in \cite{Feige2005SpectralTA,Lei2015ConsistencyOS} by taking $\sL = \{2\}$. The result for a uniform hypergraph is obtained in \cite{Lee2020RobustHC}. Note that $d$ is a fixed constant in our community detection problem, thus the Assumption \ref{eqn:assumption_d_partial} does not hold and the inequality \eqref{thm:concentration_partial} cannot be directly applied. However, we can still prove a concentration bound for a regularized version of $\rmA$, following the same strategy of the proof for \Cref{thm:concentration_partial}. 
\begin{definition}[Regularized matrix]\label{def:regularize_partial}
Given any $N \times N$ matrix $\rmA$ and an index set $\sI$, let $\rmA_{\sI}$ be the $N \times N$ matrix obtained from $\rmA$ by zeroing out the rows and columns not indexed by $\sI$. Namely,
\begin{equation}\label{eqn:restricted_adjacency_matrix_partial}
     (\rmA_{\sI})_{ij} = \indi{i, j\in \sI} \cdot \ermA_{ij}\,.
\end{equation}
\end{definition}
Note that every hyperedge of size $\ell$ containing $i$ is counted $(\ell - 1)$ times in the $i$-th row of $\rmA$. We define the $i$-th row sum of $\rmA$ as
\begin{align*}
    \textnormal{row}(i) \coloneqq\sum_{j}  \ermA_{ij} \coloneqq \sum_{j} \indi{i \neq j} \sum_{\ell \in \sL}\,\,\sum_{\substack{e\in \gE_{\ell}\\ \{i,j\}\subset e} }\etA_{e}^{(\ell)} = \sum_{\ell \in \sL}\,(\ell - 1) \sum_{e\in \gE_{\ell}:\, i\in e} \etA_e^{(\ell)}.
\end{align*}

\Cref{thm:regularization_concentration_partial} is the concentration result for the regularized $\rmA_{\sI}$, by zeroing out rows and columns corresponding to vertices with high row sums.

\begin{theorem}\label{thm:regularization_concentration_partial}
Following all the notations in \Cref{thm:concentration_partial},  for any constant $\tau >1$, define the index set
\begin{align}
    \sI=\{i\in \gV: \textnormal{row}(i)\leq \tau d\}.\label{eqn:regularization_set}
\end{align}
Let $\rmA_{\sI}$ be the regularized version of $\rmA$, as in \Cref{def:regularize_partial}. Then for any $\theta > 0$, there exists a constant $\const_{\tau}=2((5 \LM + 1)(\ell - 1) + \alpha_0 \sqrt{\tau})$ with $\alpha_0 = 16 + \frac{32}{\tau}(1+e^2)+128 \LM (\ell - 1)(\theta + 4)(1+\frac{1}{e^2})$, such that
\begin{align}
    \|\rmA_{\sI} - \E \rmA_{\sI}\| \leq \const_{\tau} \sqrt{d}\,\label{eqn:regularization_concentration_partial}
\end{align}
holds with probability at least $1-2(e/2)^{-N} -N^{-\theta}$.
\end{theorem}

\section{Algorithms blocks}\label{sec:algorithm_blocks}

In this section, we are going to present the algorithmic blocks constructing our main partition method (\Cref{alg:multiple_partition_partial}): pre-processing (\Cref{alg:parameter_preprocessing}), initial result by spectral method (\Cref{alg:multi_block_spectral_partition}), correction of blemishes via majority rule (\Cref{alg:multi_block_correction}), and merging (\Cref{alg:multi_block_merging}). 

\begin{algorithm}
\caption{\textbf{Pre-processing}}\label{alg:parameter_preprocessing}
\SetAlgoLined
\KwData{The set of integers $\sL$, and the parameters $\{a_{\ell}\}_{\ell \in \sL}$, $\{b_{\ell}\}_{\ell \in \sL}$.}
\For{each subset $\mathcal{S} \subset \sL$}{
    Let $\gH_{\mathcal{S}} = \bigcup_{\ell \in \mathcal{S}}\gH_{\ell}$ denote the restriction of the non-uniform hypergraph $\gH$ on $\mathcal{S}$. Compute
    \[
        \mathrm{SNR}_{\mathcal{S}}\coloneqq \frac{ \left[\sum_{\ell \in \mathcal{S}} (\ell - 1)(a_{\ell} - b_{\ell})K^{-\ell + 1} \right]^2 }{\sum_{\ell \in \mathcal{S}} (\ell - 1)\Big((a_{\ell} - b_{\ell})K^{- \ell + 1} + b_{\ell} \Big)}\;
    \]
}
Among all the $\mathcal{S}$, find the subset $\sL^{\star}$ which maximizes $\mathrm{SNR}_{\mathcal{S}}$, i.e.,
\[
    \sL^{\star}\coloneqq \arg\max_{\mathcal{S} \subset \sL} \,\mathrm{SNR}_{\mathcal{S}}\, ,
\]
where $\LM^{\star}$ denotes its maximal element.

\KwResult{$\sL^{\star}$}
\end{algorithm}

\subsection{Three or more communities ($K \geq 3$)}

\begin{algorithm}[h]
\caption{\textbf{Spectral Partition}}\label{alg:multi_block_spectral_partition}
\KwData{The tensors $\{\tA^{(\ell)}\}_{\ell \in \sL}$, the matrix $\rmA$, and the parameters $\{a_{\ell}\}_{\ell \in \sL}$, $\{b_{\ell}\}_{\ell \in \sL}$.}

Randomly label vertices in $\gY$ with $+1$ and $-1$ sign with equal probability, and partition $\gY$ into $2$ disjoint subsets $\gY_1$ and $\gY_2$\;

Let $\rmB_{1}$ (resp. $\rmB_{2}$) denote the adjacency matrices with all vertices in $\sZ \cup \gY_1$ (resp. $\sZ \cup \gY_2$), with rows indexed by $\sZ$ and columns indexed by $\gY_1$ (resp. $\gY_2$).\;

Pad $\rmB_{1}$, $\rmB_{2}$ with zeros to obtain the $N \times N$ matrices $\rmA_{1}$ and $\rmA_{2}$.\;

Let $d\coloneqq \sum_{\ell \in \sL}(\ell - 1)a_{\ell}$. Zero out all the rows and columns of $\rmA_{1}$ corresponding to vertices whose row sum is bigger than $20\LM d$, to obtain the matrix $(\rmA_{1})_{\sI_1}$\;

Find the space $\rmU$, spanned by the first $K$ left singular vectors of $(\rmA_{1})_{\sI_1}$\;

Randomly sample $s = 2K\log^2(N)$ vertices from $\gY_2$ without replacement. Denote the corresponding columns in $\rmA_{2}$ by $\rva_{i_1},\ldots, \rva_{i_s}$. For each $i\in \{i_1, \cdots, i_s\}$, project $\rva_{i} - \overline{\rva}$ onto $\rmU$, where the elements in vector $\overline{\rva}\in \R^{N}$ is defined by $\overline{\rva}(j) = \indi{j\in Z}\cdot (\overline{\alpha} + \overline{\beta})/2$, with $\overline{\alpha}$, $\overline{\beta}$ defined in \eqref{eqn:bar_alphabeta}\;

For each projected vector $\rmP_{\rmU}(\rva_{i}- \overline{\rva})$, identify the top $N/(2K)$ coordinates in value as a set $\gU_i$. Discard half of the $s$ sets $\gU_i$, those with the lowest blue hyperedge density in them\;

Sort the remaining sets according to blue hyperedge density and identify first $K$ distinct subsets $\widehat{\gU}_{1}^{(0)}, \ldots, \widehat{\gU}_{K}^{(0)}$ such that $|\widehat{\gU}_{j}^{(0)} \cap \widehat{\gU}_{k}^{(0)}| < \lceil {(1 - \nu)N}/{K} \rceil$ if $j\neq k$\;

\KwResult{The partition $\widehat{\gU}_{1}^{(0)}, \ldots, \widehat{\gU}_{K}^{(0)}$}
\end{algorithm}

\begin{algorithm}[h]
\KwData{The initial partition $\{\widehat{\gU}_{k}^{(0)}\}_{k\in [K]}$, the tensors $\{\tA^{(\ell)}\}_{\ell \in \sL}$, and the parameters $\{a_{\ell}\}_{\ell \in \sL}$, $\{b_{\ell}\}_{\ell \in \sL}$.}

\caption{\textbf{Correction}}\label{alg:multi_block_correction}
\For{each $u \in \sZ$}{
    Add $u$ to $\widehat{\gU}_{k}$ if the number of red hyperedges, which contains $u$ with the remaining vertices located in vertex set $\widehat{\gU}_{k}^{(0)}$, is at least $\mu_{\rm{C}}$ \eqref{eqn:mu_correction}\;
    Break ties arbitrarily\;
}
\KwResult{$\widehat{\gU}_{1}, \ldots, \widehat{\gU}_{K}$}
\end{algorithm}

\begin{algorithm}[h]
\caption{\textbf{Merging}}\label{alg:multi_block_merging}
\KwData{The partition $\{\widehat{\gU}_{k}\}_{k\in [K]}$, the tensors $\{\tA^{(\ell)}\}_{\ell \in \sL}$, and the parameters $\{a_{\ell}\}_{\ell \in \sL}$, $\{b_{\ell}\}_{\ell \in \sL}$.}

\For{all $u \in \gY$}{
    Add $u$ to $\widehat{\gV}_{k}$ if the number of blue hyperedges, which contains $u$ with the remaining vertices located in vertex set $\widehat{\gU}_{K}$, is at least $\mu_{\rm{M}}$ \eqref{eqn:mu_merge}\;
    Label the conflicts arbitrarily\;
}
\KwResult{The estimated sets $\widehat{\gV}_{1}, \ldots, \widehat{\gV}_{K}$}
\end{algorithm}

The proof of \Cref{thm:partial_multiple} is structured as follows.
\begin{lemma}\label{lem:spectral_clustering_accuracy_multiple}
  Under the assumptions of \Cref{thm:partial_multiple}, \Cref{alg:multi_block_spectral_partition} outputs a $\nu$-correct partition $\widehat{\gU}_{1}^{(0)}, \cdots, \widehat{\gU}_{K}^{(0)}$ of $\gZ = (\gZ \cap \gV_{1}) \cup \cdots \cup (\gZ \cap \gV_{K})$ with probability at least $1 - O(N^{-2})$.
\end{lemma}
Lines $4$ and $6$ contribute most complexity in \Cref{alg:multi_block_spectral_partition}, requiring $O(N^{3})$ and $O(N^{2}\log^2(N))$ each (technically, one should be able to get away with $O(N^2 \log(1/\epsilon))$ in line 4, for some desired accuracy $\epsilon$ to get the singular vectors). We will conservatively estimate the time complexity of \Cref{alg:multi_block_spectral_partition} as $O(N^3)$.

\begin{lemma}\label{lem:correction_accuracy_multiple}
Under the assumptions of \Cref{thm:partial_multiple}, for any $\nu$-correct partition $\widehat{\gU}_{1}^{(0)}, \cdots, \widehat{\gU}_{K}^{(0)}$ of $\gZ = (\gZ \cap \gV_{1}) \cup \cdots \cup (\gZ \cap \gV_{K})$ and the red hypergraph over $\gZ$,  \Cref{alg:multi_block_correction} computes a $\gamma_{\rm{C}}$-correct partition $\widehat{\gU}_{1}, \ldots, \widehat{\gU}_{K}$ with probability $1 -O(e^{-N\rho})$, while $\gamma_{\rm{C}} = \max\{\nu,\, 1 - K\rho\}$ with
$
    \rho \coloneqq k\exp\left( -\const_{\sL}(K, \nu) \cdot \mathrm{SNR}_{\sL}(K) \right)
$
where $\sL$ is obtained from \Cref{alg:parameter_preprocessing}, and $\mathrm{SNR}_{\sL}(K)$ and $\const_{\sL}(K, \nu)$ are defined in \eqref{eqn:SNRk}, \eqref{eqn:CMk}.
\end{lemma}

\begin{lemma}\label{lem:merging_accuracy_k}
    Given any $\nu$-correct partition $\widehat{\gU}_{1}, \ldots, \widehat{\gU}_{K}$ of $\gZ = (\gZ \cap \gV_{1}) \cup \cdots \cup (\gZ \cap \gV_{K})$ and the blue hypergraph between $\gY$ and $\gZ$, with probability $1 -O(e^{-N\rho})$,  \Cref{alg:multi_block_merging} outputs a $\gamma$-correct partition $\widehat{\gV}_1, \ldots, \widehat{\gV}_{K}$ of $\gV_1 \cup \gV_2 \cup \ldots \cup \gV_{K}$, while $\gamma =\max\{\nu,\, 1 - K\rho\}$.
\end{lemma}

The time complexities of Algorithms \ref{alg:multi_block_correction} and \ref{alg:multi_block_merging} are $O(N)$, since each vertex is adjacent to only constant many hyperedges.

\subsection{Binary communities ($K = 2$)}
The spectral partition step is given in \Cref{alg:binary_block_spectral_partition}, and the correction step is given in \Cref{alg:binary_block_correction}.
\begin{algorithm}[h]
\caption{\textbf{Spectral Partition}}\label{alg:binary_block_spectral_partition}
\KwData{The tensors $\{\tA^{(\ell)}\}_{\ell \in \sL}$, the matrix $\rmA$, and the parameters $\{a_{\ell}\}_{\ell \in \sL}$, $\{b_{\ell}\}_{\ell \in \sL}$.}

Zero out all the rows and the columns of $\rmA$ corresponding to vertices with row sums greater than $20\LM d$, to obtain the regularized matrix $\rmA_{\sI}$\;

Find the subspace $\rmU$, which is spanned by the eigenvectors corresponding to the largest two eigenvalues of $\rmA_{\sI}$\;

Compute $\rmP_{\rmU}\ones_{N}$, the projection of all-ones vector onto $\rmU$\;

Let $\rvv$ be the unit vector in $\rmU$ perpendicular to $\rmP_{\rmU}\ones_{N}$\;

Sort the vertices according to their values in $\rvv$. Let $\widehat{\gV}_{1}^{(0)} \subset V$ be the corresponding top $N/2$ vertices, and $\widehat{\gV}_{2}^{(0)} \subset V$ be the remaining $N/2$ vertices\;

\KwResult{$\widehat{\gV}_{1}^{(0)}$, $\widehat{\gV}_{2}^{(0)}$}
\end{algorithm}

\begin{lemma}\label{lem:spectral_clustering_accuracy_binary}
Under the conditions of \Cref{thm:partial_binary}, the \Cref{alg:binary_block_spectral_partition} outputs a $\nu$-correct partition $\widehat{\gV}_{1}^{(0)}, \widehat{\gV}_{2}^{(0)}$ of $\gV = \gV_1 \cup \gV_2$ with probability at least $1 - O(N^{-2})$.
\end{lemma}

\begin{algorithm}[h]
\caption{\textbf{Correction}}\label{alg:binary_block_correction}

\For{$v \in \widehat{\gV}_{1}^{(0)}$}{
    Label $v$ as ``bad'' if the number of blue-hyperedges which contain $v$ with the remaining vertices in $\widehat{\gV}_{2}^{(0)}$ is at least $\mu_{\rm{C}}$; otherwise label $v$ as ``good''.
}
\For{$v \in \widehat{\gV}_{2}^{(0)}$}{
    Label $v$ as ``bad'' if the number of blue-hyperedges which contain $v$ with the remaining vertices in $\widehat{\gV}_{1}^{(0)}$ is at least $\mu_{\rm{C}}$; otherwise label $v$ as ``good''\;
}
\For{$v \in \gV$}{
    \If{$v \in \widehat{\gV}_{i}^{(0)}$ and $v$ is good}{
        Assign $v$ to $\widehat{\gV}_{i}$\;
    }
    \Else{
        Assign $v$ to $\widehat{\gV}_{3-i}$\;
    }
}

\KwResult{$\widehat{\gV}_{1}$, $\widehat{\gV}_{2}$}
\end{algorithm}
\begin{lemma}\label{lem:correction_accuracy_binary}
    Given any $\nu$-correct partition $\gV_1^{(0)}, \gV_2^{(0)}$ of $\gV = \gV_1 \cup \gV_2$, with probability at least $1 -O(e^{-N\rho})$, the \Cref{alg:binary_block_correction} computes a $\gamma$-correct partition $\widehat{\gV}_1, \widehat{\gV}_2$ with $\gamma = \{\nu,\, 1 - 2\rho\}$ and $\rho = 2\exp( -\const_{\sL}(2, \nu) \cdot \mathrm{SNR}_{\sL}(2))$, where $\mathrm{SNR}_{\sL}(2)$ and $\const_{\sL}(2, \nu)$ are defined in \eqref{eqn:SNRk} and \eqref{eqn:CMk} respectively.
\end{lemma}
    
\section{Algorithm's correctness}\label{sec:analysis_multiple}
We are going to present the correctness of \Cref{alg:multiple_partition_partial} in this section. The correctness of \Cref{alg:binary_partition_partial} is deferred to \Cref{sec:analysis_binary}. We first introduce some definitions.

\subsubsection*{Vertex set splitting and adjacency matrix}
In \Cref{alg:multiple_partition_partial}, we first randomly partition the vertex set $\gV$ into two disjoint subsets $\gZ$ and $\gY$ by assigning $+1$ and $-1$ to each vertex independently with equal probability. Let $\rmB \in \R^{|\gZ|\times |\gY|}$ denote the submatrix of $\rmA$, while $\rmA$ was defined in \eqref{eqn:adjacency_matrix_entries}, where rows and columns of $\rmB$ correspond to vertices in $\gZ$ and $\gY$ respectively. Let $N_{j}$ denote the number of vertices in $\gZ \cap \gV_j$, where $\gV_j$ denotes the true partition with $|\gV_j| = \frac{N}{K}$ for all $j \in [K]$, then $N_j$ can be written as a sum of independent Bernoulli random variables, i.e.,
\begin{align}\label{eqn:dimension_ZcapVi}
    N_j = |\gZ \cap \gV_j| = \sum_{v\in \gV_j} \indi{v\in \gZ}\, ,
\end{align}
and $|\gY \cap \gV_j| = |\gV_j| - |\gZ \cap \gV_j| = \frac{N}{K} - N_j$ for each $j \in [K]$.
\begin{definition}
The splitting $\gV = \gZ \cup \gY$ is \textit{perfect} if $|\gZ \cap \gV_j| = |\gY\cap \gV_j| = N/(2K)$ for all $j \in [K]$. And the splitting $\gY = \gY_1 \cup \gY_2$ is \textit{perfect} if $|\gY_1\cap \gV_j| = |Y_2\cap \gV_j|  = N/(4K)$ for all $j \in [K]$.
\end{definition}
However, the splitting is \textit{imperfect} in most cases since the size of $\gZ$ and $\gY$ would not be exactly the same under the independence assumption. 
The random matrix $\rmB$ is parameterized by $\{\tA^{(\ell)}\}_{\ell \in \sL}$ and $\{N_j\}_{j=1}^{K}$. If we take expectation over $\{\tA^{(\ell)}\}_{\ell \in \sL}$ given the block size information $\{N_j\}_{j=1}^{K}$, then it gives rise to the expectation of the \textit{imperfect} splitting, denoted by $\widetilde{\rmB}$,
\begin{equation*}
    \widetilde{\rmB} \coloneqq \begin{bmatrix}
     \alpha \rmJ_{N_{1} \times ( \frac{N}{K} - N_{1})} & \beta \rmJ_{N_{1} \times (\frac{N}{K} - N_{2})}  & \dots & \beta \rmJ_{N_{1} \times (\frac{N}{K}- N_{K})}  \\
    \beta \rmJ_{N_{2} \times (\frac{N}{K}- N_{1})} & \alpha \rmJ_{N_{2} \times (\frac{N}{K}- N_{2})} &\dots & \beta \rmJ_{N_{2} \times (\frac{N}{K}- N_{K})}  \\
    \vdots & \vdots &  \ddots & \vdots \\
    \beta  \rmJ_{N_{K} \times (\frac{N}{K}- N_{1})}  & \beta  \rmJ_{N_{K} \times (\frac{N}{K}- N_{2})}  & \dots & \alpha \rmJ_{N_{K} \times (\frac{N}{K}- N_{K})}
  \end{bmatrix}\,, 
\end{equation*}
where $\alpha$, $\beta$ are defined in \eqref{eqn:alpha_beta}. In the \textit{perfect} splitting case, the dimension of each block is $N/(2K)\times N/(2K)$ since $\E N_j = N/(2K)$ for all $j \in [K]$, and the expectation matrix $\overline{\rmB}$ can be written as
\begin{equation*}
    \overline{\rmB} \coloneqq \begin{bmatrix}
     \alpha \rmJ_{\frac{N}{2K}} & \beta \rmJ_{\frac{N}{2K}} & \dots & \beta \rmJ_{\frac{N}{2K}}  \\
    \beta \rmJ_{\frac{N}{2K}} & \alpha \rmJ_{\frac{N}{2K}} & \dots & \beta \rmJ_{\frac{N}{2K}}  \\
    \vdots & \vdots & \ddots & \vdots \\
    \beta \rmJ_{\frac{N}{2K}} & \beta \rmJ_{\frac{N}{2K}} &\dots & \alpha \rmJ_{\frac{N}{2K}}
  \end{bmatrix}\,.
\end{equation*}

In \Cref{alg:multi_block_spectral_partition}, $\gY_1$ is a random subset of $\gY$ obtained by selecting each element with probability $1/2$ independently, and $\gY_2 = \gY \setminus \gY_1$. Let $\widetilde{N}_{j}$ denote the number of vertices in $\gY_1\cap \gV_j$, then $\widetilde{N}_{j}$ can be written as a sum of independent Bernoulli random variables, 
\begin{align}\label{eqn:dimension_Y1capVi}
    \widetilde{N}_{j} = |\gY_1 \cap \gV_j| = \sum_{v\in \gV_j} \indi{v\in \gY_1}\,,
\end{align}
and $|\gY_2\cap \gV_j| = |\gV_j| - |\gZ \cap \gV_j| - |\gY_1\cap \gV_j| =  N/K  - N_j - \widetilde{N}_{j}$ for all $j \in [K]$. 

Similarly,  $\E \widetilde{N}_{j} = N/(4K)$, which means that this splitting is expected to be perfect. However, we still obtain \textit{imperfect} splitting in most cases due to the independence assumption.

\begin{remark}
In \cite{Chin2015StochasticBM}, the authors assumed the \textit{perfect} splitting throughout the proof, which is not necessarily true. Our work provides a rigorous analysis of the algorithm provided in \cite{Chin2015StochasticBM}, without the \textit{perfect} splitting assumption.
\end{remark}

\subsubsection*{Induced sub-hypergraph}\label{subsubsec:induced_sub_hypergraph}
For the non-uniform hpyergraph $\gH = (\gV, \gE)$,  consider the vertex subset $(\gY_1 \cup \gZ) \subset \gV$, which leads our attention to the corresponding induced sub-hypergraph. 

\begin{definition}[Induced sub-hypergraph]\label{def:induced sub-hypergraph}
Let $\gH = (\gV, \gE)$ be a non-uniform random hypergraph and $\gS\subset \gV$ be any subset of the vertices of $\gH$. Then the \textit{induced sub-hypergraph} $\gH[\gS]$ is the hypergraph whose vertex set is $\gS$ and whose hyperedge set $\gE[\gS]$ consists of all of the edges in $\gE$ that have all endpoints located in $\gS$.
\end{definition}

\begin{remark}
Note that $\gH$ can be considered as a collection of $\ell$-uniform hypergraphs for varying $\ell$, a.k.a., $\gH = \cup_{\ell \in \sL} \gH_{\ell}$. The decomposition works for $\gH[\gS]$ as well, a.k.a. $\gH[\gS] = \cup_{\ell \in \sL} \gH_{\ell}[\gS]$ and $\gE[\gS] = \cup_{\ell \in \sL} \gE_{\ell}[\gS]$
\end{remark}

Let $\gH[\gY_1 \cup \gZ]$(resp. $\gH[\gY_2 \cup \gZ]$) denote the induced sub-hypergraph on vertex set $\gY_1 \cup \gZ$ (resp. $\gY_2 \cup \gZ$), and $\rmB_1 \in \R^{|\gZ|\times |\gY_1|}$ (resp. $\rmB_2 \in \R^{|\gZ|\times |\gY_2|}$) denote the adjacency matrix corresponding to the sub-hypergraph, where the rows and columns of $\rmB_1$ (resp. $\rmB_2$) are corresponding to elements in $\gZ$ and $\gY_1$ (resp., $\gZ$ and $\gY_2$). Therefore,  $\rmB_1$ and $\rmB_2$ are parameterized by $\{\tA^{(\ell)}\}_{\ell \in \sL}$, $\{N_j\}_{j=1}^{K}$ and $\{\widetilde{N}_{j}\}_{j=1}^{K}$, and the entries in $\rmB_1$ are independent of the entries in $\rmB_2$, due to the independence of hyperedges. If we take expectation over $\{\tA^{(\ell)}\}_{\ell \in \sL}$ conditioning on $\{N_{j}\}_{j=1}^{K}$ and $\{\widetilde{N}_{j}\}_{j=1}^{K}$, then it gives rise to the expectation of the \textit{imperfect} splitting, denoted by $\widetilde{\rmB}_1$,
\begin{equation}\label{eqn:tilde_B1}
    \widetilde{\rmB}_1 \coloneqq \begin{bmatrix}
     \widetilde{\alpha}_{11} \rmJ_{N_{1} \times \widetilde{N}_{1}} & \dots & \widetilde{\beta}_{1k} \rmJ_{N_{1} \times \widetilde{N}_{K}}  \\
    \vdots & \ddots & \vdots \\
    \widetilde{\beta}_{k1}  \rmJ_{N_{K} \times \widetilde{N}_{1}}  & \dots & \widetilde{\alpha}_{kk} \rmJ_{N_{K} \times \widetilde{N}_{K}} 
  \end{bmatrix}\,,
\end{equation}
where for $1\leq i,  j \leq K$, the entries $\widetilde{\alpha}_{ij}$ and $\widetilde{\beta}_{ij}$ are defined as follows:
\begin{subequations}
\begin{align}
    \widetilde{\alpha}_{ii} \coloneqq&\, \sum_{\ell \in \sL}\left\{ \binom{N_i + \widetilde{N}_i - 2}{\ell - 2} \frac{a_{\ell} - b_{\ell}}{\binom{N - 1}{\ell-1}} + \binom{\sum_{k=1}^{K}(N_{k} + \widetilde{N}_{k}) - 2}{\ell - 2} \frac{b_{\ell}}{\binom{N - 1}{\ell-1}} \right\}\,,\label{eqn:tilde_alpha} \\
    \widetilde{\beta}_{ij} \coloneqq&\, \sum_{\ell \in \sL} \binom{\sum_{k=1}^{K}(N_{k} + \widetilde{N}_{k}) - 2}{\ell - 2}\frac{b_{\ell}}{\binom{N - 1}{\ell-1}}\,,\quad i\neq j, \label{eqn:tilde_beta}
\end{align}
\end{subequations}
The expectation of the \textit{perfect} splitting, denoted by $\overline{\rmB}_1$, can be written as
\begin{equation}\label{eqn:bar_B1}
    \overline{\rmB}_1 \coloneqq \begin{bmatrix}
     \overline{\alpha} \rmJ_{\frac{N}{2K} \times \frac{N}{4K}} & \overline{\beta} \rmJ_{\frac{N}{2K} \times \frac{N}{4K}} & \dots & \overline{\beta} \rmJ_{\frac{N}{2K} \times \frac{N}{4K}} \\
    \overline{\beta} \rmJ_{\frac{N}{2K} \times \frac{N}{4K}} & \overline{\alpha} \rmJ_{\frac{N}{2K} \times \frac{N}{4K}} & \dots & \overline{\beta} \rmJ_{\frac{N}{2K} \times \frac{N}{4K}}  \\
    \vdots & \vdots & \ddots & \vdots \\
    \overline{\beta} \rmJ_{\frac{N}{2K} \times \frac{N}{4K}} & \overline{\beta} \rmJ_{\frac{N}{2K} \times \frac{N}{4K}} &\dots & \overline{\alpha} \rmJ_{\frac{N}{2K} \times \frac{N}{4K}}
  \end{bmatrix}\,,
\end{equation}
where
\begin{align}\label{eqn:bar_alphabeta}
    \overline{\alpha} \coloneqq \sum_{\ell \in \sL}\left\{ \binom{ \frac{3N}{4K} - 2}{\ell - 2} \frac{a_{\ell} - b_{\ell}}{\binom{N - 1}{\ell-1}} + \binom{\frac{3N}{4} - 2}{\ell - 2}\frac{b_{\ell}}{\binom{N - 1}{\ell-1}} \right\}\,,\quad  \overline{\beta} \coloneqq \sum_{\ell \in \sL} \binom{\frac{3N}{4} - 2}{\ell - 2}\frac{b_{\ell}}{\binom{N - 1}{\ell-1}}\,.
\end{align}
The matrices $\widetilde{\rmB}_2, \overline{\rmB}_2$ can be defined similarly, since dimensions of $|\gY_2\cap \gV_j|$ are also determined by $N_j$ and $\widetilde{N}_{j}$. Obviously, $\overline{\rmB}_2 = \overline{\rmB}_1$ since $\E \widetilde{N}_{i} = \E (N/K - N_{j} - \widetilde{N}_j) = N/(4K)$ for all $j \in [K]$.

\subsubsection*{Fixing Dimensions}
The dimensions of $\widetilde{\rmB}_1$ and $\widetilde{\rmB}_2$, as well as blocks they consist of, are not deterministic---since $N_i$ and $\widetilde{N}_{j}$, defined in \eqref{eqn:dimension_ZcapVi} and \eqref{eqn:dimension_Y1capVi} respectively, are sums of independent random variables. As such, we cannot directly compare them. In order to overcome this difficulty, we embed $\rmB_1$ and $\rmB_2$ into the following $N \times N$ matrices:
\begin{equation}\label{eqn:A1A2}
    \rmA_{1}\coloneqq \begin{bmatrix}
     \bzero_{|\gZ|\times |\gZ|} & \rmB_1 & \bzero_{|\gZ|\times |\gY_2|} \\
     \bzero_{|\gY|\times |\gZ|} & \bzero_{|\gY|\times |\gY_1|}  & \bzero_{|\gY|\times |\gY_2|} 
  \end{bmatrix}\, ,\quad 
  \rmA_{2}\coloneqq \begin{bmatrix}
     \bzero_{|\gZ|\times |\gZ|} & \bzero_{|\gZ|\times |\gY_1|} & \rmB_2 \\
     \bzero_{|\gY|\times |\gZ|} & \bzero_{|\gY|\times |\gY_1|}  & \bzero_{|\gY|\times |\gY_2|} 
    \end{bmatrix}\,.
\end{equation}
Note that $\rmA_{1}$ and $\rmA_{2}$ have the same size. Also by definition, the entries in $\rmA_{1}$ are independent of the entries in $\rmA_{2}$. If we take expectation over $\{\tA^{(\ell)}\}_{\ell \in \sL}$ conditioning on $\{N_{j}\}_{j=1}^{K}$ and $\{\widetilde{N}_{j}\}_{j=1}^{K}$, then we obtain the expectation matrices of the \textit{imperfect} splitting, denoted by $\widetilde{\rmA}_1$(resp. $\widetilde{\rmA}_2$), written as
\begin{equation}\label{eqn:tilde_A1A2}
    \widetilde{\rmA}_1\coloneqq \begin{bmatrix}
     \bzero_{|\gZ|\times |\gZ|} & \widetilde{\rmB}_1 & \bzero_{|\gZ|\times |\gY_2|} \\
     \bzero_{|\gY|\times |\gZ|} & \bzero_{|\gY|\times |\gY_1|}  & \bzero_{|\gY|\times |\gY_2|} 
  \end{bmatrix}\, ,\quad 
  \widetilde{\rmA}_2\coloneqq \begin{bmatrix}
     \bzero_{|\gZ|\times |\gZ|} & \bzero_{|\gZ|\times |\gY_1|} & \widetilde{\rmB}_2 \\
     \bzero_{|\gY|\times |\gZ|} & \bzero_{|\gY|\times |\gY_1|}  & \bzero_{|\gY|\times |\gY_2|} 
    \end{bmatrix}\,.
\end{equation}
The expectation matrix of the \textit{perfect} splitting, denoted by $\overline{\rmA}_{1}$(resp. $\overline{\rmA}_{2}$), can be written as
\begin{equation}\label{eqn:bar_A1A2}
    \overline{\rmA}_{1}\coloneqq \begin{bmatrix}
     \bzero_{ \frac{N}{2} \times \frac{N}{2}} & \overline{\rmB}_1 & \bzero_{\frac{N}{2}\times \frac{N}{4}} \\
     \bzero_{\frac{N}{2}\times \frac{N}{2}} & \bzero_{\frac{N}{2}\times \frac{N}{4}}  & \bzero_{\frac{N}{2}\times \frac{N}{4}} 
  \end{bmatrix}\, ,\quad 
  \overline{\rmA}_{2}\coloneqq \begin{bmatrix}
     \bzero_{ \frac{N}{2} \times \frac{N}{2}} & \bzero_{\frac{N}{2}\times \frac{N}{4}} & \overline{\rmB}_2\\
     \bzero_{\frac{N}{2}\times \frac{N}{2}} & \bzero_{\frac{N}{2}\times \frac{N}{4}}  & \bzero_{\frac{N}{2}\times \frac{N}{4}} 
  \end{bmatrix}\,.
\end{equation}
Obviously, $\widetilde{\rmA}_i$ and $\widetilde{\rmB}_i$(resp. $\overline{\rmA}_i$ and $\overline{\rmB}_i$) have the same non-zero singular values for $i = 1, 2$. In the remaining of this section, we will deal with $\widetilde{\rmA}_i$ and $\overline{\rmA}_i$ instead of $\widetilde{\rmB}_i$ and $\overline{\rmB}_i$ for $i = 1, 2$.

\subsection{Spectral Partition: Proof of \Cref{lem:spectral_clustering_accuracy_multiple}}\label{subsec:spectral_partition_proof}

\subsubsection{Proof Outline}
Recall that $\rmA_{1}$ is defined as the adjacency matrix of the induced sub-hypergraph $\gH[\gY_1 \cup \gZ]$ in \Cref{subsubsec:induced_sub_hypergraph}. Consequently, the index set should contain information only from $\gH[\gY_1 \cup \gZ]$. Define the index sets
\begin{align}
 \sI = \Big\{i\in \gV: \mathrm{row}(i)\leq 20\LM d \Big\}, \quad 
    \sI_{1} = \Big\{i\in \gV: \mathrm{row}(i)\big|_{\gY_1 \cup \gZ} \leq 20\LM d \Big\}\,, \label{eqn:index_sets_I1}
\end{align}
where $d = \sum_{\ell \in \sL}(\ell - 1)a_{\ell}$, and $\mathrm{row}(i)\big|_{\gY_1 \cup \gZ}$ is the row sum of $i$ on $\gH[\gY_1 \cup \gZ]$. We say $\mathrm{row}(i)\big|_{\gY_1 \cup \gZ}=0$ if $i\not\in \gY_1 \cup \gZ$, and for vertex $i\in \gY_1 \cup \gZ$, and let
\begin{align*}
    \mathrm{row}(i)\Big|_{\gY_1 \cup \gZ} \coloneqq \sum_{j=1}^{N} \sum_{\ell \in \sL}\,\,\sum_{ \substack{e\in \gE_{\ell}[\gY_1 \cup \gZ]\\ \{i, j\}\subset e} } \,\, \etA_e^{(\ell)} = \sum_{\ell \in \sL}(\ell - 1)\sum_{ \substack{e\in \gE_{\ell}[\gY_1 \cup \gZ]\\ \{i, j\}\subset e} } \,\, \etA_e^{(\ell)}\,.
\end{align*}
As a result, the matrix $(\rmA_{1})_{\sI_1}$ is obtained by restricting $\rmA_{1}$ on index set $\sI_1$. The next $4$ steps guarantee that \Cref{alg:multi_block_spectral_partition} outputs a $\nu$-correct partition.
\begin{enumerate}[label=(\roman*)]
    \item Find the singular subspace $\rmU$ spanned by the first $K$ left singular vectors of $(\rmA_{1})_{\sI_1}$.
    \item Randomly pick $s = 2K\log^2(N)$ vertices from $\gY_2$ and denote the corresponding columns in $\rmA_{2}$ by $\rva_{i_1}$,$\ldots$, $\rva_{i_s}$. Project each vector $\rva_i - \overline{\rva}$ onto the singular subspace $\rmU$, with $\overline{\rva}\in \R^{N}$ defined by $\overline{\rva}(j) = \indi{j\in Z} \cdot (\overline{\alpha} + \overline{\beta})/2$, where $\overline{\alpha}$, $\overline{\beta}$ were defined in  \eqref{eqn:bar_alphabeta}.
    \item For each projected vector $\rmP_{\rmU}(\rva_i - \overline{\rva})$, identify the top $N/(2K)$ coordinates in value and place the corresponding vertices into a set $\widehat{\gU}^{(0)}_i$. Discard half of the obtained $s$ subsets, those with the lowest \emph{blue} edge densities. 
    \item Sort the remaining sets according to blue hyperedge density and identify $K$ distinct subsets $\widehat{\gU}^{(0)}_{1}$, $\ldots$, $\widehat{\gU}^{(0)}_{K}$ such that $|\widehat{\gU}^{(0)}_i \cap \widehat{\gU}^{(0)}_j| < \lceil(1 - \nu)N/K \rceil$ if $i\neq j$.
\end{enumerate}
Based on the $4$ steps above in \Cref{alg:multi_block_spectral_partition}, the proof of \Cref{lem:spectral_clustering_accuracy_multiple} is structured in $4$ parts.
\begin{enumerate}[label=(\roman*)]
    \item Let $\widetilde{\rmU}$ denote the subspace spanned by first $K$ left singular vectors of $\widetilde{\rmA}_1$ defined in \eqref{eqn:tilde_A1A2}. \Cref{subsubsec:bound_subspace_angle_multiple} shows that the subspace angle between $\rmU$ and $\widetilde{\rmU}$  is smaller than any $c\in (0, 1)$ as long as $a_{\ell}, b_{\ell}$ satisfy certain conditions depending on $c$.
    
    \item The vector $\widetilde{\rvdelta}_i$, defined in \eqref{eqn:tilde_delta}, reflects the underlying true partition $\gZ \cap \gV_{\pi(i)}$ for each $i\in [\gS]$, where $\pi(i)$ denotes the membership of vertex $i$. \Cref{subsubsec:bound_the_projection_error} shows that $\overline{\rvdelta}_i$, an approximation of $\widetilde{\rvdelta}_i$ defined in \eqref{eqn:bar_delta}, can be recovered by the projected vector $\rmP_{\rmU}(\rva_{i} - \overline{\rva})$, since projection error  $\|\rmP_{\rmU}(\rva_{i} - \overline{\rva}) - \overline{\rvdelta}_{i}\|_2 < c\|\overline{\rvdelta}_{i}\|_2$ for any $c\in (0, 1)$ if $a_{\ell}, b_{\ell}$ satisfy the desired property in part (i).
    
    \item \Cref{subsubsec:discard_half_blue_accuracy} indicates that the coincidence ratio between the remaining sets and the true partition is at least $\nu$, after discarding half of the sets with the lowest blue edge densities.
    
    \item \Cref{lem:sample_k_distinct} proves that we can find $K$ distinct subsets $\widehat{\gU}^{(0)}_i$ within $K\log^2(N)$ trials with high probability.
\end{enumerate}

\subsubsection{Bounding the angle between $\rmU$ and $\widetilde{\rmU}$}\label{subsubsec:bound_subspace_angle_multiple} The angle between subspaces $\rmU$ and $\widetilde{\rmU}$ is defined as \[\sin \angle (\rmU, \widetilde{\rmU}) \coloneqq \|\rmP_{\rmU} - \rmP_{\widetilde{\rmU}}\|.\] A natural idea is to apply Wedin's $\sin \Theta$ Theorem  (\Cref{lem:Wedin_sin}). \Cref{lem:singular_value_approximation} indicates that the difference  of $\sigma_{j}(\widetilde{\rmA}_1)$ and $\sigma_{j}(\overline{\rmA}_{1})$ is relatively small, compared to $\sigma_{j}(\overline{\rmA}_{1})$. 

\begin{lemma}\label{lem:singular_value_approximation}
Let $\sigma_{j}(\overline{\rmA}_{1})$(resp.  $\sigma_{j}(\widetilde{\rmA}_1)$) denote the singular values of $\overline{\rmA}_{1}$ (resp.  $\widetilde{\rmA}_1$) for all $j \in [K]$, where the matrices $\overline{\rmA}_{1}$ and $\widetilde{\rmA}_1$ are defined in \eqref{eqn:bar_A1A2} and \eqref{eqn:tilde_A1A2} respectively. Then
\begin{align*}
        \sigma_{1}(\overline{\rmA}_{1}) =&\,\frac{N\left[ \overline{\alpha} + (K-1)\overline{\beta} \right] }{2\sqrt{2}K}= \frac{N}{2\sqrt{2}K}\sum_{\ell \in \sL} \left[ \binom{\frac{3N}{4K} - 2}{\ell - 2} \frac{a_{\ell} - b_{\ell}}{ \binom{N - 1}{\ell-1}} + K\binom{\frac{3N}{4} - 2}{\ell - 2}\frac{b_{\ell}}{\binom{N - 1}{\ell-1}} \right] \,,\\
        \sigma_{j}(\overline{\rmA}_{1}) =&\, \frac{N( \overline{\alpha} - \overline{\beta})}{2\sqrt{2}K} = \frac{N}{2\sqrt{2}K}\sum_{\ell \in \sL} \binom{\frac{3N}{4K} - 2}{\ell - 2} \frac{a_{\ell} - b_{\ell}}{ \binom{N - 1}{\ell-1}}\,,\quad \quad \quad \quad 2 \leq j \leq K\,,\\
        \sigma_{j}(\overline{\rmA}_{1}) =&\, 0\,,\quad \quad \quad \quad \quad \quad \quad \quad \quad \quad \quad \quad \quad \quad \quad \quad \quad \quad \quad \quad \quad K+1 \leq j \leq N\,.&
\end{align*}
with $\overline{\alpha}$, $\overline{\beta}$ defined in \eqref{eqn:bar_alphabeta}. Moreover, with probability at least $1 - 2K\exp(-K\log^2(N))$, for each $j \in [K]$,
\begin{align*}
    \frac{|\sigma_{j}(\overline{\rmA}_{1}) - \sigma_{j}(\widetilde{\rmA}_1)|}{\sigma_{j}(\overline{\rmA}_{1})} = O\left(N^{-\frac{1}{4}}\log^{\frac{1}{2} }(N)\right).
\end{align*}
\end{lemma}
Therefore, with \Cref{lem:singular_value_approximation}, we can write
$
    \sigma_{j}(\widetilde{\rmA}_1) = \sigma_{j}(\overline{\rmA}_{1})(1 + o(1)).
$
Define $\rmE_{1} \coloneqq \rmA_{1} - \widetilde{\rmA}_1$ and its restriction on $\sI_{1}$ \eqref{eqn:index_sets_I1} as 
\begin{equation}\label{eqn:E1I_k}
    (\rmE_{1})_{\sI_{1}} \coloneqq (\rmA_{1} - \widetilde{\rmA}_1)_{\sI_{1}} = (\rmA_{1})_{\sI_{1}} - (\widetilde{\rmA}_1)_{\sI_{1}}\,,
\end{equation}
as well as $\rvdelta_{1}\coloneqq (\widetilde{\rmA}_1)_{\sI_{1}} - \widetilde{\rmA}_1$. Then $(\rmA_{1})_{\sI_{1}} - \widetilde{\rmA}_1$ is decomposed as
\begin{align*}
    (\rmA_{1})_{\sI_{1}} - \widetilde{\rmA}_1 = [(\rmA_{1})_{\sI_{1}} - (\widetilde{\rmA}_1)_{\sI_{1}}] + [(\widetilde{\rmA}_1)_{\sI_{1}} -  \widetilde{\rmA}_1 ] = (\rmE_{1})_{\sI_{1}} + \rvdelta_{1}\,.
\end{align*}

\begin{lemma}\label{lem:high_degree_vertices_k}
    Let $d = \sum_{\ell \in \sL} (\ell - 1)a_{\ell}$, where $\sL$ is obtained from \Cref{alg:parameter_preprocessing}. There exists a constant $\const_{1}\geq (2^{1/\LM } - 1)^{-1/3}$ such that if $d \geq \const_{1}$, then with probability at least $1 - \exp\left(- d^{-2}N/\LM  \right)$, no more than $d^{-3}N$ vertices have row sums greater than $20\LM d$.
\end{lemma} 
\Cref{lem:high_degree_vertices_k} shows that the number of high-degree vertices is relatively small. Consequently, \Cref{cor:norm_of_high_degree_vertices_multi_block} indicates $\|\rvdelta_1\| \leq \sqrt{d}$ with high probability.
\begin{corollary}\label{cor:norm_of_high_degree_vertices_multi_block}
Assume $d \geq \max\{\const_{1}, \sqrt{2}\}$, where $\const_{1}$ is the constant in \Cref{lem:high_degree_vertices_k}, then $\|\rvdelta_1\| \leq \sqrt{d}$ with probability at least $1 - \exp\left(- d^{-2}N/\LM  \right)$.
    \end{corollary}
\begin{proof}[Proof of \Cref{cor:norm_of_high_degree_vertices_multi_block}]
   Note that $N- |\sI| \leq d^{-3}N$ and $\sI \subset \sI_{1}$, then $N - |\sI_{1}| \leq d^{-3}N$.
    From \Cref{lem:high_degree_vertices_k}, there are at most $d^{-3} N$ vertices with row sum greater than $20\LM d$ in the adjacency matrix $\rmA_{1}$, then the matrix $\rvdelta_1 = (\widetilde{\rmA}_1)_{\sI_{1}} - \widetilde{\rmA}_1$ has at most $2d^{-3} N^{2}$ non-zero entries. Every entry of $\widetilde{\rmA}_1$ in \eqref{eqn:tilde_A1A2} is bounded by $\alpha$, then,
    \begin{subequations}
        \begin{align*}
           \|\rvdelta_1\| \leq &\,  \|\rvdelta_1 \|_{\frob} = \|(\widetilde{\rmA}_1)_{\sI_{1}} - \widetilde{\rmA}_1\|_{\frob} \\
           \leq &\, \sqrt{2 d^{-3}N^{2}} \, \alpha = N\sqrt{2 d^{-3}} \sum_{\ell \in \sL} \left[ \binom{ \frac{N}{K}-2}{\ell - 2} \frac{a_{\ell} - b_{\ell}}{\binom{N - 1}{\ell-1}}  + \binom{N - 2}{\ell - 2} \frac{b_{\ell}}{\binom{N - 1}{\ell-1}} \right]\\
           \leq &\, \sqrt{2d^{-3} } \sum_{\ell \in \sL} (\ell - 1)a_{\ell} \leq \sqrt{2d^{-1}} \leq \sqrt{d}\,.
        \end{align*}
    \end{subequations}
\end{proof}
Moreover, by taking $\tau=20 \LM$, $\theta =3$ in \Cref{thm:regularization_concentration_partial}, with probability at least $1 - N^{-2}$
\begin{align}
    \|(\rmE_{1})_{\sI_{1}}\|\leq \const_{3} \sqrt{d}\,,\label{eq:E1I1}
\end{align}
where constant $\const_{3} >0$ depends on $\LM$. Together with upper bounds for $\|(\rmE_{1})_{\sI_{1}}\|$ and $\|\rvdelta_1\|$, \Cref{lem:subspace_angle_multi} shows that the angle between $\rmU$ and $\widetilde{\rmU}$ is relatively small with high probability.

\begin{lemma}\label{lem:subspace_angle_multi}
    For any $c\in (0,1)$, there exists some constant $\const_{2} >0$ such that, if
  \[
    \sum_{\ell \in \sL}(\ell - 1)(a_{\ell}-b_{\ell})\geq \const_{2} K^{\LM  - 1}\sqrt{d}\,,
  \]
then $\sin \angle (\rmU, \widetilde{\rmU}) \leq c$ with probability $1 - N^{-2}$. Here $\angle(\rmU, \widetilde{\rmU})$ is the angle between $\rmU$ and $\widetilde{\rmU}$. 
\end{lemma}

\begin{proof}[Proof of \Cref{lem:subspace_angle_multi}] 
From \eqref{eq:E1I1} and \Cref{cor:norm_of_high_degree_vertices_multi_block}, with probability at least $1 - N^{-2}$,
\begin{align}
     \|(\rmA_{1})_{\sI_{1}} - \widetilde{\rmA}_1\| \leq  \|(\rmE_{1})_{\sI_{1}}\| + \|\rvdelta_1 \|\leq (\const_{3} + 1)\sqrt{d}. \notag 
\end{align}
Since $\sigma_{K+1}(\widetilde{\rmA}_1)=0$, using \Cref{lem:singular_value_approximation} to approximate $\sigma_{K}(\widetilde{\rmA}_1)$, we obtain
\begin{align*}
    &\,\sigma_{K}(\widetilde{\rmA}_1) - \sigma_{K+1}(\widetilde{\rmA}_1) =\sigma_{K}(\widetilde{\rmA}_1)=\big( 1 + o(1) \big)\cdot \sigma_{K}(\overline{\rmA}_{1}) \geq \frac{1}{2}\sigma_{K}(\overline{\rmA}_{1})\\
    \geq &\, \frac{N}{4\sqrt{2}K}\sum_{\ell \in \sL} { \frac{3N}{4K} -2 \choose \ell - 2} \frac{a_{\ell} - b_{\ell}}{\binom{N - 1}{\ell -1}} \geq \frac{1}{8K}\sum_{\ell \in \sL} \bigg(\frac{3}{4K} \bigg)^{\ell - 2}(\ell - 1)(a_{\ell}-b_{\ell})\\
    \geq &\, \frac{1}{8K}\left(\frac{1}{2K}\right)^{\LM  - 2}\sum_{\ell \in \sL} (\ell - 1)(a_{\ell}-b_{\ell})\geq \frac{\const_{2} \sqrt{d}}{2^{\LM  + 1}}\,.
\end{align*}
Then for any $c\in(0, 1)$, we can find $\const_{2} = [2^{\LM  + 2}(\const_{3} + 1)/c]$ such that  $\|(\rmA_{1})_{\sI_{1}} - \widetilde{\rmA}_1\|\leq (1-1/\sqrt{2}) \sigma_{K}(\widetilde{\rmA}_1) 
$. By Wedin's Theorem (\Cref{lem:Wedin_sin}), the angle $\angle (\rmU, \widetilde{\rmU})$ is bounded by
    \begin{align}
        \sin \angle (\rmU, \widetilde{\rmU}) \coloneqq \|\rmP_{\rmU} - \rmP_{\widetilde{\rmU}}\| \leq &\, \frac{\sqrt{2}\|(\rmA_{1})_{\sI_{1}} - \widetilde{\rmA}_1\|}{ \sigma_{K}(\widetilde{\rmA}_1)} \leq \frac{\sqrt{2}(\const_{3}+1)\sqrt{d}}{\const_{2}\sqrt{d}/ 2^{\LM  + 1}}
        = \frac{\sqrt{2}}{2}c< c\,. \notag 
    \end{align}
\end{proof}

\subsubsection{Bound the projection error}\label{subsubsec:bound_the_projection_error}

Randomly pick $s=2K\log^2(N)$ vertices from $\gY_2$. 
Let $\rva_{i_1},\dots, \rva_{i_s}$, $\widetilde{\rva}_{i_1}$,$\dots$, $\widetilde{\rva}_{i_s}$, $\overline{\rva}_{i_1},\dots, \overline{\rva}_{i_s}$ and $\rve_{i_1}, \dots, \rve_{i_s}$ be the corresponding columns of $\rmA_{2}$, $\widetilde{\rmA}_2$, $\overline{\rmA}_{2}$ and $\rmE_2 \coloneqq \rmA_{2} - \widetilde{\rmA}_2$ respectively, where $\rmA_{2}$, $\widetilde{\rmA}_2$ and $\overline{\rmA}_{2}$ were defined in \eqref{eqn:A1A2}, \eqref{eqn:tilde_A1A2} and \eqref{eqn:bar_A1A2}. Let $\pi(i)$ denote the membership of vertex $i$. Note that entries of vector $\widetilde{\rva}_{i}$ are $\widetilde{\alpha}_{ii}$, $\widetilde{\beta}_{ij}$ or $0$, according to the membership of vertices in $\gZ$, where $\widetilde{\alpha}_{ii}$, $\widetilde{\beta}_{ij}$ were defined in \eqref{eqn:tilde_alpha}, \eqref{eqn:tilde_beta}. Then the corresponding vector $\widetilde{\rvdelta}_i \in \R^{N}$ with the entries given by
\begin{equation}\label{eqn:tilde_delta}
    \widetilde{\rva}_{i}(j) =
        \begin{cases}
            \widetilde{\alpha}_{ii}, & \textnormal{if } j\in \gZ \cap \gV_{\pi(i)}\\
            \widetilde{\beta}_{ij}, & \textnormal{if } j\in \gZ \setminus  \gV_{\pi(i)} \\
            0\,, & \textnormal{if } j\in \gY
        \end{cases}\,,\quad \widetilde{\rvdelta}_i(j) =
        \begin{cases}
            (\widetilde{\alpha}_{ii} - \widetilde{\beta}_{ij})/2 >0, & \textnormal{if } j\in \gZ \cap \gV_{\pi(i)} \\
            (\widetilde{\beta}_{ij} - \widetilde{\alpha}_{ii})/2 <0, & \textnormal{if } j\in \gZ \setminus \gV_{\pi(i)} \\
            0, & \textnormal{if } j\in \gY
        \end{cases},
\end{equation}
can be used to recover the vertex set $\gZ \cap \gV_{\pi(i)}$ based on the sign of elements in $\widetilde{\rvdelta}_i$. However, it is hard to handle with $\widetilde{\rvdelta}_i$ due to the randomness of $\widetilde{\alpha}_{ii}$, $\widetilde{\beta}_{ij}$ originated from $N_i$ and $\widetilde{N}_{j}$. Note that $N_i$ and $\widetilde{N}_{j}$ concentrate around $N/(2K)$ and $N/(4K)$ respectively as shown in \Cref{lem:singular_value_approximation}. Thus a good approximation of $\widetilde{\rvdelta}_i$, which rules out randomness of $N_i$ and $\widetilde{N}_{j}$, was given by $\overline{\rvdelta}_i \coloneqq \overline{\rva}_{i} - \overline{\rva}$, with entries given by $\overline{\rva}(j)\coloneqq \indi{j\in Z}\cdot (\overline{\alpha} + \overline{\beta})/2$, where $\overline{\alpha}$ and $\overline{\beta}$ were defined in \eqref{eqn:bar_alphabeta}, and
\begin{equation}\label{eqn:bar_delta}
\overline{\rva}_i (j) =
        \begin{cases}
            \overline{\alpha},  \quad &\textnormal{if } j\in \gZ \cap \gV_{\pi(i)} \\
            \overline{\beta}, \quad &\textnormal{if } j\in \gZ \setminus \gV_{\pi(i)} \\
            0, \quad &\textnormal{if } j\in \gY
        \end{cases} \,,\quad \overline{\rvdelta}_i(j) =
        \begin{cases}
            (\overline{\alpha} - \overline{\beta})/2 >0, \quad & \textnormal{if } j\in \gZ \cap \gV_{\pi(i)} \\
            (\overline{\beta} - \overline{\alpha})/2 <0, \quad &\textnormal{if } j\in \gZ \setminus \gV_{\pi(i)} \\
            0,  \quad &\textnormal{if } j\in \gY
        \end{cases} \,. 
\end{equation}
By construction, $\overline{\rvdelta}_i$ identifies vertex set $Z \cap \gV_{\pi(i)}$ in the case of \textit{perfect splitting} for any $i\in\{i_1, \cdots, i_s\}\cap \gY_2\cap \gV_{\pi(i)}$. However, the access to $\overline{\rvdelta}_i$ is limited in practice, thus the projection $\rmP_{\rmU}(\rva_i - \overline{\rva})$ is used instead as an approximation of $\overline{\rvdelta}_i$. 
\Cref{lem:projection_error_multi} proves that at least half of the projected vectors have small projection errors.

\begin{lemma}
\label{lem:projection_error_multi}
For any $c\in (0, 1)$, there exist constants $\const_{1}$ and $\const_{2}$ such that if $d> \const_{1}$ and \[\sum_{\ell \in \sL}(\ell - 1)(a_{\ell} -b_{\ell}) > \const_{2} K^{\LM }\sqrt{d},\] then among all projected vectors $\rmP_{\rmU}(\rva_{i} - \overline{\rva})$ for $i\in\{i_1, \cdots, i_s\}\cap \gY_2$, with probability $1 - O(N^{-K})$, at least half of them satisfy 
\begin{align}\label{eqn:good_vector}
    \|\rmP_{\rmU}(\rva_{i} - \overline{\rva}) - \overline{\rvdelta}_{i}\|_2 < c \, \|\overline{\rvdelta}_{i}\|_2.
\end{align}
\end{lemma}
\begin{proof}[Proof \Cref{lem:projection_error_multi}]
Note that $\overline{\rvdelta}_i = \rmP_{\overline{\rmU}}\overline{\rvdelta}_i$, where $\overline{\rmU}$ is spanned by the first $K$ left singular vectors of $\overline{\rmA}_{1}$ with $\textnormal{rank}(\overline{\rmA}_{1}) = K$, and $\overline{\rmA}_{1}$, $\overline{\rmA}_{2}$ preserve the same eigen-subspace. 
The approximation error between $ \rmP_{\rmU}(\rva_i - \overline{\rva})$ and $\overline{\rvdelta}_i$ can be decomposed as
\begin{subequations}
	\begin{align*}
    \rmP_{\rmU}(\rva_i - \overline{\rva}) -  \overline{\rvdelta}_i  =\,& \rmP_{\rmU}\big[ (\rva_i - \widetilde{\rva}_i) + (\widetilde{\rva}_i - \overline{\rva}_{i}) + (\overline{\rva}_i - \overline{\rva}) \big] - \rmP_{\overline{\rmU}}\overline{\rvdelta}_i\\
    =\,& \rmP_{\rmU}\rve_{i} + \rmP_{\rmU}(\widetilde{\rva}_i - \overline{\rva}_{i}) + (\rmP_{\rmU} - \rmP_{\overline{\rmU}})\overline{\rvdelta}_i~.
	\end{align*}
\end{subequations}
Then by triangle inequality, 
\begin{align*}
   \| \rmP_{\rmU}(\rva_i - \overline{\rva}) -  \overline{\rvdelta}_i\|_2 \leq \|\rmP_{\rmU}\rve_{i}\|_2 + \|\rmP_{\rmU}(\widetilde{\rva}_i - \overline{\rva}_{i})\|_2 + \|\rmP_{\rmU} - \rmP_{\overline{\rmU}}\|\cdot \|\overline{\rvdelta}_i\|_{2}~.
\end{align*}
Note that $\|\overline{\rvdelta}_i\| = O(N^{-\frac{1}{2} })$ and $\widetilde{N}_{i}$ concentrates around $N/(4K)$ for each $i \in [K]$ with deviation at most $\sqrt{N}\log(N)$, then by definitions of $\overline{\alpha}$ and $\overline{\beta}$ in \eqref{eqn:bar_alphabeta},
\begin{align*}
 \|\rmP_{\rmU}(\widetilde{\rva}_i - \overline{\rva}_{i})\|_2 \leq \|\widetilde{\rva}_i - \overline{\rva}_{i}\|_2 = O\Big(\big[ K\sqrt{N}\log(N)(\overline{\alpha} -\overline{\beta})^2 \big]^{\frac{1}{2}} \Big) = O[ N^{-\frac{3}{4}}\log^{\frac{1}{2}}(N)]=o(\|\overline{\rvdelta}_i\|_2).
\end{align*}
Meanwhile, by an argument similar to \Cref{lem:subspace_angle_multi}, it can be proved that $\sin \angle (\rmU, \overline{\rmU}) < c/2$ for any $c\in (0, 1)$, if constants $\const_{1}, \const_{2}$ are chosen properly, hence $\|\rmP_{\rmU} - \rmP_{\overline{\rmU}}\|\cdot \|\overline{\rvdelta}_i\|_{2}< \frac{c}{2} \|\overline{\rvdelta}_i\|_{2}$. \Cref{lem:projection_deviation_multiple} shows that at least half of the indices from $\{i_1, \cdots, i_s\}\cap \gY_2$ satisfy $\|\rmP_{\rmU}\rve_i \|_2 < \frac{c}{2}\|\overline{\rvdelta}_{i}\|_2$, which completes the proof.
\end{proof}
\begin{lemma}\label{lem:projection_deviation_multiple}
Let $d = \sum_{\ell \in \LM } (\ell - 1)a_{\ell}$. For any $c\in(0, 1)$, with probability $1 - O(N^{-K\log(N)})$, at least $\frac{s}{2}$ of the vectors $\rve_{i_1}, \dots, \rve_{i_s}$ satisfy
    \begin{align*}
       \|\rmP_{\rmU}\rve_i \|_2 \leq 2\sqrt{dK(\LM  + 2)/N} < \frac{c}{2}\|\overline{\rvdelta}_{i}\|_2 \,, \quad i\in \{ i_1, \cdots, i_s \} \subset \gY_2\,.
    \end{align*}
\end{lemma}

\begin{definition}
The vector $\rva_i$ satisfying \eqref{eqn:good_vector} is referred as \textit{good} vector. The index of the good vector is hence referred to as a good vertex.
\end{definition}
To avoid introducing extra notations, let $\rva_{i_1}, \dots, \rva_{i_{s_1}}$ denote \textit{good} vectors with $i_1, \cdots, i_{s_1}$ denoting good indices. \Cref{lem:projection_deviation_multiple} indicates that the number of good vectors is at least $s_1\geq \frac{s}{2} = K\log^2(N)$.  

\subsubsection{Accuracy}\label{subsubsec:discard_half_blue_accuracy} 
We are going to prove the accuracy of the initial partition obtained from \Cref{alg:multi_block_spectral_partition}. Lemmas \ref{lem:angle_accuracy_multiple}, \ref{lem:discard_half_blue} and \ref{lem:sample_k_distinct} are crucial in proving our results. We present the proof logic first and defer the Lemma statements later. 

For each projected vector $\rmP_{\rmU}(\rva_{i} - \overline{\rva})$, let $\widehat{\gU}^{(0)}_{i}$ denote the set of its largest $\frac{N}{2K}$ coordinates, where $i\in \{i_1, \cdots, i_s\}$ and $s = 2K \log^2(N)$. Note that vector $\overline{\rvdelta}_{i_j}$ in \eqref{eqn:tilde_delta} only identifies blocks $\gV_{\pi(i_j)}$ and $\gV\setminus \gV_{\pi(i_j)}$, which can be regarded as clustering two blocks with different sizes. By \Cref{lem:projection_error_multi}, good vectors satisfy $\|\rmP_{\rmU}(\rva_{i_j} - \overline{\rva}) - \overline{\rvdelta}_{i_j}\|_2 < c \, \|\overline{\rvdelta}_{i_j}\|_2$ for any $c\in (0, 1)$. Then by \Cref{lem:angle_accuracy_multiple} (after proper normalization), for a good index $i_j$, the number of vertices in $\widehat{\gU}^{(0)}_{i_j}$ clustered correctly is at least $(1 - \frac{4K}{3}c^{2})\frac{N}{K}$. By choosing $c = \sqrt{3(1-\nu)/(8K)}$, the condition $|\widehat{\gU}^{(0)}_{i_j} \cap \gV_i| > (1+\nu)/2 |\widehat{\gU}^{(0)}_{i_j}|$ in part (ii) of \Cref{lem:discard_half_blue} is satisfied. In \Cref{lem:subspace_angle_multi}, we choose 
\begin{align}
    \const_{2} = 2^{\LM  + 2}(\const_{3} + 1)/c = 2^{\LM  + 2}(\const_{3} + 1)\sqrt{8K/(3-3\nu)}\,\,, \label{eqn:C2}
\end{align}
where $\const_{3}$ defined in \eqref{eq:E1I1}. Hence, with high probability, all good vectors have at least $\mu_{\mathrm{T}}$ blue hyperedges (we call this \textit{``high blue hyperedge density"}). From \Cref{lem:projection_deviation_multiple}, at least half of the selected vectors are good. Then, in \Cref{alg:multi_block_spectral_partition}, throwing out half of the obtained sets $\widehat{\gU}^{(0)}_{i}$ (those with the lowest blue hyperedge density) guarantees that the remaining sets are good. 

Recall that, by choosing constant appropriately, we can make the subspace angle $\sin \angle (\rmU, \overline{\rmU}) < c$ for any $c\in (0, 1)$ ($\overline{\rmU}$ is spanned by the first $K$ left singular vectors of $\overline{\rmA}_{1}$). Then for vectors $\overline{\rvdelta}_{i_1}, \cdots, \overline{\rvdelta}_{i_{K}}$ with each $i_j$ selected from different vertex set $\gV_j$, there is a vector $\rmP_{\rmU}(\rva_{i_j} - \overline{\rva})$ in $\rmU$ arbitrarily close to $\overline{\rvdelta}_{i_j}$, which was proved by \Cref{lem:projection_error_multi}. From (i) of \Cref{lem:discard_half_blue}, so obtained $\widehat{\gU}_{i_j}^{(0)}$ must satisfy $|\widehat{\gU}_{i_j}^{(0)}\cap \gV_j|\geq \nu|\widehat{\gU}_{i_j}^{(0)}|$ for each $j\in [K]$. The remaining thing is to select $K$ different $\widehat{\gU}_{i_j}^{(0)}$ with each of them concentrating around distinct $\gV_j$ for each $j\in[K]$. This problem is equivalent to finding $K$ vertices in $\gY_2$, each from a different partition class, which can be done with $K\log^2(N)$ samplings as shown in \Cref{lem:sample_k_distinct}.


To summarize,  this section is a more precise and quantitative version of  the following argument:  with high probability, the following holds for the good vectors $\rva_{i_j}$, $j\in[K]$: 
\begin{align*}
    \left\{ i_j : \exists i~\mbox{s.t.}~|\widehat{\gU}_{i_j}^{(0)} \cap \gV_i| \geq \frac{1+\nu}{2} \frac{N}{K}\right\} &\, \subset \left\{ i_j : \widehat{\gU}_{i_j}^{(0)} \text{ has more than $\mu_T$ blue hyperedges}\right\} \\
    &\, \subset \left\{ i_j:  \exists i~\mbox{s.t.}~|\widehat{\gU}_{i_j}^{(0)} \cap \gV_i| \geq \nu \frac{N}{K}\right\} ~.
\end{align*}

\begin{lemma}[Adapted from Lemma $23$ in \cite{Chin2015StochasticBM}]\label{lem:angle_accuracy_multiple}
    Assume $N/K \in \N$. Let $\rvv, \overline{\rvv} \in \R^{N}$  be two unit vectors, where $\overline{\rvv}$ has $\frac{N}{K}$ of its entries being $\frac{1}{\sqrt{N}}$ with the rest being $-\frac{1}{\sqrt{N}}$. If $\sin \angle(\overline{\rvv}, \rvv) < c \leq 0.5$, then $\rvv$ contains at least $(1 - \frac{4K}{3}c^{2}) \frac{N}{K}$ positive entries $\ervv_i$ such that the corresponding entries $\overline{\ervv}_{i}$ are positive as well.
\end{lemma}

\begin{lemma}\label{lem:discard_half_blue} 
    For the set $\gX \subset \gZ$ with size $|\gX| = N/(2K)$. Define
\begin{subequations}
\begin{align*}
    \mu_1\coloneqq&\,\frac{1}{2} \sum_{\ell \in \sL}\ell(\ell - 1) \left\{ \left[ \binom{\frac{\nu N}{2K}}{\ell} + \binom{\frac{(1 - \nu)N}{2K}}{\ell} \right] \frac{a_{\ell} - b_{\ell} }{ \binom{N - 1}{\ell-1} } + \binom{\frac{N}{2K}}{\ell} \frac{b_{\ell}}{ \binom{N - 1}{\ell-1} } \right\}\,,\\
    \mu_{2} \coloneqq&\, \frac{1}{2} \sum_{\ell \in \sL} \ell(\ell - 1)\left\{ \left[ \binom{\frac{(1 + \nu) N}{4K} }{\ell} + (K-1)\binom{ \frac{(1 - \nu)N}{4K(K-1)}}{\ell} \right] \frac{a_{\ell} - b_{\ell} }{ \binom{N - 1}{\ell-1} } + \binom{\frac{N}{2K}}{\ell} \frac{b_{\ell}}{ \binom{N - 1}{\ell-1} } \right\}\,,
\end{align*}
\end{subequations}
and let $\mu_{\mathrm{T}} \coloneqq (\mu_1 + \mu_2)/2 \in [\mu_1, \mu_2]$. There exists some constant $c>0$ depending on $K$, $\{a_{\ell}\}_{\ell\in \sL}$ and $\nu$ such that for sufficiently large $N$,
    \begin{enumerate}[label=(\roman*)]
        \item  
        If $|\gX \cap \gV_j| \leq \nu |\gX|$ for each $j\in[K]$, then with probability $1 - e^{-cN}$, the number of blue hyperedges in the hypergraph induced by $\gX$ is at most $\mu_{\mathrm{T}}$.

        \item Conversely, if $|\gX \cap \gV_j| \geq \frac{1+ \nu}{2}|\gX|$ for some $j\in [K]$, then with probability $1 - e^{-cN}$, the number of blue hyperedges in the hypergraph induced by $\gX$ is at least $\mu_{\mathrm{T}}$.
    \end{enumerate}
\end{lemma}
\begin{remark}
\Cref{lem:discard_half_blue} is reduced to \cite[Lemma 31]{Chin2015StochasticBM} when $\sL = \{2\}$. 
\end{remark}


Let $\widehat{\gU}_{i}^{(0)}$ denote the set of $N/(2K)$ largest coordinates of the projected vector $\rmP_{\rmU}(\rva_i - \overline{\rva})$ with $i\in \gV_{\pi(i)}\cap \gY_2$. According to \Cref{lem:discard_half_blue}, when the blue hyperedge density inside $\widehat{\gU}_{i}^{(0)}$ is at least $\mu_{\mathrm{T}}$, it guarantees that $|\widehat{\gU}_{i}^{(0)}\cap \gV_{\pi(i)}| \geq \nu |\widehat{\gU}_{i}^{(0)}|$. Therefore, it is enough to consider half of the sets with highest blue edge densities in \Cref{alg:multi_block_spectral_partition}.
\begin{lemma}\label{lem:sample_k_distinct}
    Through random sampling without replacement in Step $6$ of \Cref{alg:multi_block_spectral_partition}, with probability $1 -N^{-\Omega(\log(N))}$, there exists at least $K$ indices $i_1,\dots, i_{K}$ in $\gY_2$ among the $K\log^2(N)$ samples such that
    \[  |\widehat{\gU}_{i_j}^{(0)} \cap \widehat{\gU}_{i_k}^{(0)}| \leq (1-\nu)N/K, ~\text{ for any } j,k\in [K] \text{ with }  j\neq k.
     \]
\end{lemma}

\subsection{Local Correction: Proof of \Cref{lem:correction_accuracy_multiple}}\label{subsec:local_correction}
For notation convenience, let $\gU_i\coloneqq \gZ \cap \gV_i$ denote the intersection of $\gZ$ and true partition $\gV_i$ for all $i\in[K]$. In \Cref{alg:multiple_partition_partial}, we first color the hyperedges with red and blue with equal probability. By running \Cref{alg:multi_block_spectral_partition} on the red hypergraph, we obtain a $\nu$-correct partition $\widehat{\gU}_{1}^{(0)}, \dots, \widehat{\gU}_{K}^{(0)}$, i.e.,
\begin{equation}\label{eqn:spectral_clustering_accuracy_k}
    |\widehat{\gU}_{k}\setminus \widehat{\gU}_{k}^{(0)}| \leq (1 - \nu)\cdot |\widehat{\gU}_{k}^{(0)}| = (1 - \nu)\cdot \frac{N}{2K}\,, \quad \forall k \in [K]\,.
\end{equation}
In the rest of this subsection, we condition on the event that \eqref{eqn:spectral_clustering_accuracy_k} holds true.

Consider a hyperedge $e = \{i_{1}, \cdots, i_{\ell}\}$ in the underlying $\ell$-uniform hypergraph. If vertices $i_{1}, \cdots, i_{\ell}$ are from the same block, then $e$ is a red hyperedge with probability $a_{\ell}/2\binom{N - 1}{\ell-1}$;  if vertices $i_{1}, \cdots, i_{\ell}$ are not from the same block, then $e$ is a red hyperedge with probability $b_{\ell}/2\binom{N - 1}{\ell-1}$. The presence of those two types of hyperedges can be denoted by 
\begin{equation*}
    \rT_{e}^{(a_{\ell})}\sim \mathrm{Bernoulli}\left( \frac{a_{\ell}}{ 2\binom{N - 1}{\ell -1} }\right)\,, \quad \rT_{e}^{(b_{\ell})}\sim \mathrm{Bernoulli}\left( \frac{b_{\ell}}{ 2\binom{N - 1}{\ell -1} }\right)\,,
\end{equation*}
respectively. 
For any finite set $\gS$, let $[\gS]^{l}$ denote the family of $l$-subsets of $\gS$, i.e., $[\gS]^{l} = \{\gZ| \gZ \subseteq \gS, |\gZ| = l \}$. 
Consider a vertex $u\in \gU_1: = \gZ \cap \gV_1$. The weighted number of red hyperedges, which contains $u\in \gU_1$ with the remaining vertices in $\widehat{\gU}_j^{(0)}$, can be written as
\begin{align}\label{eqn:weighted_neighbor_Z}
    \widehat{\rS}_{1j}^{(0)}(u)\coloneqq \sum_{\ell \in \sL}(\ell - 1)\cdot \left\{ \sum_{e\in \, \gE^{(a_{\ell})}_{1, j}} \rT_{e}^{(a_{\ell})} + \sum_{e\in \gE^{(b_{\ell})}_{1, j} } \rT_{e}^{(b_{\ell})}\right\}\,,\quad u\in \gU_1\,,
\end{align}
where $\gE^{(a_{\ell})}_{1, j}\coloneqq \gE_{\ell}([\gU_1]^{1}, [ \gU_1 \cap \widehat{\gU}_{j}^{(0)}]^{\ell - 1} )$ denotes the set of $\ell$-hyperedges with one vertex from $[\gU_1]^{1}$ and the other $\ell - 1$ from $[\gU_1 \cap \widehat{\gU}^{(0)}_{j}]^{\ell - 1}$, and 
\begin{align}
     \gE^{(b_{\ell})}_{1, j} \coloneqq \gE_{\ell} \Big([\gU_1]^{1}, \,\, [\widehat{\gU}_{j}^{(0)}]^{\ell - 1} \setminus [\gU_1\cap \widehat{\gU}_{j}^{(0)}]^{\ell - 1} \Big)
\end{align} 
denotes the set of $\ell$-hyperedges with one vertex in $[\gU_1]^{1}$ while the remaining $\ell - 1$ vertices in $[\widehat{\gU}_{j}^{(0)}]^{\ell - 1}\setminus [\gU_1 \cap \widehat{\gU}_{j}^{(0)}]^{\ell - 1}$(not all $\ell$ vertices are from $\gV_1$) with their cardinalities
    \begin{align}
       |\gE^{(a_{\ell})}_{1, j}| = \binom{|\gU_1 \cap \widehat{\gU}^{(0)}_{j}|}{\ell - 1}\,,\quad |\gE^{(b_{\ell})}_{1, j}| = \left[ \binom{|\widehat{\gU}^{(0)}_{j}|}{\ell - 1} - \binom{|\gU_1 \cap \widehat{\gU}^{(0)}_{j}|}{\ell - 1} \right]\,. \notag 
    \end{align}
We multiply $(\ell - 1)$ in \eqref{eqn:weighted_neighbor_Z} as weight since the rest $\ell - 1$ vertices are all located in $\widehat{\gU}^{(0)}_{j}$, which can be regarded as the neighbors of vertex $u$ in $\widehat{\gU}^{(0)}_{j}$. According to the facts that $|\widehat{\gU}^{(0)}_j \cap \gU_j| \geq \nu N/(2K)$ in \eqref{eqn:spectral_clustering_accuracy_k} and $|\widehat{\gU}_{j}^{(0)}| = N/(2K)$ for each $j\in[K]$, we have
\begin{align}
    |\gE^{(a_{\ell})}_{1, 1}| \geq \binom{\frac{\nu N}{2K}}{\ell - 1}\,,\quad |\gE^{(a_{\ell})}_{1, j}| \leq \binom{\frac{(1 - \nu)N}{2K}}{\ell - 1}\,,\,\, j\neq 1\,. \notag 
\end{align}
To simplify the calculation, we take the lower and upper bound of $|\gE^{(a_{\ell})}_{1, 1}|$ and $|\gE^{(a_{\ell})}_{1, j}|\,(j\neq 1)$ respectively.
By taking expectation with respect to $\rT_{e}^{(a_{\ell})}$ and $\rT_{e}^{(b_{\ell})}$, for any $u \in \gU_1$, the following holds:
\begin{subequations}
    \begin{align*}
    \E \widehat{\rS}_{11}^{(0)}(u) &= \sum_{\ell \in \sL} (\ell - 1)\cdot \left[ \binom{\frac{\nu N}{2K}}{\ell - 1} \frac{a_{\ell} - b_{\ell}}{2\binom{N - 1}{\ell-1}} + \binom{\frac{N}{2K}}{\ell - 1} \frac{b_{\ell}}{2\binom{N - 1}{\ell-1}}\right]\,,\\
    \E \widehat{\rS}_{1j}^{(0)}(u) &= \sum_{\ell \in \sL} (\ell - 1)\cdot \left[ \binom{\frac{(1-\nu)N}{2K}}{\ell - 1} \frac{a_{\ell} - b_{\ell}}{2\binom{N - 1}{\ell-1}} +  \binom{\frac{N}{2K}}{\ell - 1} \frac{b_{\ell}}{2\binom{N - 1}{\ell-1}}\right]\,, \quad j\neq 1\,.
\end{align*}
\end{subequations}
According to the assumptions in \Cref{thm:partial_multiple}, $\E \widehat{\rS}_{11}^{(0)}(u) - \E \widehat{\rS}_{1j}^{(0)}(u) = \Omega(1)$. Define
\begin{align}\label{eqn:mu_correction}
    \mu_{\rm{C}} \coloneqq \frac{1}{2}\sum_{\ell \in \sL} (\ell - 1)\cdot \left\{ \left[ \binom{\frac{\nu N}{2K}}{\ell - 1} + \binom{\frac{(1 - \nu)N}{2K}}{\ell - 1} \right] \frac{a_{\ell} - b_{\ell}}{2\binom{N - 1}{\ell-1}} + 2\cdot \binom{\frac{N}{2K}}{\ell - 1}\frac{b_{\ell}}{2\binom{N - 1}{\ell-1}}\right\}\,.
\end{align}
In \Cref{alg:multi_block_correction}, vertex $u$ is assigned to $\widehat{\gU}_{i}$ if it has the maximal number of  neighbors in $\widehat{\gU}_i^{(0)}$. If $u\in \gU_1$ is mislabeled, then one of the following events must happen:
\begin{itemize}
    \item $\widehat{\rS}_{11}^{(0)}(u) \leq \mu_{\rm{C}}$, meaning that $u$ fails to have enough neighbors in $\widehat{\gU}_1^{(0)}$ to be assigned to $\widehat{\gU}_{1}$. 
    \item $\widehat{\rS}_{1j}^{(0)}(u) \geq \mu_{\rm{C}}$ for some $j\neq 1$, meaning that $u$ survived \Cref{alg:multi_block_correction} without being corrected. 
\end{itemize}
\Cref{lem:mislabel_probability_correction} shows that the probabilities of those two events can be bounded in terms of the SNR. 
\begin{lemma}\label{lem:mislabel_probability_correction}
For sufficiently large $N$ and  any $u\in \gU_1 = \gZ \cap \gV_1$, we have
    \begin{align}\label{eqn:rhoC_k}
        \widehat{\rho}_{1}^{(0)} \coloneqq \P \left( \widehat{\rS}_{11}^{(0)}(u) \leq \mu_{\rm{C}}\right) \leq \rho \,,\quad \widehat{\rho}_{j}^{(0)} \coloneqq\P \left( \widehat{\rS}_{1j}^{(0)}(u) \geq \mu_{\rm{C}}\right) \leq \rho\,,\,(j\neq 1),
    \end{align}
where $\rho \coloneqq \exp\left( -\const_{\sL}(K, \nu) \cdot \mathrm{SNR}_{\sL}(K) \right)$ with $\const_{\sL}(K, \nu)$ and $\mathrm{SNR}_{\sL}(K)$ defined in \eqref{eqn:CMk} and \eqref{eqn:SNRk} respectively.
\end{lemma}

As a result, the probability that either of those events happened is bounded by $\rho$. For vertex $v\in \gV$, let $\Gamma_{v}$ and $\Lambda_{v}$ denote the i.i.d indicator random variables with means $\widehat{\rho}_{1}^{(0)}$ and $\widehat{\rho}_{j}^{(0)}$ ($j\neq 1$) respectively. In the true community $\gU_1$, the number of mislabeled vertices is upper bouned by
\begin{align*}
    \widehat{\rR}_{1}^{(0)} = \sum_{v \in \gU_{1}}\Gamma_{v}\, + \sum_{j=2}^{K}\sum_{v \in \gU_1\cap \widehat{\gU}_{j}^{(0)}}\Lambda_{t}\,.
\end{align*}
Together with \Cref{lem:mislabel_probability_correction}, we have
\begin{align*}
   \E \widehat{\rR}_{1}^{(0)} \leq  \frac{N}{2K} \widehat{\rho}_{1}^{(0)} + \sum_{j=2}^{K}\frac{(1 - \nu)N}{2K} \widehat{\rho}_{j}^{(0)} \leq \frac{N}{2K} \cdot K\rho = \frac{N\rho}{2}\,. 
\end{align*}
Let $t_1\coloneqq N \rho/2$, where $\nu$ denotes the correctness after \Cref{alg:multi_block_spectral_partition}, then by Chernoff bound (\Cref{lem:Chernoff}),
\begin{align}\label{eqn:wrong_prob}
   \P \left( \widehat{\rR}_{1}^{(0)} \geq N\rho \right) = \P \left( \widehat{\rR}_{1}^{(0)} - N \rho/2 \geq t_1 \right) \leq \P \left( \widehat{\rR}_{1}^{(0)} - \E \widehat{\rR}_{1}^{(0)} \geq t_1 \right) \leq e^{-c t_1} = O(e^{-N\rho})\,.
\end{align}
Then with probability $1 - O(e^{-N\rho})$, the fraction of mislabeled vertices in $\gU_1$ is smaller than $K\rho$, i.e., the correctness of $\gU_1$ is at least $\gamma_{\rm{C}} \coloneqq \max\{\nu\,, 1 - K\rho\}$. Therefore, \Cref{alg:multi_block_correction} outputs a $\gamma_{\rm{C}}$-correct partition $\widehat{\gU}_{1}, \ldots, \widehat{\gU}_{K}$ with probability  $1- O(e^{-N\rho})$.

\subsection{Merging: Proof of Lemma \ref{lem:merging_accuracy_k}}\label{subsec:merging}
By running \Cref{alg:multi_block_correction} on the red hypergraph, we obtain a $\gamma_{\rm{C}}$-correct partition $\widehat{\gU}_{1}$, $\ldots$, $\widehat{\gU}_{K}$ where $\gamma_{\rm{C}}\coloneqq\max\{\nu\,, 1 - K\rho\} \geq \nu$, i.e.,
\begin{equation}
    |\gU_{j}\cap \widehat{\gU}_{j}| \geq \nu\cdot |\widehat{\gU}_{j}| =  \frac{\nu N}{2K}\,, \quad \forall j\in [K]\,.\label{eqn:correction_accuracy_k}
\end{equation}
In the rest of this subsection, we shall condition on the event \eqref{eqn:correction_accuracy_k} and denote $\gW_{j} \coloneqq \gY\cap \gV_{j}$. The failure probability of \Cref{alg:multi_block_merging} is estimated by the presence of hyperedges between vertex sets $\gY$ and $\gZ$.


Consider a hyperedge $e = \{i_{1}, \cdots, i_{\ell}\}$ in the underlying $\ell$-uniform hypergraph. If vertices $i_{1}, \cdots, i_{\ell}$ are all from the same cluster $\gV_{j}$, then the probability that $e$ is an existing blue edge conditioning on the event that $e$ is not a red edge is 
\begin{align}\label{eqn:psi_l}
    \psi_{\ell}\coloneqq\P\left[e \text{ is a blue edge } \Big| e \text{ is not a red edge } \right] = \frac{\frac{a_{\ell}}{2\binom{N - 1}{\ell -1}}}{1 - \frac{a_{\ell}}{2\binom{N - 1}{\ell -1}}} = \big(1+o(1) \big) \cdot \frac{a_{\ell}}{2\binom{N - 1}{\ell -1}}\,,
\end{align}
and the presence of $e$ can be represented by an indicator random variable $\zeta_e^{(a_{\ell})} \sim \mathrm{Bernoulli}\left(\psi_{\ell}\right)$. Similarly, if not all the vertices $i_{1}, \cdots, i_{\ell}$ are from the same cluster $\gV_{j}$, the probability that $e$ is an existing blue edge conditioning on the event that $e$ is not red
\begin{align}\label{eqn:phi_l}
    \phi_{\ell}\coloneqq\P\left[e \text{ is a blue edge } \Big| e \text{ is not a red edge } \right] = \frac{\frac{b_{\ell}}{2\binom{N - 1}{\ell -1}}}{1 - \frac{b_{\ell}}{2\binom{N - 1}{\ell -1}}} = \big(1+o(1) \big) \cdot \frac{b_{\ell}}{2\binom{N - 1}{\ell -1}}\,,
\end{align}
and the presence of $e$ can be represented by an indicator variable $\xi_e^{(b_{\ell})} \sim \mathrm{Bernoulli}\left(\phi_{\ell}\right)$. 

For any vertex $w\in \gW_{j} \coloneqq \gY\cap \gV_{j}$ with fixed $j \in [K]$, we want to compute the number of hyperedges containing $w$ with all remaining vertices located in vertex set $\widehat{\gU}_{j}$ for some fixed $k\in [K]$. Following a similar argument given in \Cref{subsec:local_correction}, this number can be written as
\begin{align}\label{eqn:weighted_neighbor_Y}
    \widehat{\rS}_{jk}(w)\coloneqq \sum_{\ell \in \sL}(\ell - 1)\cdot \left\{ \sum_{e\in \, \widehat{\gE}^{(a_{\ell})}_{j, k}} \zeta_e^{(a_{\ell})} + \sum_{e\in \widehat{\gE}^{(b_{\ell})}_{j, k} } \xi_e^{(b_{\ell})}\right\}\,,\quad w\in \gW_{j}\,,
\end{align}
 where $\widehat{\gE}^{(a_{\ell})}_{j, k}\coloneqq \gE_{\ell}([\gW_{j}]^{1}, [\gU_{j}\cap \widehat{\gU}_j]^{\ell - 1})$ denotes the set of $\ell$-hyperedges with $1$ vertex from $[\gW_{j}]^{1}$ and the other $\ell - 1$ vertices from $[\gU_{j}\cap \widehat{\gU}_j]^{\ell - 1}$, while
    $
        \widehat{\gE}^{(b_{\ell})}_{j, k} \coloneqq \gE_{\ell} ([\gW_{j}]^{1}, \,\, [\widehat{\gU}_j]^{\ell - 1} \setminus [\gU_{j}\cap \widehat{\gU}_j]^{\ell - 1})
    $
    denotes the set of $\ell$-hyperedges with $1$ vertex in $[\gW_{j}]^{1}$ while the remaining $\ell - 1$ vertices are in $[\widehat{\gU}_j]^{\ell - 1}\setminus [\gU_{j}\cap \widehat{\gU}_j]^{\ell - 1}$, with their cardinalities 
    \begin{align*}
       |\widehat{\gE}^{(a_{\ell})}_{j, k}| = \binom{|\gU_{j} \cap \widehat{\gU}_{j}|}{\ell - 1}\,,\quad |\widehat{\gE}^{(b_{\ell})}_{j, k}| = \left[ \binom{|\widehat{\gU}_{j}|}{\ell - 1} - \binom{|\gU_{j} \cap \widehat{\gU}_{j}|}{\ell - 1} \right]\,. \notag 
    \end{align*}
Similarly, we multiply $(\ell - 1)$ in \eqref{eqn:weighted_neighbor_Y} as weight since the rest $\ell - 1$ vertices can be regarded as the neighbors of $w$ in $\widehat{\gU}_{j}$. By accuracy of \Cref{alg:multi_block_correction} in \eqref{eqn:correction_accuracy_k}, $|\widehat{\gU}_{j} \cap \gU_j| \geq \nu N/(2K)$, then
\begin{align*}
    |\widehat{\gE}^{(a_{\ell})}_{j, j}| \geq \binom{\frac{\nu N}{2K}}{\ell - 1}\,,\quad |\widehat{\gE}^{(a_{\ell})}_{j, k}| \leq \binom{\frac{(1 - \nu)N}{2K}}{\ell - 1}\,,\,\, j\neq l\,. \notag 
\end{align*}
Taking expectation with respect to $\zeta_e^{(a_{\ell})}$ and $\xi_e^{(b_{\ell})}$,  for any $w \in \gW_{j}$, we have
\begin{subequations}
    \begin{align*}
    \E \widehat{\rS}_{jj}(w) &= \sum_{\ell \in \sL} (\ell - 1)\cdot\left[ \binom{\frac{\nu N}{2K}}{\ell - 1} (\psi_{\ell} - \phi_{\ell}) + \binom{\frac{N}{2K}}{\ell - 1} \phi_{\ell}\right],\\
    \E \widehat{\rS}_{jk}(w) &= \sum_{\ell \in \sL} (\ell - 1)\cdot\left[ \binom{\frac{(1 - \nu) n}{2K}}{\ell - 1} (\psi_{\ell} - \phi_{\ell}) +  \binom{\frac{N}{2K}}{\ell - 1} \phi_{\ell} \right],\,\, j\neq l\,.
\end{align*}
\end{subequations}
By assumptions in \Cref{thm:partial_multiple}, $\E \widehat{\rS}_{jj}(w) - \E \widehat{\rS}_{jk}(w) = \Omega(1)$. We define
\begin{align}
    \mu_{\rm{M}} \coloneqq \frac{1}{2}\sum_{\ell \in \sL} (\ell - 1)\cdot\left\{ \left[ \binom{\frac{\nu N}{2K}}{\ell - 1} + \binom{\frac{(1 - \nu)N}{2K}}{\ell - 1} \right](\psi_{\ell} - \phi_{\ell}) + 2\binom{\frac{N}{2K}}{\ell - 1} \phi_{\ell}\right\}\,.\label{eqn:mu_merge}
\end{align}
After \Cref{alg:multi_block_merging}, if a vertex $w\in \gW_{j}$ is mislabelled, one of the following events must  happen
\begin{itemize}
    \item $\widehat{\rS}_{jj}(w) \leq \mu_{\rm{M}}$, which implies that $u$ fails to have enough neighbors in $\widehat{\gU}_{j}$ to be assigned to $\widehat{\gU}_{j}$.
    \item $\widehat{\rS}_{jk}(w) \geq \mu_{\rm{M}}$ for some $j\neq k$, which implies that $w$ survived \Cref{alg:multi_block_merging} without being corrected.
\end{itemize}
By an argument similar to \Cref{lem:mislabel_probability_correction}, we can prove that for any $w\in \gW_{j}$,
\begin{align}
    \widehat{\rho}_{j} \coloneqq \P( \widehat{\rS}_{jj}(w)\leq \mu_{\rm{M}} ) \leq \rho\,,\quad \widehat{\rho}_{k} \coloneqq \P ( \widehat{\rS}_{jk}(w) \geq \mu_{\rm{M}} ) \leq \rho\,, \,(j\neq l), \notag
\end{align}
where $\rho \coloneqq \exp\left( -\const_{\sL}(K, \nu) \cdot \mathrm{SNR}_{\sL}(K) \right)$ with $\const_{\sL}(K, \nu)$ and $\mathrm{SNR}_{\sL}(K)$ defined in \eqref{eqn:CMk} and \eqref{eqn:SNRk} respectively. The misclassified probability for $w\in \gW_{j}$ is then upper bounded by $\sum_{j=1}^{K}\widehat{\rho}_{j}$.
Let $\Gamma_{t}$ be the i.i.d indicator random variables with mean $K \rho$. Then, the number of mislabelled vertices in $\gW_{j}$ is at most
\begin{align}
    \widehat{\rR}_j = \sum_{v\in \gW_{j}} \Gamma_{v}.
\end{align}
Consequently, $\E \widehat{\rR}_j \leq N/(2K)\cdot K\rho = N\rho/2$. Let $t_{j} \coloneqq N\rho/2$, by Chernoff bound (\Cref{lem:Chernoff}),
\begin{align*}
   \P \left( \widehat{\rR}_j \geq N\rho \right) = \P \left( \widehat{\rR}_j - N\rho/2\geq t_j \right) \leq \P \left( \widehat{\rR}_j - \E \widehat{\rR}_j \geq t_j \right) \leq e^{-c t_j} = O(e^{-N\rho})\,.
\end{align*}
Hence with probability $1 - O(e^{-N\rho})$, the fraction of mislabeled vertices in $\gW_{j}$ is smaller than $K\rho$, i.e., the correctness in $\gW_{j}$ is at least $\gamma_{\rm{M}} \coloneqq\max\{ \nu,\, 1 - K\rho \}$.

\subsection{Proof of  \Cref{thm:partial_multiple}}
Now we are ready to prove \Cref{thm:partial_multiple}. The correctness of \Cref{alg:multi_block_correction} and \Cref{alg:multi_block_merging} are denoted by $\gamma_{\rm{C}}$ and $\gamma_{\rm{M}}$ respectively, then with probability at least $1 - O(e^{-N\rho})$, the correctness $\gamma$ of \Cref{alg:multiple_partition_partial} is
$
    \gamma \coloneqq \min\{\gamma_{\rm{C}}, \gamma_{\rm{M}}\} = \max\{\nu,\, 1 - K\rho \}.
$
We will have $\gamma = 1 - K\rho$ if $\nu \leq 1 - K\rho$, equivalently,
\begin{align}\label{eqn:gammacondition}
    \mathrm{SNR}_{\sL}(K) \geq \frac{1}{\const_{\sL}(K, \nu)} \log \Big( \frac{K}{1 - \nu}\Big)\,,
\end{align}
otherwise $\gamma = \nu$. The inequality \eqref{eqn:gammacondition} holds since
\begin{subequations}
    \begin{align*}
        \mathrm{SNR}_{\sL}(K)
    =&\, \frac{ \left[\sum_{\ell \in \sL} (\ell - 1)(a_{\ell} - b_{\ell})K^{-\ell + 1} \right]^2 }{\sum_{\ell \in \sL} (\ell - 1)\cdot \big( (a_{\ell} - b_{\ell})K^{-\ell + 1} + b_{\ell} \big)}\\
    \geq &\, \frac{[\sum_{\ell \in \sL} (a_{\ell}-b_{\ell})]^2}{K^{2\LM -2}(\LM -1)d}\geq \frac{(\const_{\nu})^2}{\LM -1} \log \Big(\frac{K}{1 - \nu} \Big) \geq \frac{1}{\const_{\sL}(K, \nu)}\log \Big(\frac{K}{1-\nu} \Big)\,.
    \end{align*}
\end{subequations}
where the first two inequalities hold since $d \coloneqq \sum_{\ell \in \sL}(\ell - 1)a_{\ell}$ and condition \eqref{eqn:SNR_condition}, while the last inequality holds by taking $\const_{\nu}\geq \max\{ \sqrt{(\LM -1)/\const_{\sL}(K, \nu) }, \,\,\const_{2} \}$ with $\const_{2}$ defined in \eqref{eqn:C2}.

\begin{remark}\label{rem:two_constants}
The lower bound $\const_{\eqref{eqn:degree_constant}}$ in \eqref{eqn:degree_constant} comes from the requirement in \Cref{lem:high_degree_vertices_k} that only a few high degree vertices be deleted. The constant $\const_{\nu}$ in \eqref{eqn:SNR_condition} comes from the requirement in \Cref{lem:subspace_angle_multi} that the subspace angle is small. When $\const_{\eqref{eqn:degree_constant}}$ is not so large (or the hypergraph is too sparse), one could still achieve good accuracy $\gamma$ if $\const_{\nu}$ is large enough (the difference between $a_{\ell}$ and $b_{\ell}$ is large enough).
\end{remark}

\begin{remark}
The condition \eqref{eqn:gammacondition} indicates that the improvement of accuracy from local refinement (\Cref{alg:multi_block_correction} and \Cref{alg:multi_block_merging}) will be guaranteed when $\mathrm{SNR}_{\sL}(K)$ is large enough. 
If $\mathrm{SNR}_{\sL}(K)$ is small,  we use correctness of \Cref{alg:multi_block_spectral_partition} instead, i.e., $\gamma = \nu$, to represent the correctness of \Cref{alg:multiple_partition_partial}. 
\end{remark}

\subsection{Proof of \Cref{cor:weak_consistency}}
    For any fixed $\nu \in (1/K, 1)$, $\mathrm{SNR}_{\sL}(K) \to \infty$ implies $\rho \to 0$ and 
    \begin{align}
        d \coloneqq \sum_{\ell \in \sL}(\ell - 1)a_{\ell} \to \infty.
    \end{align}
    We now check the condition \eqref{eqn:SNR_condition} in \Cref{ass:regimes_partial}. Note that $\ell \geq 2$ for each $\ell \in \sL$, then
    \begin{align*}
        &\,\frac{\big[\sum_{\ell \in \sL} (\ell - 1)(a_{\ell}-b_{\ell}) \big]^{2}}{d} \\
    =&\, K^{\LM - 1} \cdot K^{2} \frac{\big[\sum_{\ell \in \sL} (\ell - 1)(a_{\ell}-b_{\ell}) K^{-2 + 1} \big]^{2}}{\sum_{\ell \in \sL}(\ell - 1)a_{\ell}\, K^{-\LM + 1}} \\
        \geq &\,K^{\LM + 1} \frac{\big[\sum_{\ell \in \sL} (\ell - 1)(a_{\ell}-b_{\ell})K^{-\ell + 1} \big]^{2}}{\sum_{\ell \in \sL}(\ell - 1)a_{\ell} \, K^{-\ell + 1}}\\
            \geq &\,\frac{\big[\sum_{\ell \in \sL} (\ell - 1)(a_{\ell}-b_{\ell})K^{-\ell + 1} \big]^{2}}{\sum_{\ell \in \sL}(\ell - 1)\big((a_{\ell} - b_{\ell})\, K^{-\ell + 1} + b_{\ell}\big)} = \mathrm{SNR}_{\sL}(K) \to \infty.
    \end{align*}
Consequently, there exists some $\nu$ close to $1$ and some constant $\const_{\nu} >0$ such that condition \eqref{eqn:SNR_condition} is satisfied. Applying \Cref{thm:partial_multiple}, we find $\gamma = 1 - o(1)$, which implies weak consistency. The constraint of $\mathrm{SNR}_{\sL}(K) = o(\log(N))$ is used in the proof of \Cref{lem:correction_accuracy_multiple}, see \Cref{rem:SNRlogn}.

\section*{Acknowledgements}
I.D. and H.X.W. are partially supported by  NSF DMS-2154099.  I.D. and Y.Z. acknowledge support from NSF  DMS-1928930 during their participation in the program Universality and Integrability in Random Matrix Theory and Interacting Particle Systems hosted by the Mathematical Sciences Research Institute in Berkeley, California during the Fall semester of 2021.  Y.Z. is partially supported by  NSF-Simons Research Collaborations on the Mathematical and Scientific Foundations of Deep Learning. Y.Z. thanks Zhixin Zhou for his helpful comments. The authors thank the anonymous reviewers for detailed comments and suggestions which greatly improve the presentation of this work.

\printbibliography

\newpage
\appendix
\phantomsection
\addcontentsline{toc}{section}{Appendices}
\section{Proof of \Cref{thm:concentration_partial} and \Cref{thm:regularization_concentration_partial}}\label{sec:proof_concentration_partial}

\subsection{Discretization}
To prove \Cref{thm:concentration_partial}, we start with a standard $\epsilon$-net argument.

\begin{lemma}[{\cite[Lemma 4.4.1]{Vershynin2018HighDP}} ]
Let $\rmW$ be any $N \times N$ Hermitian matrix and let $\sN_{\epsilon}$ be an $\epsilon$-net on the unit sphere $\S^{N - 1}$ with $\epsilon \in (0,1)$, then 
$
\|\rmW\| \leq \frac{1}{1-\epsilon}\sup_{\rvx \in \sN_{\epsilon}} |\langle \rmW \rvx, \rvx\rangle |.
$
\end{lemma}

By {\cite[Corollary 4.2.13]{Vershynin2018HighDP}}, the size of $\sN_{\epsilon}$ is bounded by $|\sN_{\epsilon}|\leq (1+2/\epsilon)^{N}$. We would have $\log|\sN_{\epsilon}| \leq N \log(5)$ when $\sN_{\epsilon}$ is taken as an $(1/2)$-net of $\S^{N}$. Define $\rmW \coloneqq \rmA - \E \rmA$, then $\ermW_{ii}=0$ for each $i\in [N]$ by the definition of adjacency matrix in \eqref{eqn:expected_adjacency_multiple}, and we obtain 
\begin{align}
    \|\rmA - \E \rmA\| = \|\rmW\| \leq 2\sup_{\rvx \in \sN }| \langle \rmW \rvx, \rvx\rangle|\,.\label{eqn:epsilon_net_partial}
\end{align}
For any fixed $\rvx \in \S^{N - 1}$, consider the \textit{light} and \textit{heavy} pairs as follows.
\begin{align}
    &\light(\rvx)=\Bigg\{ (i,j): |\ervx_{i} \ervx_{j}|\leq {\frac{\sqrt d}{N}}\Bigg\}, \quad \heavy(\rvx)=\Bigg\{ (i,j): |\ervx_{i} \ervx_{j}|> {\frac{\sqrt d}{N}}\Bigg\}\,,\label{eqn:light_heavy_pair_partial}
\end{align}
where $d = \sum_{\ell \in \sL}(\ell - 1)d_{\ell}$. Thus by the triangle inequality, 
\[
    | \langle \rvx, \rmW \rvx\rangle | \leq \left|\sum_{(i,j)\in \light(\rvx)} \ermW_{ij}\ervx_{i} \ervx_{j} \right| + \left|\sum_{(i,j)\in \heavy(\rvx)} \ermW_{ij}\ervx_{i} \ervx_{j} \right|.
\] 
We denote $\sN \coloneqq \sN_{\epsilon}$ for convenience. Then by \eqref{eqn:epsilon_net_partial},
\begin{align}
    \|\rmA - \E \rmA\| \leq 2\sup_{\rvx \in \sN }\Bigg|\sum_{(i,j)\in \light(\rvx)} \ermW_{ij} \ervx_{i} \ervx_{j}\Bigg| + 2\sup_{\rvx \in \sN }\Bigg|\sum_{(i,j)\in \heavy(\rvx)} \ermW_{ij} \ervx_{i} \ervx_{j} \Bigg|\,.\label{eqn:contribution_partial}
\end{align}

\subsection{Contribution from light pairs}\label{sec:light_partial}
For each $\ell$-hyperedge $e\in \gE_{\ell}$, we define $\etW_{e}^{(\ell)} \coloneqq \etA_{e}^{(\ell)} - \E\etA_{e}^{(\ell)}.$ Then for any fixed $\rvx \in \S^{N - 1}$, the contribution from light couples can be written as
\begin{align} 
       &\,\sum_{(i,j)\in \light(\rvx)} \ermW_{ij} \ervx_{i} \ervx_{j} = \sum_{(i,j)\in \light(\rvx)} \left( \sum_{\ell \in \sL}\,\, \sum_{\substack{e\in \gE_{\ell}\\ \{i,j\}\subset e} } \etW_{e}^{(\ell)} \right) \ervx_{i}\ervx_{j}  \notag \\
       =&\, \sum_{\ell \in \sL} \sum_{e\in \gE_{\ell}} \etW_{e}^{(\ell)} \left( \sum_{\substack{(i,j)\in \light(\rvx)\\ i\neq j,\, \{i,j\}\subset e} }\ervx_{i} \ervx_{j} \right) = \sum_{\ell \in \sL} \sum_{e\in  \gE_{\ell}} \etY_{e}^{(\ell)} \,,\label{eqn:lightW_partial}
\end{align}
where the constraint $i\neq j$ comes from the fact $\ermW_{ii} =0$ and we denote
\begin{align*}
    \etY_{e}^{(\ell)}\coloneqq \etW_{e}^{(\ell)} \left( \sum_{\substack{(i,j)\in \light(\rvx)\\i\neq j, \, \{i,j\}\subset e}}\,\, \ervx_{i} \ervx_{j}\right)\,.
\end{align*}
Note that $\E \etY_{e}^{(\ell)} = 0$, and by the  definition of light pair \eqref{eqn:light_heavy_pair_partial},
\begin{align*}
|\etY_{e}^{(\ell)}|\leq \ell(\ell - 1)\sqrt{d}/N \leq \ell(\ell - 1)\sqrt{d}/N\,, \quad \forall \ell \in \sL\,.
\end{align*}
Moreover, \eqref{eqn:lightW_partial} is a sum of independent, mean-zero random variables, and
\begin{align*}
    \sum_{\ell \in \sL} \sum_{e\in \gE_{\ell}} \E [(\etY_{e}^{(\ell)})^2] \coloneqq &\,\sum_{\ell \in \sL} \sum_{e\in \gE_{\ell}} \Bigg[ \E[(\etW_{e}^{(\ell)})^2] \Bigg( \sum_{\substack{(i,j)\in \light(\rvx)\\ i\neq j, \{i,j\}\subset  e}}\ervx_{i} \ervx_{j} \Bigg)^2 \Bigg] \\
    \leq &\, \sum_{\ell \in \sL} \sum_{e\in \gE_{\ell}} \Bigg[ \E[ \etA_{e}^{(\ell)} ] \cdot \ell(\ell - 1) \Bigg(\sum_{\substack{(i,j)\in \light(\gX)\\ i\neq j, \{i,j\}\subset e}} \ervx_{i}^2 \ervx_{j}^2 \Bigg) \Bigg]\\ 
    \leq &\, \sum_{\ell \in \sL} \frac{d_{\ell}  \cdot  \ell(\ell - 1)}{\binom{N - 1}{\ell-1}}\binom{N}{\ell - 2} \sum_{(i,j)\in [N]^2}\ervx_{i}^2 \ervx_{j}^2 \\
    \leq &\, \sum_{\ell \in \sL} \frac{d_{\ell} \ell(\ell - 1)^2}{N-\ell+2} \leq \frac{2}{N}\sum_{\ell \in \sL} d_{\ell} (\ell - 1)^3\leq \frac{2d(\ell - 1)^2}{N}\,,
\end{align*}
when $N\geq 2\ell - 2$, where $d_{\ell}=\max d_{[i_1,\dots, i_{\ell}]}$ and $d = \sum_{\ell \in \sL}(\ell - 1)d_{\ell}$. Then Bernstein's (\Cref{lem:Bernstein}) implies that for any $\alpha>0$,
\begin{subequations}
\begin{align*}
    &\P \Bigg( \Bigg| \sum_{(i,j)\in\light(\rvx)} \ermW_{ij}\ervx_{i} \ervx_{j}\Bigg| \geq \alpha\sqrt{d}\Bigg) = \P \Bigg( \Bigg| \sum_{\ell \in \sL} \sum_{e\in  \gE_{\ell}} \etY_{e}^{(\ell)} \Bigg| \geq \alpha\sqrt{d}\Bigg) \\
    \leq & 2\exp \Bigg( -\frac{\frac{1}{2}\alpha^2d}{\frac{2d}{N}(\ell - 1)^2+ \frac{1}{3} (\ell - 1)\LM \frac{\sqrt{d}}{N} \alpha\sqrt{d}}\Bigg)
    \leq 2\exp \Bigg( -\frac{\alpha^2 N}{4(\ell - 1)^2+\frac{2\alpha(\ell - 1)\LM}{3}}\Bigg).
\end{align*}
\end{subequations}
Therefore by taking a union bound,
    \begin{align}
        &\,\P \Bigg(  \sup_{\rvx \in \sN }\Bigg| \sum_{(i,j)\in\light(\rvx)} \ermW_{ij} \ervx_{i} \ervx_{j} \Bigg|\geq \alpha\sqrt{d}\Bigg)\leq |\sN| \cdot \P \Bigg( \Bigg| \sum_{(i,j)\in\light(\rvx)} \ermW_{ij}\ervx_{i} \ervx_{j}\Bigg| \geq \alpha\sqrt{d}\Bigg) \notag\\
        \leq&\, 2\exp \Bigg( \log(5) \cdot N-\frac{\alpha^2 N}{4(\ell - 1)^2+\frac{2\alpha(\ell - 1)\LM}{3}}\Bigg)\leq 2e^{-N}\, , \label{eqn:lightbound_partial}
    \end{align}
where we choose $\alpha=5\ell(\ell - 1)$ in the last line. 

\subsection{Contribution from heavy pairs}
Note that for any $i\neq j$,
\begin{align}
    \E \ermA_{ij}\leq \sum_{\ell \in \sL} \binom{N - 2}{\ell - 2}\frac{d_{\ell}}{\binom{N - 1}{\ell-1}} \leq  \sum_{\ell \in \sL}\frac{(\ell - 1)d_{\ell}}{N}=\frac{d}{N}.\label{eqn:EA_partial}
\end{align}
and
\begin{align}\label{eqn:heavy_exp_partial}
   &\Bigg|\sum_{(i,j)\in \heavy(\rvx)} \E \ermA_{ij} \ervx_{i} \ervx_{j}\Bigg| = \Bigg|\sum_{(i,j)\in \heavy(\rvx)} \E \ermA_{ij} \frac{\ervx_{i}^2 \ervx_{j}^2}{\ervx_{i} \ervx_{j}}\Bigg| \\
    \leq \,& \sum_{(i,j)\in \heavy(\rvx)} \frac{d}{N}\frac{\ervx_{i}^2 \ervx_{j}^2}{|\ervx_{i}\ervx_{j}|}\leq \sqrt{d}\sum_{(i,j)\in \heavy(\rvx)} \ervx_{i}^2 \ervx_{j}^2\leq \sqrt{d}. \notag
\end{align}
Therefore it suffices to show that, with high probability, 
\begin{align}\label{eqn:Osqrtd_partial}
    \sum_{(i,j)\in \heavy(\rvx)}  \ermA_{ij} \ervx_{i} \ervx_{j}= O \Big(\sqrt{d} \Big).
\end{align}
Here we use the discrepancy analysis from \cite{Feige2005SpectralTA, Cook2018SizeBC}. We consider the weighted graph associated with the adjacency matrix $\rmA$.

\begin{definition}[Uniform upper tail property, \textbf{UUTP}]
Let $\rmM$ be an $N \times N$ random symmetric matrix with non-negative entries and $\rmQ$ be an $N \times N$ symmetric matrix with entries $\ermQ_{ij}\in [0,a]$ for all $i,j\in [N]$. Define
\begin{align*}
    \mu \coloneqq \sum_{i,j=1}^{N} \ermQ_{ij}\E \ermM_{ij}, \quad \tilde{\sigma}^2 \coloneqq\sum_{i,j=1}^{N} \ermQ_{ij}^2 \E \ermM_{ij}.
\end{align*}
We say that $\rmM$ satisfies the uniform upper tail property $\textnormal{\textbf{UUTP}}(c_0,\gamma_0)$ with $c_0>0,\gamma_0\geq 0$, if for any $a,t>0$,
\begin{align*}
    \P \Bigg(  f_{\rmQ}(\rmM)\geq (1 + \gamma_0)\mu +t\Bigg)\leq \exp \Bigg( -c_0 \frac{\tilde{\sigma}^2}{a^2} h\Bigg( \frac{at}{\tilde{\sigma}^2}\Bigg)\Bigg).
\end{align*}
where function $f_{\rmQ}(\rmM): \R^{N \times N} \mapsto \R$ is defined by $f_{\rmQ}(\rmM)\coloneqq \sum_{i,j=1}^{N} \ermQ_{ij} \ermM_{ij}$ for $\rmM \in \R^{N \times N}$, and function $h(\gX)\coloneqq(1 + x)\log(1+x) - x$ for all $x>-1$.
\end{definition}

\begin{lemma}\label{lem:UUTPforA_partial}
Let $\rmA$ be the adjacency matrix of non-uniform hypergraph $\gH = \cup_{\ell \in \sL}\gH_{\ell}$, then $\rmA$ satisfies $\textnormal{\bf{UUTP}}(c_0,\gamma_0)$ with $c_0 = [\ell(\ell - 1)]^{-1}, \gamma_0 = 0$.
\end{lemma}
\begin{proof}[Proof of \Cref{lem:UUTPforA_partial}]
Note that
\begin{align*}
    &\,f_{\rmQ}(\rmA) - \mu = \sum_{i,j=1}^{N} \ermQ_{ij}(\ermA_{ij} - \E \ermA_{ij}) = \sum_{i,j=1}^{N} \ermQ_{ij}\ermW_{ij}\\
    = &\, \sum_{i,j=1}^{N} \ermQ_{ij} \Bigg( \sum_{\ell \in \sL}\,\, \sum_{\substack{e\in \gE_{\ell}\\ i\neq j,\, \{i,j\}\subset e} } \etW_{e}^{(\ell)}  \Bigg) =\sum_{\ell \in \sL} \,\, \sum_{e\in \gE_{\ell}} \etW_{e}^{(\ell)} \Bigg( \sum_{\{i,j\}\subset e, \, i\neq j}\ermQ_{ij}\Bigg) =\sum_{\ell \in \sL} \,\,\sum_{e\in \gE_{\ell}}\etZ_{e}^{(\ell)}\,,
\end{align*}
where $\etZ_{e}^{(\ell)} = \etW_{e}^{(\ell)} \big( \sum_{\{i,j\}\subset e, i\neq j} \ermQ_{ij}\big)$ are independent centered random variables upper bounded by
$
    |\etZ_{e}^{(\ell)}| \leq \sum_{ \{i,j\} \subset e,\, i\neq j}\ermQ_{ij}\leq \ell(\ell - 1)a
$
for each $ \ell \in \{2, \dots, M\}$ since  $\ermQ_{ij} \in [0, a]$. Moreover, the variance of the sum can be written as
\begin{align*}
    &\, \sum_{\ell \in \sL} \sum_{e\in \gE_{\ell}} \E(\etZ_{e}^{(\ell)})^2 = \sum_{\ell \in \sL} \,\, \sum_{e\in \gE_{\ell}} \E(\etW_{e}^{(\ell)})^2 \Bigg( \sum_{\{i,j\}\subset e, \, i\neq j} \ermQ_{ij}\Bigg)^2\\
    \leq &\,\sum_{\ell \in \sL} \,\,\sum_{e\in \gE_{\ell}}\E [\etA_{e}^{(\ell)}] \cdot \ell(\ell - 1) \sum_{ \{i,j \}\subset e, i\neq j }\ermQ_{ij}^2
    \leq \ell(\ell - 1)\sum_{i,j=1}^{N} \ermQ_{ij}^2 \E \ermA_{ij} = \ell(\ell - 1)\tilde{\sigma}^2.
\end{align*}
where the last inequality holds since by definition $\E \ermA_{ij} = \sum_{\ell \in \sL} \,\,\sum_{\substack{e\in \gE_{\ell}\\ \{i,j\}\subset e}}\E [\etA_{e}^{(\ell)}]$. Then by Bennett's inequality \Cref{lem:Bennett}, we obtain
\begin{align*}
    \P (f_{\rmQ}(\rmA)- \mu\geq t) \leq \exp\Bigg( -\frac{\tilde{\sigma}^2}{\ell(\ell - 1)a^2} h \Bigg( \frac{at}{\tilde{\sigma}^2}\Bigg) \Bigg)\,
\end{align*}
where the inequality holds since the function $x\cdot h(1/x) = (1+x)\log(1 + 1/x) - 1$ is decreasing with respect to $\gX$.
\end{proof}

\begin{definition}[Discrepancy property, \bf{DP}]
Let $\rmM$ be an $N \times N$ matrix with non-negative entries. For $S,T\subset [N]$, define
$
    e_{\rmM}(S,T)=\sum_{i\in S, j\in T}\ermM_{ij}.
$
We say $\rmM$ has the discrepancy property with parameter $\delta>0$, $\kappa_1>1, \kappa_2\geq 0$, denoted by $\textnormal{\bf{DP}}(\delta,\kappa_1,\kappa_2)$, if for all non-empty $S,T\subset [N]$, at least one of the following hold:
\begin{enumerate}
    \item $e_{\rmM}(S,T)\leq \kappa_1 \delta |S| |T|;$
    \item  $e_{\rmM}(S,T) \cdot \log \Bigg(\frac{e_{\rmM}(S,T)}{\delta |S|\cdot |T|}\Bigg)\leq \kappa_2 (|S|\vee |T|)\cdot \log \Bigg(\frac{en}{|S|\vee |T|}\Bigg) $.
\end{enumerate}
\end{definition}

\Cref{lem:DPforM_partial} shows that if a symmetric random matrix $\rmA$ satisfies the upper tail property $\textnormal{\textbf{UUTP}}(c_0,\gamma_0)$ with parameter $c_0>0,\gamma_0\geq 0$, then the discrepancy property holds with high probability.

\begin{lemma}[{\cite[Lemma 6.4]{Cook2018SizeBC}}]\label{lem:DPforM_partial}
Let $\rmM$ be an $N \times N$ symmetric random matrix with non-negative entries. Assume that for some $\delta>0$, $\E \ermM_{ij}\leq \delta$ for all $i,j\in [N]$ and $\rmM$ has  $\textnormal{\textbf{UUTP}}(c_0,\gamma_0)$ with parameter $c_0,\gamma_0 >0$. Then for any $\theta > 0$,  the  discrepancy  property $\textnormal{\bf{DP}}(\delta,\kappa_1,\kappa_2)$ holds for $\rmM$ with probability at least $1 - N^{-\theta}$ with 
$
    \kappa_1 = e^2(1+\gamma_0)^2,  \kappa_2 = \frac{2}{c_0}(1+\gamma_0)(\theta+4).
$
\end{lemma}

When the discrepancy property holds, then deterministically the  contribution from heavy pairs is $O(\sqrt{d})$, as shown in the following lemma.
\begin{lemma}[{\cite[Lemma 6.6]{Cook2018SizeBC}}]\label{lem:heavy_bound_partial}
Let $\rmM$ be a non-negative symmetric $N \times N$ matrix with all row sums bounded by $d$. Suppose $\rmM$ has $\textnormal{\bf{DP}}(\delta,\kappa_1,\kappa_2)$ with $\delta=Cd/N$ for some $C>0,\kappa_1>1,\kappa_2\geq 0$. Then for any $x\in \S^{N - 1}$,
\begin{align*}
    \Bigg| \sum_{(i,j)\in \heavy(\gX)} \ermM_{ij} \ervx_{i} \ervx_{j} \Bigg|\leq \alpha_0\sqrt{d},
\end{align*}
where
$
    \alpha_0=16+32C(1+\kappa_1)+64\kappa_2  (1+\frac{2}{\kappa_1\log\kappa_1}).
$
\end{lemma}

\Cref{lem:max_degree_partial} proves that $\rmA$ has bounded row and column sums with high probability.
\begin{lemma}\label{lem:max_degree_partial}
For any $\theta>0$, there is a constant $\alpha_1>0$ such that with probability at least $1-N^{-\theta}$,
\begin{align}\label{eqn:bounded_degree_partial}
    \max_{1\leq i\leq N} \,\sum_{j=1}^{N} \ermA_{ij} \leq \alpha_1 d
\end{align}
with $\alpha_1= 4 + \frac{2(\ell - 1)(1+\theta)}{3c}$ and $d\geq c\log N$.
\end{lemma}
\begin{proof}
For a fixed $i\in [N]$,
\begin{subequations}
\begin{align*}
   &\sum_{j=1}^{N} \ermA_{ij} =\sum_{\ell \in \sL}\sum_{e\in \gE_{\ell}: i\in e} (\ell - 1)\etA_{e}^{(\ell)}\,, \quad \, \sum_{j=1}^{N} (\ermA_{ij} - \E \ermA_{ij}) =\sum_{\ell \in \sL}\sum_{e\in \gE_{\ell}: i\in e} (\ell - 1)\etW_{e}^{(\ell)}\,, \\
    &\sum_{j=1}^{N}\E \ermA_{ij} \leq \sum_{\ell \in \sL} \binom{N - 1}{\ell-1} \frac{(\ell - 1)d_{\ell}}{\binom{N - 1}{\ell-1}}=d,\\ 
   & \sum_{\ell \in \sL} (\ell - 1)^2\sum_{e\in \gE_{\ell}: i\in e} \E[(\etW_{e}^{(\ell)})^2] \leq \sum_{\ell \in \sL} (\ell - 1)^2\sum_{e\in \gE_{\ell}: i\in e} \E[\etA_{e}^{(\ell)}]\leq (\ell - 1)d\,. 
\end{align*}
\end{subequations}
Then for $\alpha_1= 4 + \frac{2(\ell - 1)(1+\theta)}{3c}$, by Bernstein's inequality, with the assumption that $d\geq c\log N$,
\begin{align} 
    \P\Bigg (\sum_{j=1}^{N} \ermA_{ij}\geq \alpha_1 d \Bigg) &\leq \P \Bigg(\sum_{j=1}^{N} \ermA_{ij} - \E \ermA_{ij}\geq (\alpha_1-1) d\Bigg) \notag\\
    &\leq \exp \Bigg( -\frac{\frac{1}{2}(\alpha_1-1)^2 d^2}{(\ell - 1)d+\frac{1}{3}(\ell - 1)(\alpha_1-1)d}\Bigg) 
    \leq N^{-\frac{3c(\alpha_1-1)^2}{(\ell - 1)(2\alpha_1+4)}}\leq N^{-1-\theta}.
\end{align}
Taking a union bound over $i\in [N]$, then \eqref{eqn:bounded_degree_partial} holds with probability $1-N^{-\theta}$.
\end{proof}

Now we are ready to obtain \eqref{eqn:Osqrtd_partial}.
\begin{lemma}
For any $\theta > 0$,  there is a constant $\beta$ depending on $\theta, c, \LM$ such that with probability at least $1-2N^{-\theta}$,
 \begin{align}\label{eqn:betasqrtd_partial}
  \Bigg|\sum_{(i,j)\in \mathcal H(\gX)}  \ermA_{ij}\ervx_{i} \ervx_{j}\Bigg|\leq \beta \sqrt{d}.   
 \end{align}
\end{lemma}
\begin{proof}
By \Cref{lem:UUTPforA_partial}, $\rmA$ satisfies $\textnormal{\bf{UUTP}}(\frac{1}{\ell(\ell - 1)},0)$. From \eqref{eqn:EA_partial} and \Cref{lem:DPforM_partial}, the property $\textnormal{\bf{DP}}(\delta, \kappa_1,\kappa_2)$ holds for $\rmA$ with probability at least $1-N^{-\theta}$ with 
\begin{align*}
    \delta = \frac{d}{N}, \quad \kappa_1=e^2,\quad  \kappa_2=2\ell(\ell - 1)(\theta+4).
\end{align*}
Let $\gE_1$ be the event that $\textnormal{\textbf{DP}}(\delta, \kappa_1,\kappa_2)$ holds for $\rmA$. Let $\gE_2$ be the event that all row sums of $\rmA$ are bounded by $\alpha_1 d$. Then
$
    \P (\gE_1 \cap \gE_2)\geq 1-2N^{-\theta}.
$
On the event $\gE_1 \cap \gE_2$, by \Cref{lem:heavy_bound_partial}, \eqref{eqn:betasqrtd_partial} holds with 
$\beta= \alpha_0\alpha_1$, where 
 \begin{align*}
    &\alpha_0= 16+32(1+e^2)+128\ell(\ell - 1)(\theta + 4) (1+e^{-2}), 
    \quad \alpha_1= 4+\frac{2(\ell - 1)(1+\theta)}{3c}.
\end{align*} 
\end{proof}

\subsection{Proof of \Cref{thm:concentration_partial}}
\begin{proof}
From \eqref{eqn:lightbound_partial}, with probability at least $1-2e^{-N}$, the contribution from light pairs in \eqref{eqn:contribution_partial} is bounded by  $2\alpha \sqrt{d}$ with $\alpha=5\ell(\ell - 1)$. From \eqref{eqn:heavy_exp_partial} and  \eqref{eqn:betasqrtd_partial}, with probability at least $1-2N^{-\theta}$, the contribution from heavy pairs in \eqref{eqn:contribution_partial} is bounded by $2\sqrt{d}+2\beta\sqrt{d}$.  Therefore with probability at least $1-2e^{-N}-2N^{-\theta}$, 
\[
    \|\rmA - \E \rmA \|\leq \const_{\LM} \sqrt{d}, 
\]
where $\const_{\LM} $ is a constant depending only on $c, \theta, \LM$ such that
$
    \const_{\LM} = 2(\alpha+1+\beta).
$ 
In particular, we can take $\alpha=5\ell(\ell - 1),$ $\beta=512\ell(\ell - 1)(\theta + 5)\left( 2+\frac{(\ell - 1)(1+\theta)}{c}\right)$, and
$
    \const_{\LM} =512\ell(\ell - 1)(\theta + 6)\left( 2+\frac{(\ell - 1)(1+\theta)}{c}\right).
$
This finishes the proof of \Cref{thm:concentration_partial}.
\end{proof}

\subsection{Proof of \Cref{thm:regularization_concentration_partial}}

Let $\setS\subset [N]$ be any given subset.  From \eqref{eqn:lightbound_partial}, with probability at least $1-2e^{-N}$,
\begin{align}\label{eqn:S_partial}
    \sup_{\rvx \in \sN }  \Bigg| \sum_{(i,j)\in \light(\rvx)}{\Bigg(\rmA_{\setS} - \E \rmA_{\setS}\Bigg)}_{ij} \ervx_{i} \ervx_{j}\Bigg|\leq 5\ell(\ell - 1)\sqrt{d}.
\end{align}
Since there are at most $2^{N}$ many choices for $\setS$, by taking a union bound, with probability at least $1-2(e/2)^{-N}$, we have for all $\setS\subset [N]$, \eqref{eqn:S_partial} holds. In particular, by taking $\setS = \sI = \{i\in [N]: \textnormal{row}(i)\leq \tau d\}$, with probability at least $1-2(e/2)^{-N}$, we have
\begin{align}\label{eqn:reg1_partial}
    \sup_{\rvx \in \sN}  \Bigg| \sum_{(i,j)\in \light(\rvx)}[(\rmA - \E \rmA)_{\sI}]_{ij} \ervx_{i} \ervx_{j}\Bigg|\leq 5\ell(\ell - 1)\sqrt{d}.
\end{align}
Similar to \eqref{eqn:heavy_exp_partial}, deterministically,
\begin{align}\label{eqn:reg2_partial}
     \Bigg|\sum_{(i,j)\in \heavy(\rvx)} [(\E \rmA)_{\sI}]_{ij}\ervx_{i} \ervx_{j}\Bigg|\leq (\ell - 1)\sqrt{d}.
\end{align}
Next we show the contribution from heavy pairs for $\rmA_{\sI}$ is bounded.
\begin{lemma}
For any $\theta > 0$,  there is a constant $\beta_{\tau}$ depending on $\theta, c, \LM,\tau$ such that with probability at least $1-N^{-\theta}$,
 \begin{align}\label{eq:betasqrtd_tau}
  \Bigg|\sum_{(i,j)\in \heavy(\rvx)}  [(\rmA)_{\sI}]_{ij}\ervx_{i} \ervx_{j}\Bigg|\leq \beta_{\tau} \sqrt{d}.   
 \end{align}
\end{lemma}
\begin{proof}
Note that $\rmA$ satisfies $\textnormal{\bf{UUTP}}\left(\frac{1}{\ell(\ell - 1)},0 \right)$ from \Cref{lem:UUTPforA_partial}. According to \Cref{lem:DPforM_partial}, with probability at least $1-N^{-\theta}$, $\textnormal{\bf{DP}}(\delta,\kappa_1,\kappa_2)$ holds for $\rmA$ with
\begin{align*}
    \delta =\frac{d}{N}, \quad \kappa_1=e^2,\quad  \kappa_2=2\ell(\ell - 1)(\theta+4).
\end{align*}
The $\textnormal{\bf{DP}}(\delta,\kappa_1,\kappa_2)$ property holds for $\rmA_{\sI}$ as well, since $\rmA_{\sI}$ is obtained from $\rmA$ by restricting to $\sI$. Note that all row sums in $\rmA_{\sI}$ are bounded by $\tau d$. By \Cref{lem:heavy_bound_partial},
\begin{align}\label{eq:reg3}
    \Bigg| \sum_{(i,j)\in \heavy(\rvx)} [\rmA_{\sI}]_{ij}\ervx_{i} \ervx_{j} \Bigg|\leq \alpha_0\sqrt{\tau d},
\end{align}
where we can take
$
    \alpha_0=16+\frac{32}{\tau}(1+e^2)+128\ell(\ell - 1)(\theta+4)\left (1+\frac{1}{e^2}\right).
$
\end{proof}
We can then take $\beta_{\tau}=\alpha_0\sqrt{\tau}$ in \Cref{eq:betasqrtd_tau}. Therefore, combining \eqref{eqn:reg1_partial}, \eqref{eqn:reg2_partial}, \Cref{eq:reg3}, with probability at least $1-2(e/2)^{-N}-N^{-\theta}$, there exists a constant $C_{\tau}$ depending only on $\tau, M, K$ such that 
$
    \|(\rmA-\E \rmA)_{\sI}\|\leq \const_{\tau} \sqrt{d},
$
where $\const_{\tau}=2( (5\LM + 1)(\ell - 1)+\alpha_0\sqrt{\tau})$. This finishes the proof of \Cref{thm:regularization_concentration_partial}.

\section{Deferred proofs}

\subsection{Proofs of Lemmas in Section \ref{sec:partial_prelim}}
\begin{proof}[Proof of \Cref{lem:general_eigen_value_approximation}]
By Weyl's inequality (\Cref{lem:weyl}), the difference between eigenvalues of $\widetilde{\E \rmA}$ and $\E \rmA$ can be upper bounded by
\begin{subequations}
\begin{align*}
    |\lambda_{j}( \widetilde{\E \rmA} ) - \lambda_{j}(\E \rmA)| & \leq \|\widetilde{\E \rmA} - \E \rmA \|_2 \leq \|\widetilde{\E \rmA} - \E \rmA \|_{\frob}\\
    &\leq \left[2K \cdot \frac{N}{K} \cdot \sqrt{N}\log(N) \cdot (\alpha - \beta)^2 \right]^{1/2}= O \Big(N^{3/4}\log^{1/2}(N)(\alpha - \beta)\Big).
\end{align*}
\end{subequations}
The lemma follows, as $\lambda_{j}(\E \rmA) = \Omega \left ( N (\alpha - \beta) \right )$ for all $1 \leq j \leq K$. 
\end{proof}

\subsection{Proofs of Lemmas in \Cref{subsec:spectral_partition_proof}}
\begin{proof}[Proof of \Cref{lem:singular_value_approximation}]
We first compute the singular values of $\overline{\rmB}_1$. From \eqref{eqn:bar_B1}, the rank of matrix $\overline{\rmB}_1$ is $K$, and the least non-trivial singular value of $\overline{\rmB}_1$ is
\begin{align*}
    \sigma_{K} (\overline{\rmB}_1) = \frac{N}{2\sqrt{2}K}(\overline{\alpha} - \overline{\beta}) = \frac{N}{2\sqrt{2}K}\sum_{\ell \in \sL} \binom{\frac{3N}{4K} - 2}{\ell - 2} \frac{a_{\ell} - b_{\ell}}{ \binom{N - 1}{\ell-1}}\,,
\end{align*}
where $\sL$ is obtained from \Cref{alg:parameter_preprocessing}. By the definition of $\overline{\rmA}_1$ in \eqref{eqn:bar_A1A2}, the least non-trivial singular value  of $\overline{\rmA}_1$ is
\begin{align*}
    \sigma_{K} (\overline{\rmA}_1) = \sigma_{K} (\overline{\rmB}_1) = \frac{N}{2\sqrt{2}K}(\overline{\alpha} - \overline{\beta}) = \frac{N}{2\sqrt{2}K}\sum_{\ell \in \sL} \binom{\frac{3N}{4K} - 2}{\ell - 2} \frac{a_{\ell} - b_{\ell}}{ \binom{N - 1}{\ell-1}}\,.
\end{align*}
Recall that $N_i$, defined in \eqref{eqn:dimension_ZcapVi}, denotes the number of vertices in $\gZ\cap \gV_i$, which can be written as  $N_i = \sum_{v\in \gV_i} \indi{v \in Z}$. The expectation of $N_i$ is $\E N_i = N/(2K)$. By Hoeffding's \Cref{lem:Hoeffding}, 
\begin{align*}
    \P\left( \bigg|N_i - \frac{N}{2K}\bigg| \geq \sqrt{N}\log(N) \right) \leq 2\exp \left( - K\log^2 (N)\right)\,. 
\end{align*}
Similarly, $\widetilde{N}_{i}$, defined in \eqref{eqn:dimension_Y1capVi}, satisfies
\begin{align}
    \P\left( \bigg|\widetilde{N}_{i} - \frac{N}{4K}\bigg| \geq \sqrt{N}\log(N) \right) \leq 2\exp \left( - K\log^2 (N)\right)\,. \notag
\end{align}
As defined in \eqref{eqn:tilde_B1} and \eqref{eqn:bar_B1}, both $\widetilde{\rmB}_1$ and $\overline{\rmB}_1$ are deterministic block matrices. Then with probability at least $1 - 2K\exp \left( - K\log^2 (N)\right)$, the dimensions of each block inside $\widetilde{\rmB}_1$ and $\overline{\rmB}_1$ are approximately the same, with deviations up to $\sqrt{N}\log(N)$.
Consequently, the matrix $\widetilde{\rmA}_1$, which was defined in \eqref{eqn:tilde_A1A2}, can be treated as a perturbed version of $\overline{\rmA}_1$.  By Weyl's inequality (\Cref{lem:weyl}), for any $i\in [K]$,
\begin{subequations}
\begin{align*}
    |\sigma_i(\overline{\rmB}_1) - \sigma_i(\widetilde{\rmB}_1)| &=  |\sigma_i(\overline{\rmA}_1) - \sigma_i(\widetilde{\rmA}_1)|\leq \|\overline{\rmA}_1 - \widetilde{\rmA}_1\|_2 \leq \|\overline{\rmA}_1 - \widetilde{\rmA}_1\|_{\frob}\\
    &\leq \left[2K \cdot \frac{N}{K} \cdot \sqrt{N}\log(N) \cdot (\overline{\alpha} - \overline{\beta})^2 \right]^{1/2}= O \left(N^{3/4}\log^{1/2}(N) \cdot (\overline{\alpha} - \overline{\beta})\right).
\end{align*}
\end{subequations}
As a result, with probability at least $1 - 2K\exp( -K \log^2(N))$, we have
\begin{align*}
    \frac{|\sigma_{K}(\overline{\rmA}_1) - \sigma_{K}(\widetilde{\rmA}_1)|}{\sigma_{K}(\overline{\rmA}_1)} = \frac{|\sigma_{K}(\overline{\rmB}_1) - \sigma_{K}(\widetilde{\rmB}_1)|}{\sigma_{K}(\overline{\rmB}_1)} = O\left(N^{-1/4}\log^{1/2}(N)\right).
\end{align*}
Therefore, singular values of $\widetilde{\rmA}_1$ can be approximated by singular values of $\overline{\rmA}_1$ with vanishing errors.
\end{proof}

\begin{proof}[Proof of \Cref{lem:high_degree_vertices_k}]

Let $\gX\subset \gV$ be a subset of vertices in hypergraph $\gH = (\gV, \gE)$ with size $|\gX| = cN$ for some $c\in (0,1)$ to be decided later.
Suppose $\gX$ is a set of vertices with high degrees that we want to zero out. 
 We first count the $\ell$-uniform hyperedges on $\gX$ separately, then weight them by $(\ell - 1)$, and finally sum over $\ell$ to compute the row sums in $\rmA$  corresponding to each vertex in $\gX$.  Let $\gE_{\ell}(\gX)$ denote the set of $\ell$-uniform hyperedges with all vertices located in $\gX$, and $\gE_{\ell} (\gX^{\complement})$ denote the set of $\ell$-uniform hyperedges with all vertices in $\gX^{\complement}=  \gV\setminus X$, respectively. Let $\gE_{\ell}(\gX, \gX^{\complement})$ denote the set of $\ell$-uniform hyperedges with at least $1$ endpoint in $\gX$ and $1$ endpoint in $\gX^{\complement}$. 
The relationship between total row sums and the number of non-uniform hyperedges in the vertex set $\gX$ can be expressed as
    \begin{align}\label{eqn:non_uniform_hyperedges_in_X}
        \sum_{v \in \gX} \mathrm{row}(v) \leq &\, \sum_{\ell \in \sL} (\ell - 1)\Big( \ell |\gE_{\ell}(\gX)| + (\ell - 1)|\gE_{\ell}(\gX, \gX^{\complement})| \Big) 
    \end{align} 
If the row sum of each vertex $v \in \gX$ is at least $20 \LM d$, where $d \coloneqq \sum_{\ell \in \sL} (\ell - 1)a_{\ell}$, it follows
\begin{equation}\label{eq:lowerboundEM}
\sum_{\ell \in \sL} (\ell - 1)\Big(\ell |\gE_{\ell}(\gX)| + (\ell - 1)|\gE_{\ell}(\gX, \gX^{\complement})| \Big) \geq cN\cdot(20 \LM d)\,.
\end{equation}
Then either 
\begin{align*}
    \sum_{\ell \in \sL} \ell(\ell - 1) |\gE_{\ell}(\gX)| \geq 4 \LM cN d, \quad \text{or} \quad  \sum_{\ell \in \sL} (\ell - 1)^2|\gE_{\ell}(\gX, \gX^{\complement})| \geq 16 \LM cN d.  \notag 
\end{align*}

\paragraph{Concentration of $\sum_{\ell \in \sL} \ell(\ell - 1)|\gE_{\ell}(\gX)|$} 

Recall that $\big|\gE_{\ell}(\gX)\big|$ denotes the number of $\ell$-uniform hyperedges with all vertices located in $\gX$, which can be viewed as the sum of independent Bernoulli random variables $\rT_{e}^{(a_{\ell})}$ and $\rT_{e}^{(b_{\ell})}$ given by
    \begin{equation}\label{eqn:EmX_Bernoulli_partial}
        \rT_{e}^{(a_{\ell})}\sim \mathrm{Bernoulli}\left( \frac{a_{\ell}}{ \binom{N - 1}{\ell -1} }\right), \quad \rT_{e}^{(b_{\ell})}\sim \mathrm{Bernoulli}\left( \frac{b_{\ell}}{ \binom{N - 1}{\ell -1} }\right)\,.
    \end{equation}
Let \{$\gV_{1}, \dots, \gV_{K}$\} be the true partition of $\gV$. Suppose that there are $\eta_{i} cN$ vertices in block $\gV_{i}\cap \gX$ for each $i\in [K]$ with restriction $\sum_{i=1}^{K}\eta_{i} = 1$, then  $\big|\gE_{\ell}(\gX)\big|$ can be written as
\begin{align}
   \big|\gE_{\ell}(\gX)\big| = \sum_{e\in \gE_{\ell}(\gX, a_{\ell})} \rT_{e}^{(a_{\ell})} \, + \sum_{e\in \gE_{\ell}(\gX, b_{\ell})} \rT_{e}^{(b_{\ell})}\,, \notag 
\end{align}
where $\gE_{\ell}(\gX, a_{\ell}) \coloneqq \cup_{i=1}^{K}\gE_{\ell}(\gV_i\cap \gX)$ denotes the union for sets of hyperedges with all vertices in the same block $\gV_i\cap \gX$ for some $i\in [K]$, and
\begin{align}
    \gE_{\ell}(\gX, b_{\ell}) \coloneqq \, \gE_{\ell}(\gX) \setminus \gE_{\ell}(\gX, a_{\ell}) = \gE_{\ell}(\gX) \setminus \Big( \cup_i^{K} \gE_{\ell}(\gV_i\cap \gX) \Big) \notag 
\end{align}
denotes the set of hyperedges with vertices crossing different $\gV_i\cap \gX$. 
We can compute the expectation of $\big|\gE_{\ell}(\gX)\big|$ as 
    \begin{equation}\label{eqn:E_gX_partial}
        \E |\gE_{\ell}(\gX)| = \sum_{i=1}^{K} \binom{\eta_{i} cN}{\ell} \frac{a_{\ell} - b_{\ell} }{ \binom{N - 1}{\ell -1} } + \binom{cN}{\ell} \frac{b_{\ell}}{ \binom{N - 1}{\ell -1} }.
    \end{equation}
Then
\begin{align}\label{eqn:expectation_EX_partial}
    \sum_{\ell \in \sL} \ell(\ell - 1) \cdot \E |\gE_{\ell}(\gX)| & = \sum_{\ell \in \sL} \ell(\ell - 1) \Bigg[ \sum_{i=1}^{K} \binom{\eta_{i} cN}{\ell} \frac{a_{\ell} - b_{\ell} }{ \binom{N - 1}{\ell -1} } + \binom{cN}{\ell} \frac{b_{\ell}}{ \binom{N - 1}{\ell -1} } \Bigg].
\end{align}
As $\sum\limits_{i=1}^{K} \eta_i=1$, it follows that 
$
\sum_{i=1}^{K} \binom{\eta_i c N}{\ell} \leq \binom{cN}{\ell}
$ by induction, thus
\[
 \frac{a_{\ell} - b_{\ell} }{ \binom{N - 1}{\ell -1} } \sum_{i=1}^{K} \binom{\eta_{i} cN}{\ell}  +  \frac{b_{\ell}}{ \binom{N - 1}{\ell -1} } \binom{cN}{\ell} =  \frac{a_{\ell}}{ \binom{N - 1}{\ell -1} } \sum_{i=1}^{K} \binom{\eta_{i} cN}{\ell}  + \frac{b_{\ell}}{ \binom{N - 1}{\ell -1} } \left( \binom{cN}{\ell} - \sum_{i=1}^{K} \binom{\eta_{i} cN}{\ell} \right)  
\]
where both terms on the right are  positive numbers. Using this and taking $b_{\ell} = a_{\ell}$, we obtain the following upper bound for all $N$,
\begin{align*}
 \sum_{\ell \in \sL} \ell(\ell - 1) \E |\gE_{\ell}(\gX)|\leq \sum_{\ell \in \sL} \ell(\ell - 1)\binom{cN}{\ell} \frac{a_{\ell}}{\binom{N - 1}{\ell-1}} \leq cN\sum_{\ell \in \sL} (\ell - 1)a_{\ell} =cNd \,.
\end{align*}
Note that $\sum_{\ell \in \sL} \ell(\ell - 1) |\gE_{\ell}(\gX)|$ is a weighted sum of independent Bernoulli random variables (corresponding to hyperedges), each upper bounded by $\LM^2$. Also, its variance is bounded by
\begin{align}
    \sigma^2\coloneqq \,& \Var\left(\sum_{\ell \in \sL} \ell(\ell - 1) |\gE_{\ell}(\gX)|\right)=\sum_{\ell \in \sL} \LM^2(\ell - 1)^2 \Var\left( |\gE_{\ell}(\gX)|\right) \notag \\
    \leq\,& \sum_{\ell \in \sL} \LM^2(\ell - 1)^2\E |\gE_{\ell}(\gX)|\leq \LM^2cNd. \notag 
\end{align}
We can apply Bernstein's \Cref{lem:Bernstein} and obtain
\begin{align}
  \,& \P\left( \sum_{\ell \in \sL} \ell(\ell - 1) |\gE_{\ell}(\gX)| \geq 4 \LM cN d \right)\notag \\\leq \,& \P\left( \sum_{\ell \in \sL} \ell(\ell - 1)( |\gE_{\ell}(\gX)| -\E|\gE_{\ell}(\gX)| )\geq 3 \LM cN d \right)\notag\\
\leq \,& \exp \left(-\frac{(3 \LM cN d)^2}{\LM^2cNd + \LM^2cNd/3} \right) \leq \exp(-6cNd)\,.\label{eqn:chernoff_result_partial}
\end{align}

\paragraph{Concentration of  $\sum_{\ell \in \sL} (\ell - 1)^2|\gE_{\ell}(\gX, \gX^{\complement})|$}

For any finite set $\setS$, let $[\setS]^{j}$ denote the family of $j$-subsets of $S$, i.e., $[\setS]^{j} = \{\gZ| \gZ\subseteq \setS, |\gZ| = j \}$. Let $\gE_{\ell}([\gY]^{j}, [\gZ]^{\ell - j})$ denote the set of $\ell$-hyperedges, where $j$ vertices are from $\gY$ and $\ell - j$ vertices are from $\gZ$ within each $\ell$-hyperedge. We want to count the number of $\ell$-hyperedges between $\gX$ and $\gX^{\complement}$, according to the number of vertices located in $\gX^{\complement}$ within each $\ell$-hyperedge. Suppose that there are $j$ vertices from $\gX^{\complement}$ within each $\ell$-hyperedge for some $1\leq j \leq \ell - 1$.
\begin{enumerate}[label=(\roman*)]
     \item Assume that all those $j$ vertices are in the same $[\gV_{i}\setminus \gX]^{j}$. If the remaining $\ell - j$ vertices are from $[\gV_{i}\cap \gX]^{\ell - j}$, then this $\ell$-hyperedge is connected with probability $a_{\ell} / \binom{N - 1}{\ell-1}$, otherwise $b_{\ell} / \binom{N - 1}{\ell-1}$. The number of this type $\ell$-hyperedges can be written as
    \begin{align}
        \sum_{i=1}^{K} \left[ \sum_{e\in \gE^{(a_{\ell})}_{j, i}} \rT_{e}^{(a_{\ell})} + \sum_{e\in \gE^{(b_{\ell})}_{j, i}} \rT_{e}^{(b_{\ell})} \right]\,, \notag 
    \end{align} 
    where $\gE^{(a_{\ell})}_{j, i}\coloneqq \gE_{\ell}([\gV_i \cap \gX^{\complement}]^{j}, [\gV_i\cap \gX]^{\ell - j})$, and
        \begin{align}
            \gE^{(b_{\ell})}_{j, i} \coloneqq \gE_{\ell} \Big([\gV_i\cap \gX]^{\complement}]^{j}, \,\, [\gX]^{\ell - j} \setminus [\gV_i\cap \gX]^{\ell - j} \Big) \notag 
        \end{align}
    denotes the set $\ell$-hyperedges with $j$ vertices in $[\gV_i\cap \gX^{\complement}]^{j}$ and the remaining $\ell - j$ vertices  in $[\gX]^{j}\setminus [\gV_i\cap \gX]^{j}$. We compute all possible choices and upper bound the cardinality of $\gE^{(a_{\ell})}_{j, i}$ and $\gE^{(b_{\ell})}_{j, i}$ by
        \begin{align}
            \big|\gE^{(a_{\ell})}_{j, i} \big| \leq &\, \binom{ (\frac{1}{K} - \eta_{i}c)N}{j} \binom{\eta_{i} cN}{\ell - j}\,,\quad \big|\gE^{(b_{\ell})}_{j, i} \big| \leq \binom{ (\frac{1}{K} - \eta_{i}c)N}{j} \left[ \binom{cN}{\ell - j} - \binom{\eta_{i} cN}{\ell - j} \right]\,. \notag 
        \end{align}
        
    \item If those $j$ vertices in $[\gV\setminus \gX]^{j}$ are not in the same $[\gV_i \cap \gX]^{j}$ (which only happens $j \geq 2$), then the number of this type hyperedges can be written as $\sum_{e\in \gE^{(b_{\ell})}_{j}} \rT_{e}^{(b_{\ell})}$, where 
    \begin{subequations}
        \begin{align*}
            \gE^{(b_{\ell})}_{j} \coloneqq&\, \gE_{\ell} \Big([V\setminus \gX]^{j}\setminus \big( \cup_{i=1}^{K}[\gV_i\setminus \gX]^{j} \big), \,\,\, [\gX]^{\ell - j}\Big)\,,\\
            \big| \gE^{(b_{\ell})}_{j} \big| \leq &\, \left[ \binom{(1-c)N}{j} - \sum_{i=1}^{K} \binom{( \frac{1}{K} - \eta_{i}c)N}{j} \right] \binom{cN}{\ell - j}\,.
        \end{align*}
    \end{subequations}
\end{enumerate}
Therefore, $|\gE_{\ell}(\gX, \gX^{\complement})|$ can be written as a sum of independent Bernoulli random variables,
\begin{align}\label{eq:EXmcBernoulli}
    |\gE_{\ell}(\gX, \gX^{\complement})| = \sum_{j=1}^{\ell - 1}\sum_{i=1}^{K} \left[ \sum_{e\in \gE^{(a_{\ell})}_{j, i}} \rT_{e}^{(a_{\ell})} + \sum_{e\in \gE^{(b_{\ell})}_{j, i}} \rT_{e}^{(b_{\ell})} \right] + \sum_{j=2}^{\ell - 1}\sum_{e\in \gE^{(b_{\ell})}_{j}} \rT_{e}^{(b_{\ell})}\,.
\end{align}
Note that $\binom{(1-c)N}{1} = \sum_{i=1}^{K}\binom{(1/k - \eta_i c)N}{1}$ and $\big| \gE^{(b_{\ell})}_{1} \big| = 0$. Then $\E (|\gE_{\ell}(\gX, \gX^{\complement})|)$ can be computed as follows:
\begin{align}
    &\, \E \left(|\gE_{\ell}(\gX, \gX^{\complement})| \right) \label{eqn:regrouping_EXXc}\\
    =&\, \sum_{j=1}^{\ell - 1} \sum_{i=1}^{K} \binom{( \frac{1}{K} - \eta_{i}c)N}{j} \Bigg\{ \binom{\eta_{i} cN}{\ell - j} \frac{a_{\ell}}{ \binom{N - 1}{\ell-1} } + \bigg[ \binom{cN}{\ell - j} - \binom{\eta_{i} cN}{\ell - j} \bigg] \frac{b_{\ell}}{ \binom{N - 1}{\ell-1} } \Bigg\} \notag \\
    \,& + \sum_{j=1}^{\ell - 1} \Bigg[ {(1 - c)N \choose j} - \sum_{i=1}^{K}{( \frac{1}{K} - \eta_{i}c)N \choose j} \Bigg] {cN \choose \ell - j} \frac{b_{\ell}}{ \binom{N - 1}{\ell-1} } \notag \\ 
    =\,& \sum_{j=1}^{\ell - 1} \sum_{i=1}^{K}\binom{( \frac{1}{K} - \eta_{i}c)N}{j} \binom{\eta_{i} cN}{\ell - j} \frac{a_{\ell} - b_{\ell}}{ \binom{N - 1}{\ell-1} } + \sum_{j=1}^{\ell - 1} \binom{(1-c)N}{j} \binom{cN}{\ell - j} \frac{b_{\ell}}{ \binom{N - 1}{\ell-1} }\notag\\
    =&\, \sum_{i=1}^{K} \Bigg[ \binom{\frac{N}{K}}{\ell} -  \binom{\eta_{i} cN}{\ell} - \binom{(\frac{1}{K} - \eta_{i}c)N}{\ell} \Bigg]\frac{a_{\ell} - b_{\ell}}{\binom{N - 1}{\ell -1} } + \Bigg[ \binom{N}{\ell} - \binom{cN}{\ell} - \binom{(1 - c)N}{\ell} \Bigg] \frac{b_{\ell}}{\binom{N - 1}{\ell -1} }\notag\,,
\end{align}
where we used the fact $\binom{(1-c)N}{1} = \sum_{i=1}^{K}\binom{(1/k - \eta_i c)N}{1}$ in the first equality and Vandermonde's identity $\binom{N_{1} + N_{2}}{\ell} = \sum_{j=0}^{\ell}\binom{N_{1}}{j} \binom{N_{2}}{\ell - j}$ in last equality. Note that 
\[
f_c \coloneqq \binom{N}{\ell} - \binom{cN}{\ell} - \binom{(1 - c)N}{\ell}
\]
counts the number of subsets of $\gV$ with $\ell$ elements such that at least one element belongs to $\gX$ and at least one element belongs to $\gX^{\complement}$. On the other hand,
\[
g_c = \sum_{i=1}^{K} \left[ \binom{\frac{N}{K}}{\ell} -  \binom{\eta_{i} cN}{\ell} - \binom{(\frac{1}{K} - \eta_{i}c)N}{\ell} \right]\,.
\]
counts the number of subsets of $\gV$ with $\ell$ elements such that all elements belong to a single $\gV_i$, and given such an $i$, that at least one element belongs to $\gX \cap \gV_i$ and at least one belongs to $\gX^{\complement} \cap \gV_i$. 
\begin{figure}[ht]
    \centering
    \includegraphics[width= 0.4\linewidth]{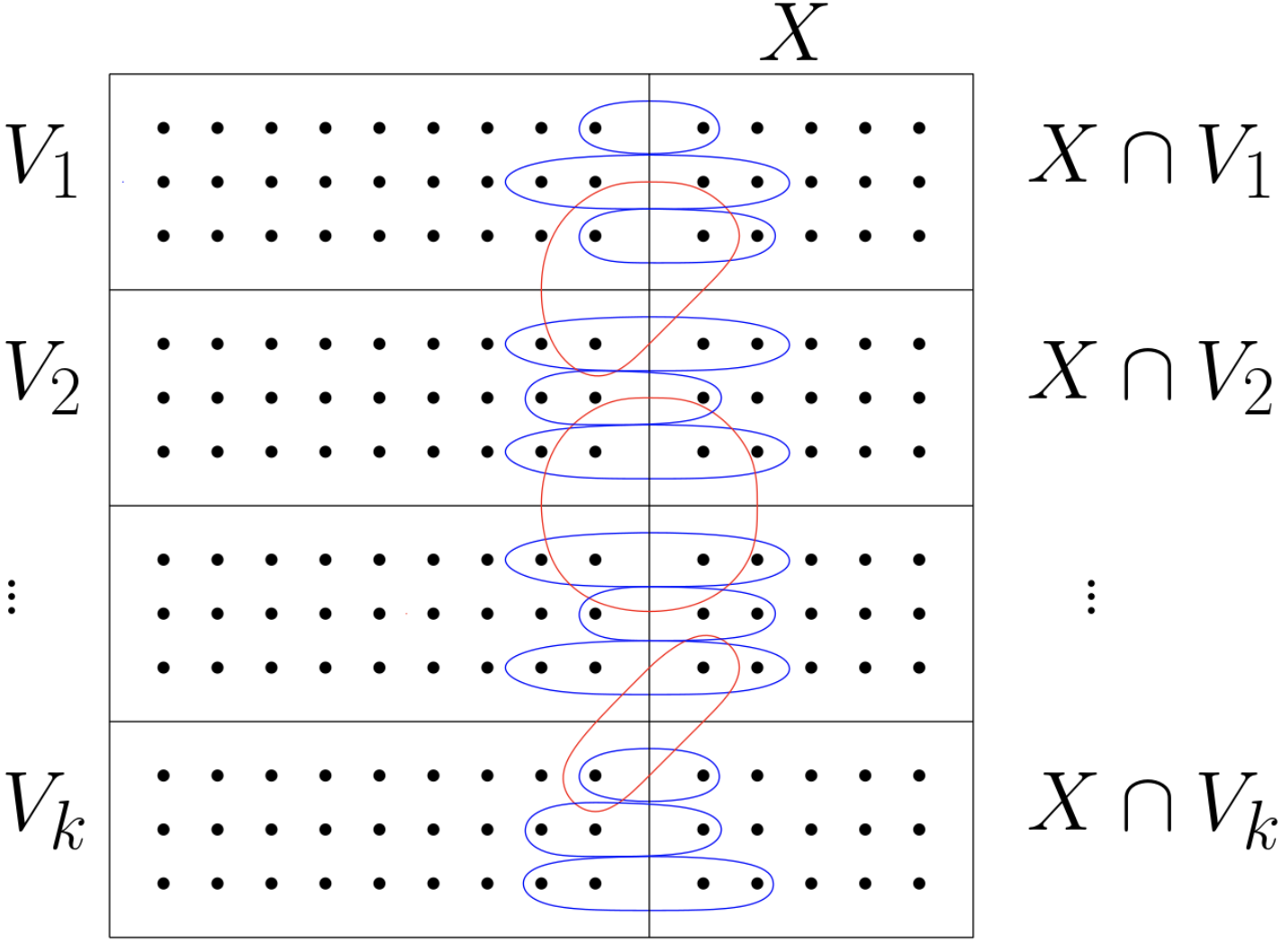}
    \caption{Comparison of $f_c$ and $g_c$}
    \label{fig:comparison_fcgc}
\end{figure}
As in \Cref{fig:comparison_fcgc}, $g_{c}$ only counts the blue pairs while $f_c$ counts red pairs in addition. By virtue of the fact that there are fewer conditions imposed on the sets included in the count for $f_c$, we must have $f_c \geq g_c$. Thus, rewriting \eqref{eqn:regrouping_EXXc}, we obtain
\begin{align*}
\E (|\gE_{\ell}(\gX, \gX^{\complement})|) &\,= g_c \frac{a_{\ell}}{\binom{N - 1}{\ell -1} } + (f_c-g_c) \frac{b_{\ell}}{\binom{N - 1}{\ell -1} }\,~.
\end{align*}
Since both terms in the above sum are positive, we can upper bound by taking $a_{\ell}=b_{\ell}$ to obtain
\begin{align}
    \E \left( |\gE_{\ell}(\gX, \gX^{\complement})| \right) \leq f_c  \frac{a_{\ell}}{\binom{N - 1}{\ell -1} } = \Bigg[ \binom{N}{\ell} - \binom{cN}{\ell} - \binom{(1 - c)N}{\ell} \Bigg] \frac{a_{\ell}}{\binom{N - 1}{\ell -1} } ~. \notag 
\end{align}
By summing over $\ell$, the expectation of $\sum_{\ell \in \sL}(\ell - 1)^2|\gE_{\ell}(\gX, \gX^{\complement})|$ satisfies
\begin{subequations}
\begin{align}
  \sum_{\ell \in \sL} (\ell - 1)^2\cdot \E \left(|\gE_{\ell}(\gX, \gX^{\complement})| \right)&\leq  \sum_{\ell \in \sL} (\ell - 1)^2\Bigg[ \binom{N}{\ell} - \binom{cN}{\ell} - \binom{(1 - c)N}{\ell} \Bigg] \frac{a_{\ell}}{\binom{N - 1}{\ell -1} }~, \notag \\
  &\leq 2N \sum_{\ell \in \sL}(1-c^{\ell}-(1-c)^{\ell})(\ell - 1)a_{\ell}\leq 8 \LM cN d\,, \notag 
\end{align}
\end{subequations}
where the last line holds when $c\in (0,  2^{1/\LM} - 1]$, since
\begin{align}
    &\, [(1 - c) + c]^{\ell} - c^{\ell} -(1-c)^{\ell} \notag\\
    = &\, \binom{\ell}{1}(1 - c)^{\ell - 1}c + \binom{\ell}{2}(1 - c)^{\ell - 2}c^2 + \cdots + \binom{\ell}{\ell - 1}(1 - c)^{1}c^{\ell - 1}\notag\\
    \leq &\, \binom{\ell}{1}c + \binom{\ell}{2}c^2 + \cdots + \binom{\ell}{\ell - 1}c^{\ell - 1}\leq (1 + c)^{\ell} - 1 \leq 2\ell c\,, \label{eqn:czero}
\end{align}
where the last inequality holds by the following Claim.
\begin{claim}
    Let $\ell \geq 2$ be some finite integer. Then for $0 < c <= 2^{1/\ell} - 1$, it follows that $(1 + c)^{\ell} - 1 \leq 2\ell c$.
\end{claim}
\begin{proof}[Proof of the Claim]
    We finish the proof by induction. First, the argument  $(1 + c)^{j} - 1 \leq 2jc$ holds true for the base cases $j = 1, 2$ . Suppose that the argument holds for the case $j \geq 2$. For the case $j + 1 \leq \ell$, it follows that
    \begin{align*}
        (1 + c)^{j + 1} - 1 = (1 + c)^{j} + c(1 + c)^{j} - 1 \leq 2jc + c(1 + c)^{j} \leq 2(j + 1)c,
    \end{align*}
    where the last inequality holds true if $c(1 + c)^{j} \leq 2c$, and it holds since $c \leq 2^{1/\ell} - 1 \leq 2^{1/j} - 1$ for all $j\leq \ell$. 
\end{proof}

Similarly, we apply Bernstein \Cref{lem:Bernstein} again with $K=\LM^2$, $\sigma^2\leq 8\LM^{3} cNd$ and obtain
\begin{align}
  & \, \P\left( \sum_{\ell \in \sL} (\ell - 1)^2 |\gE_{\ell}(\gX,\gX^{\complement})| \geq 16 \LM cN d \right) \notag\\
  \leq \,&\, \P\left( \sum_{\ell \in \sL} (\ell - 1)^2( |\gE_{\ell}(\gX,\gX^{\complement})| -\E|\gE_{\ell}(\gX,\gX^{\complement})|)\geq 8 \LM cN d \right) \leq \exp(-6cNd/\LM)\,. \label{eqn:chernoff_second_partial}
\end{align}
By the binomial coefficient upper bound $\binom{N}{K}\leq (\frac{eN}{K})^{K}$ for $1\leq K \leq N$, there are at most 
    \begin{align}\label{eqn:binomial_upper_bound_partial}
        \binom{N}{cN}\leq \left(\frac{e}{c}\right)^{cN}= \exp(-c(\log(c)-1)N)
    \end{align}
    many subsets $\gX$ of size $|\gX| = cN$. Let  $d$ be sufficiently large so that $d^{-3} \leq c_0$. Substituting $c = d^{-3}$ in \eqref{eqn:binomial_upper_bound_partial}, we have
        \begin{equation}
            \binom{N}{d^{-3}N} \leq \exp \left[ 3d^{-3} \log(d)N \right]. \notag 
        \end{equation}
Taking $c=d^{-3}$ in \eqref{eqn:chernoff_result_partial} and \eqref{eqn:chernoff_second_partial}, we obtain
\begin{align}
     & \P\left( \sum_{\ell \in \sL} (\ell - 1) (\ell|\gE_{\ell}(\gX)| + (\ell - 1)|\gE_{\ell}(\gX,\gX^{\complement})|)\geq 20 \LM d^{-2} N \right)\leq 2\exp(-2d^{-2}N/\LM)\,. \notag 
\end{align}
Taking a union bound over all possible $\gX$ with $|\gX|=d^{-3}N$, we obtain with probability at least $1-2\exp(3d^{-3}\log(d) N - 2d^{-2}N/\LM)\leq 1-2\exp(-d^{-2}N/\LM)$, no more than $d^{-3}N$ many vertices have total row sum greater than $20 \LM d$. Note that we have imposed the condition that $c = d^{-3} \in (0,  2^{1/\LM} - 1]$ in \eqref{eqn:czero}, thus $d\geq (2^{1/\LM} - 1)^{-1/3}$, producing the lower bound in \Cref{ass:regimes_partial}. 
\end{proof} 

\begin{proof}[Proof of \Cref{lem:projection_deviation_multiple}]
Throughout the proof, let $\pi(i)$ denote the membership of vertex $i \in \gV$. Let $\rvu_i(j)$ denote the $j$-th entry of the vector $\rvu_i \in \R^{N}$. Let $\<\rvu_{i}, \rve_{i}\>$ denote the inner product of two vectors.

Note that $\rmU$ is spanned by first $K$ singular vectors of $(\rmA_1)_{\sI_1}$. Let $\{\rvu_i\}_{i=1}^{K}$ be an orthonormal basis of the $\rmU$. For each fixed $i\in \gV_{\pi(i)} \cap \gY_2\cap \{i_1, \cdots, i_s\}$, the projection $\rmP_{\rmU} \coloneqq \sum_{k=1}^{K}\< \rvu_{k},\cdot\, \> \rvu_{k}$ is defined via
\begin{align*}
    \rmP_{\rmU}\rve_i =&\, \sum_{k=1}^{K}\< \rvu_{k}, \rve_i\> \rvu_{k}, \quad 
    \|\rmP_{\rmU}\rve_i\|_2^2 =\, \sum_{k=1}^{K} \< \rvu_{k}, \rve_i\>^2.
\end{align*}
As a consequence of independence between entries in $\rmA_1$ and entries in $\rmA_2$, defined in \eqref{eqn:A1A2}, it is known that $\{\rvu_{k}\}_{k=1}^{K}$ and $\rve_i$ are independent of each other, since $\rve_i$ are columns of $\rmE_2 \coloneqq \rmA_2 - \widetilde{\rmA}_2$. If the expectation is taken over $\{\tA^{(\ell)}\}_{\ell \in \sL}$ conditioning on $\{\rvu_{k}\}_{k=1}^{K}$, then
    \begin{align*}
        \E_{\{\tA^{(\ell)}\}_{\ell \in \sL}} \left[ \< \rvu_{k}, \rve_i\> \Big\vert \{\rvu_{k}\}_{k=1}^{K} \right]=&\, \sum_{j=1}^{N}\rvu_{k}(j) \cdot \E \left( \Big[ (\rmA_2)_{ji} - (\E\rmA_2)_{ji} \Big] \right) = 0\,,\\
        \E_{\{\tA^{(\ell)}\}_{\ell \in \sL}} \left[ \|\rmP_{\rmU}\rve_i\|_2^2 \Big\vert \{\rvu_{k}\}_{k=1}^{K} \right] =&\, \sum_{k=1}^{K} \E_{\{\tA^{(\ell)}\}_{\ell \in \sL}} \left[ \< \rvu_{k}, \rve_i\>^2 \Big\vert \{\rvu_{k}\}_{k=1}^{K} \right]\,,
    \end{align*}
where $\sL$ is obtained from \Cref{alg:parameter_preprocessing}. We expand each $\< \rvu_{k}, \rve_i\>^2$ and rewrite it into $2$ parts,
    \begin{align}
        \< \rvu_{k}, \rve_i\>^2 =&\, \sum_{j_1 = 1}^{K} \sum_{j_2 = 1}^{K} \rvu_{k}(j_1)\rve_i(j_1)\rvu_{k}(j_2)\rve_i(j_2) \notag \\
    =&\, \underbrace{\sum_{j = 1}^{K}[\rvu_{k}(j)]^2[\rve_i(j)]^2}_{(a)} + \underbrace{\sum_{j_1 \neq j_2} \rvu_{k}(j_1)\rve_i(j_1)\rvu_{k}(j_2)\rve_i(j_2)}_{(b)}\,,\quad \forall k\in [K]\,.\label{eqn:contribution_part_a_b}
    \end{align}
    Part $(a)$ is the contribution from graph, i.e., $2$-uniform hypergraph, while part $(b)$ is the contribution from $\ell$-uniform hypergraph with $\ell \geq 3$, which only occurs in hypergraph clustering. The expectation of part $(a)$ in \eqref{eqn:contribution_part_a_b} in is upper bounded by $\alpha$ as defined in \eqref{eqn:alpha_beta}, since for any $k\in [K]$,
    \begin{align*}
        &\E_{\{\tA^{(\ell)}\}_{\ell \in \sL}} \left[ \sum_{j = 1}^{K}[\rvu_{k}(j)]^2[\rve_i(j)]^2 \Bigg\vert \{\rvu_{k}\}_{k=1}^{K} \right] = \sum_{j=1}^{N}[\rvu_{k}(j)]^2 \cdot \Var\Big((\rmA_{2})_{ji} \Big) \\
       \leq\,& \sum_{j=1}^{N}[\rvu_{k}(j)]^2 \cdot (\E\rmA_{2})_{ji} \leq \alpha = \sum_{\ell \in \sL}\Bigg[ \binom{\frac{N}{K} -2}{\ell - 2} \frac{a_{\ell} - b_{\ell}}{\binom{N - 1}{\ell-1} } + \binom{N}{\ell - 2} \frac{b_{\ell}}{\binom{N - 1}{\ell-1} }\Bigg]\\
       \leq\,& \sum_{\ell \in \sL} \binom{N}{\ell - 2} \frac{a_{\ell}}{\binom{N - 1}{\ell-1} } \leq \frac{2}{N}\sum_{\ell \in \sL} (\ell - 1)a_{\ell} = \frac{2d}{N},
    \end{align*}
where $\|\rvu_{k}\|_2^2 = \sum_{j=1}^{N}[\rvu_{k}(j)]^2 = 1$. For part $(b)$ in \eqref{eqn:contribution_part_a_b},  
\begin{subequations}
    \begin{align*}
        &\, \E_{\{\tA^{(\ell)}\}_{\ell \in \sL}} \left[ \sum_{j_1 \neq j_2} \rvu_{k}(j_1)\rve_i(j_1)\rvu_{k}(j_2)\rve_i(j_2) \Bigg\vert \{\rvu_{k}\}_{k=1}^{K} \right]\\
        =&\, \sum_{j_1 \neq j_2} \rvu_{k}(j_1)\rvu_{k}(j_2) \E \left[ \Big( (\rmA_2)_{j_{1}i} -  (\E\rmA_2)_{j_{1}i} \Big) \Big( (\rmA_2)_{j_{2}i} - (\E\rmA_2)_{j_{2}i} \Big) \right]\\
        =&\, \sum_{j_1 \neq j_2} \rvu_{k}(j_1)\rvu_{k}(j_2) \E \Bigg(\sum_{\ell \in \sL}\,\,\sum_{\substack{e\in \gE_{\ell}[\gY_2\cup \gZ]\\ \{i,j_1\}\subset e} }\, (\etA_{e}^{(\ell)} - \E\etA_{e}^{(\ell)} ) \Bigg) \Bigg(\sum_{\ell \in \sL}\,\,\sum_{\substack{e\in \gE_{\ell}[\gY_2\cup \gZ]\\ \{i,j_2\}\subset e} }\, (\etA_{e}^{(\ell)} - \E\etA_{e}^{(\ell)} ) \Bigg).
    \end{align*}
\end{subequations}
Note that $\etA^{(\ell)}_{e_1}$ and $\etA^{(\ell)}_{e_2}$ are independent for distinct hyperedges $e_1 \neq e_2$, then only the terms with hyperedge $e\supset \{i, j_1, j_2\}$ have nonzero contribution. Then the expectation of part $(b)$ can be rewritten as
    \begin{align}
        &\, \E_{\{\tA^{(\ell)}\}_{\ell \in \sL}} \left[ \sum_{j_1 \neq j_2} \rvu_{k}(j_1)\rve_i(j_1)\rvu_{k}(j_2)\rve_i(j_2) \Bigg\vert \{\rvu_{k}\}_{k=1}^{K} \right]\notag \\
        = &\, \sum_{j_1 \neq j_2} \rvu_{k}(j_1)\rvu_{k}(j_2) \sum_{\ell \in \sL}\,\, \sum_{\substack{e\in \gE_{\ell}[\gY_2\cup \gZ]\\ \{i,j_1, j_2\}\subset e} } \E \big(\etA_{e}^{(\ell)} - \E\etA_{e}^{(\ell)} \big)^2 \notag\\
        \leq &\, \sum_{j_1 \neq j_2} \rvu_{k}(j_1)\rvu_{k}(j_2) \sum_{\ell \in \sL} \,\,\sum_{\substack{e\in \gE_{\ell}[\gY_2\cup \gZ]\\ \{i,j_1, j_2\}\subset e} } \E\etA_{e}^{(\ell)}\notag\\
        = &\sum_{j_1 \neq j_2} \rvu_{k}(j_1)\rvu_{k}(j_2) \sum_{\ell \in \sL}\sum_{\substack{e\in \gE_{\ell}[\gY_2\cup \gZ]\\ \{i,j_1, j_2\}\subset e} } \frac{a_{\ell}}{\binom{N - 1}{\ell-1}}\,.\label{eqn:expectation_Y2cupZ}
    \end{align}
Note that $|\gY_2 \cup \gZ| \leq N$, then the number of possible hyperedges $e$, while $e\in \gE_{\ell}[\gY_2\cup \gZ]$ and $e \supset \{i,j_1, j_2\}$, is at most $\binom{N}{\ell - 3}$. Thus \eqref{eqn:expectation_Y2cupZ} is upper bounded by
\begin{subequations}
    \begin{align*}
        &\, \sum_{j_1 \neq j_2} \rvu_{k}(j_1)\rvu_{k}(j_2) \sum_{\ell \in \sL} \binom{N}{\ell - 3} \frac{a_{\ell}}{\binom{N - 1}{\ell-1}} \\
        \leq &\,\sum_{j_1 \neq j_2} \rvu_{k}(j_1)\rvu_{k}(j_2) \sum_{\ell \in \sL} \frac{(\ell - 1)(\ell - 2)}{(N-\ell)^2}a_{\ell}
        \leq \frac{d \LM }{N^{2}} \sum_{j_1 \neq j_2} \rvu_{k}(j_1)\rvu_{k}(j_2) \,\\
        \leq &\, \frac{d \LM }{2N^{2}} \sum_{j_1 \neq j_2}\Big( [\rvu_{k}(j_1)]^2 + [\rvu_{k}(j_2)]^2 \Big)
        \leq \frac{d \LM (N-1) }{2N^{2}}\Bigg( \sum_{j_1 = 1}^{N} [\rvu_{k}(j_1)]^2 + \sum_{j_2 = 1}^{N}[\rvu_{k}(j_2)]^2 \Bigg)\\
        \leq & \frac{d \LM }{N}\,,
    \end{align*}
\end{subequations}
where $\|\rvu_{k}\|_2 = 1$, $d = \sum_{\ell \in \sL}(\ell - 1)a_{\ell}$. With the upper bounds for part $(a)$ and $(b)$ in \eqref{eqn:contribution_part_a_b}, the conditional expectation of $\|\rmP_{\rmU}\rve_i\|_2^2$ is then bounded by

\begin{align*}
     \E_{\{\tA^{(\ell)}\}_{\ell \in \sL}} \left[ \|\rmP_{\rmU}\rve_i\|_2^2 \Big\vert \{\rvu_{k}\}_{k=1}^{K} \right] =&\, \sum_{k=1}^{K} \E_{\{\tA^{(\ell)}\}_{\ell \in \sL}} \left[ \< \rvu_{k}, \rve_i\>^2 \Big\vert \{\rvu_{k}\}_{k=1}^{K} \right] \leq \frac{Kd}{N}(\LM + 2)\,,
\end{align*}
Let $\rX_{i}$ be the Bernoulli random variable defined by
\begin{equation*}
    \rX_{i} = \indi{ \|\rmP_{\rmU}\rve_{i} \|_2 > 2 \sqrt{ Kd(\LM + 2)/N }  } \,, \quad i\in \{i_1, \ldots, i_s\}\,.
\end{equation*}
By Markov's inequality,
\begin{align*}
    \E \rX_{i} = \P\left( \|\rmP_{\rmU}\rve_i \|_2 > 2 \sqrt{ Kd(\LM + 2)/N } \right) \leq \frac{ \E_{\{\tA^{(\ell)}\}_{\ell \in \sL}} \left[ \|\rmP_{\rmU}\rve_i\|_2^2 \Big\vert \{\rvu_{k}\}_{k=1}^{K} \right] }{4Kd(\LM + 2)/N} \leq \frac{1}{4}\,.
\end{align*}
Let $\delta \coloneqq s[2 \sum_{j=1}^{s}\E \rX_{i_j}]^{-1} - 1$ where $s = 2K\log^2(N)$. By Hoeffding \Cref{lem:Hoeffding},
\begin{align*}
    &\, \P \Bigg( \sum_{j=1}^{s} \rX_{i_j} \geq \frac{s}{2}\Bigg) = \P \Bigg( \sum_{j=1}^{s}(\rX_{i_j} - \E \rX_{i_j}) \geq \delta \sum_{j=1}^{s}\E \rX_{i_j}\Bigg) \\
    \leq &\, \exp \Bigg( - \frac{2\delta^2 \big(\sum_{j=1}^{s}\E \rX_{i_j} \big)^2}{s} \Bigg)  = O \Bigg( \frac{1}{N^{K\log(N)}} \Bigg)\,.
\end{align*}
Therefore, with probability $1 - O(N^{-K\log(N)})$, at least $s/2$ of the vectors $\rve_{i_1}, \dots, \rve_{i_s}$ satisfy $$\|\rmP_{\rmU}\rve_i \|_2 \leq 2 \sqrt{Kd(\LM + 2)/N}.$$ Meanwhile, for any $c\in(0, 2)$, there exists some large enough constant $\const_{2} \geq 2^{\LM + 1}\sqrt{(\LM + 2)/K}/c$ such that if $\sum_{\ell \in \sL}(\ell - 1)(a_{\ell} -b_{\ell}) > \const_{2} K^{\LM-1}\sqrt{d}$, then
\begin{align*}
	&\|\overline{\rvdelta}_i\|_2 =\frac{\sqrt{N}(\bar{\alpha} - \bar{\beta})}{2} = \frac{\sqrt{N}}{2}\sum_{\ell \in \sL} \binom{ \frac{3N}{4K} - 2}{\ell - 2} \frac{a_{\ell} - b_{\ell}}{\binom{N - 1}{\ell-1}} \\
	=\,& \frac{(1 + o(1))}{2\sqrt{N}} \sum_{\ell \in \sL} \Big( \frac{3}{4K} \Big)^{\ell - 2} (\ell - 1)(a_{\ell} - b_{\ell}) \geq \frac{K}{(2K)^{\LM-1} \sqrt{N}} \sum_{\ell \in \sL} (\ell - 1)(a_{\ell} - b_{\ell}),\\
	> \,& \frac{ \const_{2} K^{\LM}\sqrt{d} }{(2K)^{\LM-1}  \sqrt{N} }\geq \frac{ 2^{\LM + 1 }\sqrt{(\LM + 2)} K^{\LM}\sqrt{d} }{c(2K)^{\LM-1}  \sqrt{kn}} \\
 > &\frac{2}{c}~ 2\sqrt{ \frac{Kd(\LM + 2)}{N}} > \frac{2}{c} \|\rmP_{\rmU}\rve_i \|_2 .
\end{align*}
 \end{proof}

\begin{proof}[Proof of \Cref{lem:angle_accuracy_multiple}]
    Split the vertex set $\gV$ into $\widehat{\gV}_{+} = \{i|\ervv_{i} >0\}$ and $\widehat{\gV}_{-} = \{i|\ervv_{i} \leq 0\}$. Without loss of generality, assume that the first $\frac{N}{K}$ entries of $\overline{\rvv}$ are positive. We can write $\rvv$ in terms of its orthogonal projection onto $\overline{\rvv}$. Choose $\rvepsilon$ such that $\rvepsilon \perp \overline{\rvv}$ and $\|\rvepsilon\|_2 < c$ with $c_{1} \geq \sqrt{1 - c^{2}}$, then we can write
   \begin{align}
    \rvv = c_{1}\overline{\rvv} + \rvepsilon = \Big[\ervepsilon_{1} + \frac{c_{1}}{\sqrt{N}}, \cdots, \ervepsilon_{\frac{N}{K}} + \frac{c_{1}}{\sqrt{N}},\,\, \ervepsilon_{\frac{N}{K}+1} - \frac{c_{1}}{\sqrt{N}} , \cdots,  \ervepsilon_{N} -\frac{c_{1}}{\sqrt{N}} \Big]^{\sT}\,,
    \end{align}
    The number of entries in $\rvepsilon$ smaller than $-\sqrt{1 - c^{2}}/\sqrt{N}$ is at most $c^{2}(1 - c^{2})^{-1}N$. Note that $c_{1} \geq \sqrt{1 - c^{2}}$, thus at least $\frac{N}{K} - \frac{c^{2}}{1 - c^{2}}N$ indices $i$ with $\overline{\ervv}_{i} = \frac{1}{\sqrt{N}}$ will have $\ervv_i > 0$. Then the ratio we are seeking is at least 
    \begin{align*}
        \frac{ \frac{N}{K} - \frac{c^{2}}{1 - c^{2}}N}{ \frac{N}{K}} = 1 - \frac{Kc^2}{1 - c^2} > 1 - \frac{4K}{3}c^2. 
    \end{align*}
\end{proof}

\begin{proof}[Proof of \Cref{lem:discard_half_blue}]
We start with the following simple claim: for any $\ell \geq 2$ and any $\nu \in [1/2, 1)$, 
\begin{align} \label{eqn:asympt_binom}
\nu^{\ell} + (1-\nu)^{\ell} < \left ( \frac{1+\nu}{2} \right)^{\ell}~.
\end{align}
Indeed, one quick way to see this is by induction on $\ell$; we will induct from $\ell$ to $\ell + 2$. Assume the inequality is true for $\ell$; then  
\begin{eqnarray*}
\nu^{\ell + 2} +(1-\nu)^{\ell + 2} & = & \nu^2 \nu^{\ell} + (1-\nu)^2 (1-\nu)^{\ell} \\
& = &  \nu^2 \nu^{\ell} + (1-2\nu+\nu^2)(1-\nu)^{\ell} ~\leq ~\nu^2 \nu^{\ell}  
+ \nu^2 (1-\nu)^{\ell} \\
& = & \nu^2 (\nu^{\ell} + (1-\nu)^{\ell} ) ~<~ (\nu^2 + (1-\nu)^2) (\nu^{\ell} + (1-\nu)^{\ell} ) \\
& < & \left (\frac{1+\nu}{2} \right)^2 \left ( \frac{1+\nu}{2} \right)^{\ell} = \left ( \frac{1+\nu}{2} \right)^{\ell + 2}~,
\end{eqnarray*}
where we have used the induction hypothesis together with $1-2\nu \leq 0$ and $(1-\nu)^2>0$. After easily checking that the inequality works for $\ell =2,3$, the induction is complete. 

We shall now check that the quantities defined in   \Cref{lem:discard_half_blue} obey the relationship $\mu_2 \geq \mu_1$ and $\mu_2 - \mu_1 = \Omega(N)$, for $N$ large enough.  First, note that the only thing we need to check is that, for sufficiently large $N$, 
\[
\binom{\frac{\nu N}{2K}}{\ell} + \binom{\frac{(1-\nu)N}{2K}}{\ell} \leq \binom{\frac{(1+\nu)N}{4K}}{\ell} + (K - 1) \binom{\frac{(1-\nu)N}{4K(K - 1)}}{\ell}~;
\]
in fact, we will show the stronger statement that for any $\ell \geq 2$ and $N$ large enough, 
\begin{align} \label{eqn:strict_binom}
\binom{\frac{\nu N}{2K}}{\ell} + \binom{\frac{(1-\nu)N}{2K}}{\ell} <  \binom{\frac{(1+\nu)N}{4K}}{\ell} ~,
\end{align}
and this will suffice to see that the second part of the assertion, $\mu_2-\mu_1 = \Omega(N)$, is also true. 
Asymptotically, $\binom{\frac{\nu N}{2K}}{\ell} \sim \frac{\nu^{\ell}}{\ell !} \left ( \frac{N}{2K} \right)^{\ell} $, $\binom{\frac{(1-\nu)N}{2K}}{\ell} \sim \frac{(1-\nu)^{\ell}}{\ell !} \left ( \frac{N}{2K} \right)^{\ell}$, and $\binom{\frac{(1+\nu)N}{4K}}{\ell} \sim \frac{\left ( \frac{1+\nu}{2} \right)^{\ell}}{\ell !} \left ( \frac{N}{2K} \right)^{\ell}$. Then the inequality \eqref{eqn:strict_binom} follows from the claim \eqref{eqn:asympt_binom}. 

\vspace{.2cm}

Let \{$\gV_{1}, \dots, \gV_{K}$\} be the true partition of $\gV$. Recall that hyperedges in $\gH = \cup_{\ell \in \sL}\gH_{\ell}$ are colored red and blue with equal probability in \Cref{alg:multiple_partition_partial}.  Let $\gE_{\ell}(\gX)$ denote the set of blue $\ell$-uniform hyperedges with all vertices located in the vertex set $\gX$. 
Assume $|\gX \cap \gV_k| = \eta_k |\gX|$ with $\sum_{k=1}^{K}\eta_{k} = 1$. For each $\ell \in\sL$, the presence of hyperedge $e\in \gE_{\ell}(\gX)$ can be represented by independent Bernoulli random variables
\begin{equation*}
    \rT_{e}^{(a_{\ell})}\sim \mathrm{Bernoulli}\left( \frac{a_{\ell}}{ 2 \binom{N - 1}{\ell-1} }\right), \quad \rT_{e}^{(b_{\ell})} \sim \mathrm{Bernoulli}\left( \frac{b_{\ell}}{ 2\binom{N - 1}{\ell-1} }\right)\,,
\end{equation*}
depending on whether $e$ is a hyperedge with all vertices in the same community. Denote by \[\gE_{\ell}(\gX, a_{\ell}) \coloneqq \cup_{i=1}^{K}\gE_{\ell}(\gV_i\cap \gX)\] the union of all $\ell$-uniform sets of hyperedges with all vertices in the same $\gV_i\cap \gX$ for some $i\in [K]$, and by
\begin{align*}
    \gE_{\ell}(\gX, b_{\ell}) \coloneqq\, \gE_{\ell}(\gX) \setminus \gE_{\ell}(\gX, a_{\ell}) = \gE_{\ell}(\gX) \setminus \Big( \cup_i^{K} \gE_{\ell}(\gV_i\cap \gX) \Big)
\end{align*}
the set of $\ell$-uniform hyperedges with vertices across different blocks $\gV_i\cap \gX$. Then the cardinality $|\gE_{\ell}(\gX)|$ can be written as the
\begin{align*}
   |\gE_{\ell}(\gX)| = \sum_{e\in \gE_{\ell}(\gX, a_{\ell})} \rT_{e}^{(a_{\ell})} \, + \sum_{e\in \gE_{\ell}(\gX, b_{\ell})} \rT_{e}^{(b_{\ell})}\,,
\end{align*}
and by summing over $\ell$, the weighted cardinality $|\gE(\gX)|$ is written as
\begin{align*}
    |\gE(\gX)| \coloneqq \sum_{\ell \in \sL}\ell(\ell - 1)|\gE_{\ell}(\gX)| = \sum_{\ell \in \sL}\ell(\ell - 1)\left\{ \sum_{e\in \gE_{\ell}(\gX, a_{\ell})} \rT_{e}^{(a_{\ell})} \, + \sum_{e\in \gE_{\ell}(\gX, b_{\ell})} \rT_{e}^{(b_{\ell})} \right\}\,,
\end{align*}
with its expectation
\begin{align} \label{eq:true_expect}
    \E |\gE(\gX)| = \sum_{\ell \in \sL}\ell(\ell - 1)\left\{ \sum_{i=1}^{K} \binom{ \eta_{i} \frac{N}{2K}}{\ell} \frac{a_{\ell} - b_{\ell} }{ 2\binom{N - 1}{\ell-1} } + \binom{\frac{N}{2K}}{\ell} \frac{b_{\ell}}{ 2 \binom{N - 1}{\ell-1} } \right\}\,,
\end{align}
since 
\begin{align*}
    |\gE_{\ell}(\gX, a_{\ell})| = \sum_{i=1}^{K} |\gE_{\ell}(\gV_i\cap \gX)|= \sum_{i=1}^{K} \binom{\eta_{i} \frac{N}{2K}}{\ell}\,, \quad |\gE_{\ell}(\gX, b_{\ell})|= \binom{ \frac{N}{2K}}{\ell} - \sum_{i=1}^{K} \binom{\eta_{i} \frac{N}{2K}}{\ell}\,,
\end{align*}

Next, we prove the two Statements in \Cref{lem:discard_half_blue} separately. First, assume that  $|\gX \cap \gV_i| \leq \nu |\gX|$ ( i.e., $\eta_i \leq \nu$) for each $i\in [K]$. Then
        \begin{align*}
            \E |\gE(\gX)| \leq &\, \frac{1}{2} \sum_{\ell \in \sL}\ell(\ell - 1) \left\{ \left[ \binom{\frac{\nu N}{2K}}{\ell} + \binom{\frac{(1 - \nu)N}{2K}}{\ell} \right] \frac{a_{\ell} - b_{\ell} }{ \binom{N - 1}{\ell-1} } + \binom{\frac{N}{2K}}{\ell} \frac{b_{\ell}}{ \binom{N - 1}{\ell-1} } \right\} =:\mu_1\,.
        \end{align*}
To justify the above inequality, note that since $\sum\limits_{i=1}^{K} \eta_i = 1$, the sum 
$\sum_{i=1}^{K} \binom{\eta_i \frac{N}{2K}}{\ell}
$ is maximized when all but $2$ of the $\eta_i$ are $0$, and since all $\eta_i \leq \nu$, this means that 
\[
\sum_{i=1}^{K} \binom{\eta_i \frac{N}{2K}}{\ell} \leq \binom{\frac{\nu N}{2K}}{\ell} + \binom{\frac{(1-\nu)N}{2K}}{\ell}\,.
\] 
Note that $\ell(\ell - 1)(\rT_{e}^{(a_{\ell})} - \E \rT_{e}^{(a_{\ell})})$ and $\ell(\ell - 1)(\rT_{e}^{(b_{\ell})} - \E \rT_{e}^{(b_{\ell})})$ are independent mean-zero random variables bounded by $\ell(\ell - 1)$ for all $\ell \in \sL$, and $\Var(|\gE(\gX)|) \leq \LM^2(\ell - 1)^2\E |\gE(\gX)| = \Omega(N)$. Recall that $\mu_{\mathrm{T}}\coloneqq (\mu_1 + \mu_2)/2$. Define $t = \mu_{\mathrm{T}} - \E|\gE(\gX)|$, then $ 0 < (\mu_2 - \mu_1)/2 \leq t \leq \mu_{\mathrm{T}}$, hence $t = \Omega(N)$. By Bernstein's \Cref{lem:Bernstein}, we have
\begin{eqnarray*}
       \P \Big( |\gE(\gX)| \geq \mu_{\mathrm{T}} \Big) = \P \Big( |\gE(\gX)| - \E|\gE(\gX)|
       \geq  t \Big) 
       \leq & \exp\left(-\frac{t^2/2}{\Var(|\gE(\gX)|) + \ell(\ell - 1)t/3} \right) = O(e^{-cN})\,,
\end{eqnarray*}
where $c>0$ is some constant. On the other hand, if $|\gX \cap \gV_i| \geq \frac{1+\nu}{2}|\gX|$ for some $i\in [K]$,  then
    \begin{align*}
       &\,  \E|\gE(\gX)|\geq \frac{1}{2} \sum_{\ell \in \sL}\ell(\ell - 1) \left\{ \left[ \binom{\frac{(1 + \nu)N}{4K} }{\ell} + (K - 1)\binom{ \frac{(1 - \nu)N}{4K(K - 1)}}{\ell} \right] \frac{a_{\ell} - b_{\ell} }{ \binom{N - 1}{\ell-1} } + \binom{\frac{N}{2K}}{\ell} \frac{b_{\ell}}{ \binom{N - 1}{\ell-1} } \right\} =:\mu_2\,.
    \end{align*}
    The above can be justified by noting that at least one $|X\cap \gV_i| \geq \frac{1+\nu}{2} |\gX|$, and that the rest of the vertices will yield a minimal binomial sum when they are evenly split between the remaining $\gV_j$. Similarly, define $t = \mu_{\mathrm{T}} - \E|\gE(\gX)|$, then $0 < (\mu_2 - \mu_1)/2 \leq -t = \Omega(N)$, and Bernstein's \Cref{lem:Bernstein} gives
    \begin{align*}
       \P \Big( |\gE(\gX)| \leq \mu_{\mathrm{T}} \Big) = &\, \P \Big( |\gE(\gX)| - \E|\gE(\gX)| \leq -t \Big)\\
       \leq &\, \exp\left(-\frac{t^2/2}{\Var(|\gE(\gX)|) + \ell(\ell - 1)(-t)/3} \right) = O(e^{-c^{\prime} N})\,,
\end{align*}
where $c^{\prime}>0$ is some other constant.
\end{proof}

\begin{proof}[Proof of \Cref{lem:sample_k_distinct}]
    If vertex $i$ is uniformly chosen from $\gY_2$, the probability that $i \notin \gV_{k}$ for some $k \in [K]$ is 
\begin{align*}
    \P(i\notin \gV_{k} | i\in \gY_2) = \frac{\P(i\notin \gV_{k}, i\in \gY_2) }{\P(i \in \gY_2)} = 1 - \frac{|\gV_{k} \cap \gY_2|}{|\gY_2|} = 1 - \frac{\frac{N}{K} - N_{k} - \widetilde{N}_{k}}{N - \sum_{t=1}^{K}(N_t + \widetilde{N}_t) }\,,\quad k\in [K]\,,
\end{align*}
where $N_{t}$ and $\widetilde{N}_{t}$, defined in \eqref{eqn:dimension_ZcapVi} and \eqref{eqn:dimension_Y1capVi}, denote the cardinality of $\gZ \cap \gV_t$ and $\gY_1\cap \gV_t$ respectively. As proved in \Cref{lem:singular_value_approximation}, with probability at least $1 - 2\exp( -K \log^2(N))$, we have
\begin{align*}
    |N_{t} - N/(2K)|\leq \sqrt{N}\log(N), \quad |\widetilde{N}_{t} - N/(4K)|\leq \sqrt{N}\log(N).
\end{align*}
Then
$
    \P(i\notin \gV_{k} | i\in \gY_2) = 1-\frac{1}{K} \Big(1 + o(1) \Big)\,.
$
After $K\log^2(N)$ samples from $\gY_2$, the probability that there exists at least one node which belongs to $\gV_{k}$ is at least
\[
    1 - \bigg(1 - \frac{1+o(1)}{K} \bigg)^{K\log^2(N)}=1-N^{-(1+o(1))K\log(\frac{K}{K - 1})\log N}.
\]
The proof is completed by a union bound over $k\in [K]$.
\end{proof}

\subsection{Proofs of Lemmas in \Cref{subsec:local_correction}}
\begin{proof}[Proof of \Cref{lem:mislabel_probability_correction}]
    We calculate $\P(\widehat{\rS}_{11}^{(0)}(u) \leq \mu_{\rm{C}})$ first. Define $t_{\rm{1C}} \coloneqq \mu_{\rm{C}} - \E \widehat{\rS}_{11}^{(0)}(u)$, then by Bernstein \Cref{lem:Bernstein},
\begin{subequations}
    \begin{align*}
    &\, \P \left( \widehat{\rS}_{11}^{(0)}(u) \leq \mu_{\rm{C}}\right) = \P \left( \widehat{\rS}_{11}^{(0)}(u) - \E \widehat{\rS}_{11}^{(0)}(u) \leq t_{\rm{1C}}\right)\\
    \leq &\, \exp\left( - \frac{t_{\rm{1C}}^2/2}{\Var[\widehat{\rS}_{11}^{(0)}(u)] +(\LM - 1)\cdot  t_{\rm{1C}}/3} \right) \leq \exp\left( -\frac{3t_{\rm{1C}}^2/(\LM - 1)}{6(\LM - 1)\cdot \E \widehat{\rS}_{11}^{(0)}(u) +  2 t_{\rm{1C}}} \right)\\
    \leq&\, \exp\left( -\frac{[(\nu)^{\LM-1} - (1 - \nu)^{\LM-1}]^2 }{(\LM - 1)^2\cdot 2^{2\LM + 3}} \cdot \frac{ \left[\sum_{\ell \in \sL} (\ell - 1)\left(\frac{a_{\ell} - b_{\ell}}{K^{\ell - 1}} \right) \right]^2 }{\sum_{\ell \in \sL} (\ell - 1)\left(\frac{a_{\ell} - b_{\ell}}{K^{\ell - 1}} + b_{\ell} \right)}\right)\,,
\end{align*}
\end{subequations}
where $\sL$ is obtained from \Cref{alg:parameter_preprocessing} with $\LM$ denoting the maximum value in $\sL$, and the last two inequalities hold since $\Var[\widehat{\rS}_{11}^{(0)}(u)] \leq (\LM - 1)^2\E \widehat{\rS}_{11}^{(0)}(u)$, and for sufficiently large $N$,
    \begin{align*}
     t_{\rm{1C}} \coloneqq&\, \mu_{\rm{C}} - \E \widehat{\rS}_{11}^{(0)}(u) = -\frac{1}{2} \sum_{\ell \in \sL} (\ell - 1)\cdot\left[ \binom{\frac{\nu N}{2K}}{\ell - 1} - \binom{\frac{(1 - \nu)N}{2K}}{\ell - 1} \right]\frac{a_{\ell} - b_{\ell}}{2\binom{N - 1}{\ell-1}}\\
     \leq &\, - \frac{1}{2}\sum_{\ell \in \sL} \frac{(\nu)^{\ell - 1} - (1 - \nu )^{\ell - 1}}{2^{\ell}}\cdot(\ell - 1) \frac{a_{\ell} - b_{\ell}}{K^{\ell - 1}}(1 + o(1))\\
    \leq &\, -\frac{(\nu)^{\LM - 1} - (1 - \nu)^{\LM - 1} }{2^{\LM + 2}} \sum_{\ell \in \sL}(\ell - 1) \cdot\frac{a_{\ell} - b_{\ell}}{K^{\ell - 1}} \,,
    \end{align*}
as well as the following bound on the expectation
    \begin{align*}
    &\, 6(\LM - 1)\E \widehat{\rS}_{11}^{(0)}(u) +  2t_{\rm{1C}} = 2\mu_{\rm{C}} + (6\LM - 8)\E \widehat{\rS}_{11}^{(0)}(u)\\
         = &\, \sum_{\ell \in \sL} (\ell - 1)\left\{  \left[ (6\LM - 7)\binom{\frac{\nu N}{2K}}{\ell - 1} + \binom{\frac{(1 - \nu)N}{2K}}{\ell - 1} \right] \frac{a_{\ell} - b_{\ell}}{2\binom{N - 1}{\ell-1}} +  6(\LM - 1)\binom{\frac{N}{2K}}{\ell - 1}\frac{b_{\ell}}{2\binom{N - 1}{\ell-1}}\right\}\\
         = &\, \sum_{\ell \in \sL} (\ell - 1)\left[ \frac{(6\LM - 7)\cdot (\nu)^{\ell - 1} + (1 - \nu )^{\ell - 1} }{2^{\ell}}\cdot \frac{a_{\ell} - b_{\ell}}{K^{\ell - 1}} + \frac{(6\LM - 6)b_{\ell}}{2^{\ell}K^{\ell - 1}}
         \right](1+o(1))\\
         \leq &\, \sum_{\ell \in \sL} \frac{(6\LM - 7)\cdot(\nu)^{\ell - 1} + (1 - \nu )^{\ell - 1} }{2^{\ell}}\cdot (\ell - 1)\left(\frac{a_{\ell} - b_{\ell}}{K^{\ell - 1}} + b_{\ell}\right)(1+o(1))\\
         \leq &\, \frac{3(\LM - 1)}{2}\sum_{\ell \in \sL} (\ell - 1)\left(\frac{a_{\ell} - b_{\ell}}{K^{\ell - 1}} + b_{\ell}\right)\,.
\end{align*}
Similarly, for $\P(\widehat{\rS}_{1j}^{(0)}(u) \geq \mu_{\rm{C}})$, define $t_{\rm{jC}} \coloneqq \mu_{\rm{C}} - \E \widehat{\rS}_{1j}^{(0)}(u)$ for $j\neq 1$, by Bernstein's \Cref{lem:Bernstein}, 
\begin{subequations}
\begin{align*}
    &\, \P \left( \widehat{\rS}_{1j}^{(0)}(u) \geq \mu_{\rm{C}}\right) = \P \left( \widehat{\rS}_{1j}^{(0)}(u) - \E \widehat{\rS}_{1j}^{(0)}(u) \geq t_{\rm{jC}}\right)\\
    \leq &\, \exp\left( - \frac{t_{\rm{jC}}^2/2}{\Var[\widehat{\rS}_{1j}^{(0)}(u)] + (\LM - 1)\cdot t_{\rm{jC}}/3} \right) \leq \exp\left( -\frac{3t_{\rm{jC}}^2/(\LM - 1)}{6(\LM - 1)\cdot\E \widehat{\rS}_{1j}^{(0)}(u) +  2 t_{\rm{jC}}} \right)\\
    \leq&\, \exp\left( -\frac{[(\nu)^{\LM-1} - (1 - \nu )^{\LM-1}]^2 }{(\LM - 1)^2\cdot 2^{2\LM + 3}} \cdot \frac{ \left[\sum_{\ell \in \sL} (\ell - 1)\left(\frac{a_{\ell} - b_{\ell}}{K^{\ell - 1}} \right) \right]^2 }{\sum_{\ell \in \sL} (\ell - 1)\left(\frac{a_{\ell} - b_{\ell}}{K^{\ell - 1}} + b_{\ell} \right)} \right)\,.
\end{align*}
\end{subequations}
The last two inequalities holds since $\Var[\widehat{\rS}_{1j}^{(0)}(u)] \leq (\LM - 1)^2\E \widehat{\rS}_{1j}^{(0)}(u)$, and for sufficiently large $N$, 
\begin{subequations}
    \begin{align*}
     t_{\rm{jC}} \coloneqq&\, \mu_{\rm{C}} - \E \widehat{\rS}_{1j}^{(0)}(u) = \frac{1}{2} \sum_{\ell \in \sL} (\ell - 1)\cdot \left[ \binom{\frac{\nu N}{2K}}{\ell - 1} - \binom{\frac{(1 - \nu)N}{2K}}{\ell - 1} \right]\frac{a_{\ell} - b_{\ell}}{2\binom{N - 1}{\ell-1}}\\
    \geq &\, \frac{(\nu)^{\LM - 1} - (1 - \nu )^{\LM - 1} }{2^{\LM + 2}} \sum_{\ell \in \sL}(\ell - 1) \cdot\frac{a_{\ell} - b_{\ell}}{K^{\ell - 1}} \,\,,
    \end{align*}
\end{subequations}
as well as the following bound on the expectation
\begin{subequations}
    \begin{align*}
    &\,6\E \widehat{\rS}_{1j}^{(0)}(u) +  2t_{\rm{jC}} = 2\mu_{\rm{C}} + (6\LM - 8)\E \widehat{\rS}_{1j}^{(0)}(u)\\
         = &\, \sum_{\ell \in \sL} (\ell - 1)\left\{ \left[ \binom{\frac{\nu N}{2K}}{\ell - 1} + (6\LM - 7)\binom{\frac{(1 - \nu)N}{2K}}{\ell - 1} \right] \frac{a_{\ell} - b_{\ell}}{2\binom{N - 1}{\ell-1}} +  6(\LM -1)\binom{\frac{N}{2K}}{\ell - 1}\frac{b_{\ell}}{2\binom{N - 1}{\ell-1}}\right\}\\
        = &\, \sum_{\ell \in \sL} (\ell - 1)\cdot\left( \frac{(\nu)^{\ell - 1} + (6\LM - 7)\cdot (1 - \nu )^{\ell - 1}}{2^{\ell}}\cdot \frac{a_{\ell} - b_{\ell}}{K^{\ell - 1}} + \frac{(6\LM - 6)\cdot b_{\ell}}{2^{\ell}K^{\ell - 1}}\right)(1+o(1))\\
         \leq &\, \sum_{\ell \in \sL} \frac{ (\nu)^{\ell - 1} + (6\LM - 7)\cdot(1 - \nu )^{\ell - 1}}{2^{\ell}}\cdot (\ell - 1)\left(\frac{a_{\ell} - b_{\ell}}{K^{\ell - 1}} + b_{\ell}\right)(1+o(1))\\
         \leq &\, \frac{3(\LM - 1)}{2}\sum_{\ell \in \sL} (\ell - 1)\left(\frac{a_{\ell} - b_{\ell}}{K^{\ell - 1}} + b_{\ell}\right).
\end{align*}
\end{subequations}
\end{proof}

\section{Algorithm correctness for the binary case}\label{sec:analysis_binary}

We will show the correctness of \Cref{alg:binary_partition_partial} and prove \Cref{thm:partial_binary} in this section. Without loss of generality, we assume $N$ is even to guarantee the existence of a binary partition of size $N/2$. The method to deal with the odd $N$ case was discussed in \Cref{lem:general_eigen_value_approximation}. The analysis will mainly follow from the analysis in \Cref{sec:analysis_multiple}. We only detail the differences. 

First, with $\alpha$ and $\beta$ defined in \eqref{eqn:alpha_beta}, \eqref{eqn:expected_adjacency_multiple} becomes 
\begin{align}
  \overline{\rmA} \coloneqq \E \rmA = 
  \begin{bmatrix}
     \alpha \rmJ_{\frac{N}{2}}  &  \beta \rmJ_{\frac{N}{2}}\\
    \beta  \rmJ_{\frac{N}{2}} & \alpha \rmJ_{\frac{N}{2}}
  \end{bmatrix} - \alpha \rmI_{N}\,, \label{eqn:expected_adjacency_binary}
\end{align}

Second, let the index set be \(\sI = \{i\in \gV: \mathrm{row}(i) \leq 20\LM d\}\), as shown in \eqref{eqn:restricted_adjacency_matrix_partial}.  Let $\rvu_i$ (resp. $\overline{\rvu}_i$) denote the eigenvector associated to $\lambda_i(\rmA_{\sI})$ (resp. $\lambda_i(\overline{\rmA})$) for $i =1, 2$. Define two linear subspaces $\rmU \coloneqq \mathrm{Span}\{\rvu_{1}, \rvu_{2}\}$ and $\overline{\rmU} \coloneqq \mathrm{Span}\{\overline{\rvu}_{1}, \overline{\rvu}_{2}\}$, then the angle between $\rmU$ and $\overline{\rmU}$ is defined as
    $
      \sin \angle(\rmU, \overline{\rmU}) \coloneqq \|\rmP_{\rmU} - \rmP_{\overline{\rmU}}\|
    $,
where $\rmP_{\rmU}$ and $\rmP_{\overline{\rmU}}$ are the orthogonal projections onto $\rmU$ and $\overline{\rmU}$, respectively.

\subsection{Proof of \Cref{lem:spectral_clustering_accuracy_binary}}

\subsubsection{Bound the angle between $\rmU$ and $\overline{\rmU}$}\label{subsubsec:bound_subspace_angle_binary} 

The strategy to bound the angle is similar to \Cref{subsubsec:bound_subspace_angle_multiple}, except that we apply Davis-Kahan Theorem (\Cref{lem:Davis_Kahan_sin}) here. 

Define $\rmE \coloneqq \rmA - \overline{\rmA}$ and its restriction on $\sI$, namely $\rmE_{\sI} \coloneqq (\rmA - \overline{\rmA})_{\sI} = \rmA_{\sI} - \overline{\rmA}_{\sI}$, as well as $\rvdelta\coloneqq \overline{\rmA}_{\sI} - \overline{\rmA}$. Then the deviation $\rmA_{\sI} - \overline{\rmA}$ is decomposed as
\begin{align}
    \rmA_{\sI} - \overline{\rmA} = (\rmA_{\sI} - \overline{\rmA}_{\sI}) + (\overline{\rmA}_{\sI} - \overline{\rmA}) = \rmE_{\sI} + \rvdelta\,.\notag
\end{align}

\Cref{thm:regularization_concentration_partial} indicates $\|\rmE_{\sI}\|\leq \const_{3}\sqrt{d}$ with probability at least $1 - N^{-2}$ when taking $\tau=20\LM, \theta =3$, where $\const_{3}$ is a constant depending only on $\LM$. Moreover, \Cref{lem:high_degree_vertices_k} shows that the number of vertices with high degrees is relatively small. Consequently, an argument similar to \Cref{cor:norm_of_high_degree_vertices_multi_block} leads to the conclusion $\|\rvdelta\| \leq \sqrt{d}$ with high probability. Together with upper bounds for $\|\rmE_{\sI}\|$ and $\|\rvdelta\|$, \Cref{lem:subspace_angle_binary} shows that the angle between $\rmU$ and $\overline{\rmU}$ is relatively small with high probability.

\begin{lemma}\label{lem:subspace_angle_binary}
    For any $c\in (0,1)$, there exists a constant $\const_{2}$ depending on $\LM$ and $c$ such that if
  \begin{align}
     \sum_{\ell \in \sL}(\ell - 1)(a_{\ell}-b_{\ell})\geq \const_{2} \cdot 2^{\LM+2}\sqrt{d},\notag
  \end{align}
   then $\sin \angle (\rmU, \overline{\rmU}) \leq c$ with probability $1 - N^{-2}$. 
\end{lemma}
\begin{proof}[Proof of \Cref{lem:subspace_angle_binary}]
    First, with probability $1 - N^{-2}$, we have
\begin{align}
     \|\rmA_{\sI} - \overline{\rmA}\| \leq  \|\rmE_{\sI}\| + \|\rvdelta \|\leq (\const_{3}+1)\sqrt{d}. \notag 
\end{align}
    According to the definitions in \eqref{eqn:alpha_beta}, $\alpha \geq \beta$ and $\alpha  = O(1/N)$, $\beta = O(1/N)$. Meanwhile, \Cref{lem:eigenvalues_HSBM} shows that $|\lambda_2(\overline{\rmA})| = [- \alpha + (\alpha - \beta)N/2]$ and $|\lambda_{3}(\overline{\rmA})| = \alpha$. Then
    \begin{subequations}
        \begin{align*}
        &\,|\lambda_2(\overline{\rmA})| - |\lambda_{3}(\overline{\rmA})| = \frac{N}{2}(\alpha - \beta) - 2\alpha \geq \frac{3}{4}\cdot \frac{N}{2}(\alpha - \beta) = \frac{3N}{8} \sum_{\ell \in \sL} \binom{\frac{N}{2} - 2}{\ell - 2}  \frac{(a_{\ell} - b_{\ell})}{\binom{N - 1}{\ell-1}}\\ \geq &\, \frac{1}{4}\sum_{\ell \in \sL} \frac{(\ell - 1)(a_{\ell} - b_{\ell})}{2^{\ell - 2}} \geq \frac{1}{2^{\LM}}\sum_{\ell \in \sL}(\ell - 1)(a_{\ell} - b_{\ell}) \geq 4\const_{2}\sqrt{d}\,.
        \end{align*}
    \end{subequations}
    Then for some large enough $\const_{2}$, the following condition for Davis-Kahan Theorem (\Cref{lem:Davis_Kahan_sin}) is satisfied
\begin{align}
      \|\rmA_{\sI} - \overline{\rmA} \|\leq (1-1/\sqrt{2}) \left(|\lambda_2(\overline{\rmA})| - |\lambda_{3}(\overline{\rmA})|\right). \notag 
\end{align}
Then for any $c\in(0, 1)$, we can choose $\const_{2} = (\const_{3} + 1)/c$ such that 
    \begin{align}
        \|\rmP_{\rmU} - \rmP_{\overline{\rmU}}\| \leq \frac{2\|\rmA_{\sI} - \overline{\rmA}\|}{ |\lambda_2(\overline{\rmA})| - |\lambda_{3}(\overline{\rmA})|} \leq \frac{2(\const_{3}+1)\sqrt{d}}{4\const_{2}\sqrt{d}} = \frac{c}{2} \leq c\,. \notag 
    \end{align}
\end{proof}

Now, we focus on the accuracy of \Cref{alg:binary_block_spectral_partition}, once the conditions in \Cref{lem:subspace_angle_binary} are satisfied.
\begin{lemma}[{\cite[Lemma 23]{Chin2015StochasticBM}}]\label{lem:vector_angle_binary}
If $\sin \angle(\overline{\rmU}, \rmU) \leq c \leq \frac{1}{4}$, there exists a unit vector $\rvv\in \rmU$ such that the angle between $\overline{\rvu}_{2}$ and $\rvv$ satisfies $\sin \angle(\overline{\rvu}_{2}, \rvv) \leq 2 \sqrt{c}$.
\end{lemma}
\begin{proof}[Proof of \Cref{lem:vector_angle_binary}]
    Recall that $\overline{\rmU} \coloneqq \mathrm{Span}\{\overline{\rvu}_{1}, \overline{\rvu}_{2}\}$, $\rmU \coloneqq \mathrm{Span}\{\rvu_{1}, \rvu_{2}\}$, where 
    \begin{align*}
        \overline{\rvu}_{1} &= [\frac{1}{\sqrt{N}}, \cdots, \frac{1}{\sqrt{N}}]^{\sT},\\
        \overline{\rvu}_{2} &= [\frac{1}{\sqrt{N}}, \cdots, \frac{1}{\sqrt{N}}, -\frac{1}{\sqrt{N}}, \cdots, -\frac{1}{\sqrt{N}}]^{\sT}\,.
    \end{align*}
    By assumption, the angle between $\overline{\rmU}$ and $\rmU$ is bounded by
    \begin{align*}
        \sin \angle(\overline{\rmU}, \rmU) \coloneqq \|\rmP_{\overline{\rmU}} - \rmP_{\rmU}\| \leq c \leq \frac{1}{4},
    \end{align*}
    Define $\rvv_{k} \coloneqq \rmP_{\rmU} \overline{\rvu}_{k}$ and $\overline{\rvx}_{k} =  \overline{\rvu}_{k} - \rvv_{k}$ for $k = 1, 2$. Note that $\rvv_{k} \in \rmU$ and $\overline{\rvu}_{k} \in \overline{\rmU}$, then
    \begin{align*}
        \|\overline{\rvx}_{k}\|_2 = \|\overline{\rvu}_{k} - \rvv_{k}\|_2 \leq \|\rmP_{\overline{\rmU}} - \rmP_{\rmU}\| \cdot \|\overline{\rvu}_{k}\|_{2} \leq c.
    \end{align*}
    Meanwhile, $\rvv_{k} =  \overline{\rvu}_{k} - \overline{\rvx}_{k}$, then it implies $\|\rvv_{k}\|_{2}\geq 1 - c$. Since $\overline{\rvu}_{1}^{\sT} \overline{\rvu}_{2} = 0$, we have 
    \begin{align*}
        |\rvv_{1}^{\sT}\overline{\rvu}_{2}| = |\overline{\rvu}_{1}^{\sT}\overline{\rvu}_{2} - \overline{\rvx}_{1}^{\sT}\overline{\rvu}_{2}| \leq |\overline{\rvu}_{1}^{\sT}\overline{\rvu}_{2}| + |\overline{\rvx}_{1}^{\sT}\overline{\rvu}_{2}| \leq c\,.
    \end{align*}
    Together with the fact that  $|\rvv_{1}^{\sT}\overline{\rvx}_{2}| \leq c$, $|\rvv_{2}^{\sT}\overline{\rvx}_{1}| \leq c$, we have
    \begin{align*}
        |\rvv_{1}^{\sT}\rvv_{2}| =&\, |\overline{\rvu}_{1}^{\sT}\overline{\rvu}_{2} - \overline{\rvx}_{1}^{\sT}\rvv_{2} - \rvv_{1}^{\sT}\overline{\rvx}_{2} + \overline{\rvx}_{1}^{\sT}\overline{\rvx}_{2} |\\
        \leq &\, |\overline{\rvx}_{1}^{\sT}\rvv_{2}| + |\rvv_{1}^{\sT}\overline{\rvx}_{2}| + |\overline{\rvx}_{1}^{\sT}\overline{\rvx}_{2}| \leq 2c + c^{2}\,.
    \end{align*}
    Note that $|\overline{\rvx}_{k}^{\sT}\overline{\rvu}_{k}| \leq \|\overline{\rvx}_{k}\|_{2} \leq c$ for $k = 1, 2$, then
    \begin{align*}
        |\rvv_{k}^{\sT}\overline{\rvu}_{k}| =&\, |\overline{\rvu}_{k}^{\sT}\overline{\rvu}_{k} - \overline{\rvx}_{k}^{\sT}\overline{\rvu}_{k}| \geq |\overline{\rvu}_{k}^{\sT}\overline{\rvu}_{k}| - |\overline{\rvx}_{k}^{\sT}\overline{\rvu}_{k}| \geq 1 - c.
    \end{align*}
    Let $\rvv_{\perp} = \rvv_{2} - \rvv_{1}\frac{\rvv_{1}^{\sT}\rvv_{2}}{\|\rvv_{1}\|^{2}_{2}} \in \rmU$, using the estimates above, we then have 
    \begin{align*}
        |\rvv_{\perp}^{\sT}\overline{\rvu}_{2}| =&\, |\rvv_{2}^{\sT}\overline{\rvu}_{2} - \frac{\rvv_{1}^{\sT}\rvv_{2}}{\|\rvv_{1}\|^{2}_{2}}\rvv_{1}^{\sT}\overline{\rvu}_{2}|\\
        \geq &\, |\rvv_{2}^{\sT}\overline{\rvu}_{2}| - \bigg| \frac{\rvv_{1}^{\sT}\rvv_{2}}{\|\rvv_{1}\|^{2}_{2}}\rvv_{1}^{\sT}\overline{\rvu}_{2} \bigg|\\
        \geq &\, (1 - c) - \frac{(2c + c^{2})c}{(1 - c)^{2}} \geq 1 - 2c.
    \end{align*}
    Define the unit vector $\rvv \coloneqq \frac{\rvv_{\perp}}{\|\rvv_{\perp}\|_2}$, then we have
    \begin{align*}
        \sin \angle(\rvv, \overline{\rvu}_{2}) = \sqrt{1 - |\overline{\rvu}_{2}^{\sT}\rvv|^{2}} \leq \sqrt{1 - (1 - 2c)^{2}} = 2\sqrt{c}\,,
    \end{align*}
    which is the desired result.
\end{proof}

The vector $\rvv$ constructed by \Cref{alg:binary_block_spectral_partition} is the unit vector perpendicular to $\rmP_{\rmU}\ones_{N}$, where $\rmP_{\rmU}\ones_{N}$ is the projection of all-ones vector onto $\rmU$. \Cref{lem:subspace_angle_binary} and \Cref{lem:vector_angle_binary} together give the following corollary.
\begin{corollary}
    For any  $c \in (0, 1)$, there exists a unit vector $\rvv\in \rmU$ such that the angle between $\overline{\rvu}_{2}$ and $\rvv$ satisfies $\sin \angle( \overline{\rvu}_{2}, \rvv) \leq c < 1$ with probability  $1 - O(e^{-N})$.
\end{corollary}
\begin{proof}
    For any $c \in (0, 1)$, we could choose constants $\const_{2}$, $\const_{3}$ in \Cref{lem:subspace_angle_binary} such that $\sin \angle(\overline{\rmU}, \rmU) \leq \frac{c^{2}}{4} < 1$. Then by \Cref{lem:vector_angle_binary}, we construct $\rvv$ such that $\sin \angle(\overline{\rvu}_{2}, \rvv) \leq c$.
\end{proof}

The proof of \Cref{lem:spectral_clustering_accuracy_binary} is then completed by choosing $\const_{2}$, $\const_{3}$ in \Cref{lem:subspace_angle_binary} such that $c \leq \frac{1}{4}$, and applying \Cref{lem:angle_accuracy_multiple} with $K = 2$.

\subsection{Proof of \Cref{lem:correction_accuracy_binary}}
The proof strategy is similar to \Cref{subsec:local_correction} and \Cref{subsec:merging}. In \Cref{alg:binary_partition_partial}, we first color the hyperedges with red and blue with equal probability. By running \Cref{alg:binary_block_correction} on the red graph, we obtain a $\nu$-correct partition $\widehat{\gV}_{1}, \widehat{\gV}_{2}$ of $\gV = \gV_1 \cup \gV_2$, i.e., $|\gV_{k}\cap \widehat{\gV}_{k}| \geq \nu N/2$ for $k = 1, 2$. We condition on this event in the rest of the proof. Meanwhile, we shall also condition the event that the maximum red degree of a vertex is at most $\log^{2}(N)$ with high probability, which is a simple consequence of concetration by Bernstein's inequality (\Cref{lem:Bernstein}).

Similarly, we consider the probability of a hyperedge $e = \{i_{1}, \ldots, i_{\ell}\}$ being blue conditioning on the event that $e$ is not a red hyperedge in each underlying $\ell$-uniform hypergraph separately. If vertices $i_{1}, \ldots, i_{\ell}$ are all from the same true cluster, then the probability is $\psi_{\ell}$, otherwise $\phi_{\ell}$, where $\psi_{\ell}$ and $\phi_{\ell}$ are defined in \eqref{eqn:psi_l} and \eqref{eqn:phi_l}, and the presence of those hyperedges are represented by random variables $\zeta_e^{(a_{\ell})} \sim \mathrm{Bernoulli}\left(\psi_{\ell}\right)$, $\xi_e^{(b_{\ell})} \sim \mathrm{Bernoulli}\left(\phi_{\ell}\right)$, respectively.

Following a similar argument in \Cref{subsec:local_correction}, the row sum of $u$ can be written as
\begin{align}
    \widehat{\rS}_{kj}^{(0)}(u)\coloneqq \sum_{\ell \in \sL}(\ell - 1)\cdot \left\{ \sum_{e\in \, \gE^{(a_{\ell})}_{k, j}} \zeta_e^{(a_{\ell})} + \sum_{e\in \gE^{(b_{\ell})}_{k, j}} \xi_e^{(b_{\ell})}\right\}\,,\quad u\in \gV_{k}\,, \notag 
\end{align}
where $\gE^{(a_{\ell})}_{k, j}\coloneqq \gE_{\ell}([\gV_{k}]^{1}, [\gV_{k}\cap \widehat{\gV}_{j}^{(0)}]^{\ell - 1})$ denotes the set of $\ell$-hyperedges with $1$ vertex from $[\gV_{k}]^{1}$ and the other $\ell - 1$ vertices from $[\gV_{k}\cap \widehat{\gV}_{j}^{(0)}]^{\ell - 1}$, while
    \begin{align*}
        \gE^{(b_{\ell})}_{k, j} \coloneqq \gE_{\ell} \Big([\gV_{k}]^{1}, \,\, [\widehat{\gV}_{j}^{(0)}]^{\ell - 1} \setminus [\gV_{k}\cap \widehat{\gV}_{j}^{(0)}]^{\ell - 1} \Big)
    \end{align*}
denotes the set of $\ell$-hyperedges with $1$ vertex in $[\gV_{k}]^{1}$ while the remaining $\ell - 1$ vertices in $[\widehat{\gV}_{j}^{(0)}]^{\ell - 1}\setminus[\gV_{k}\cap \widehat{\gV}_{j}^{(0)}]^{\ell - 1}$, with their cardinalities
    \begin{align}
       |\gE^{(a_{\ell})}_{k, j}| \leq \binom{|\gV_{k}\cap \widehat{\gV}_{j}^{(0)}|}{\ell - 1}\,,\quad |\gE^{(b_{\ell})}_{k, j}| \leq \left[ \binom{|\widehat{\gV}_{j}^{(0)}|}{\ell - 1} - \binom{|\gV_{k}\cap \widehat{\gV}_{j}^{(0)}|}{\ell - 1} \right]\,. \notag 
    \end{align}
According to the facts $|\gV_{k}\cap \widehat{\gV}_{k}^{(0)}| \geq \nu N/2$, $|\gV_{k}| = N/2$, $|\widehat{\gV}_{k}^{(0)}| = N/2$ for $k = 1, 2$, we have
\begin{align}
    |\gE^{(a_{\ell})}_{k, k}| \geq \binom{ \frac{\nu N}{2}}{\ell - 1}\,,\quad |\gE^{(a_{\ell})}_{k, j}| \leq \binom{\frac{(1 - \nu)N}{2}}{\ell - 1}\,,\,\, j\neq k\,. \notag 
\end{align}
To simplify the calculation, we take the lower and upper bound of $|\gE^{(a_{\ell})}_{k, k}|$ and $|\gE^{(a_{\ell})}_{k, j}|(j\neq k)$ respectively. Taking expectation with respect to $\zeta_e^{(a_{\ell})}$ and $\xi_e^{(b_{\ell})}$, for any $u \in \gV_{k}$, we have
\begin{subequations}
    \begin{align}
    \E \widehat{\rS}_{kk}^{(0)}(u) &= \sum_{\ell \in \sL} (\ell - 1)\cdot\left[ \binom{\frac{\nu N}{2}}{\ell - 1} (\psi_{\ell} - \phi_{\ell}) + \binom{\frac{N}{2}}{\ell - 1} \phi_{\ell}\right], \notag \\
    \E \widehat{\rS}_{kj}^{(0)}(u) &= \sum_{\ell \in \sL} (\ell - 1)\cdot\left[ \binom{\frac{(1 - \nu)N}{2}}{\ell - 1} (\psi_{\ell} - \phi_{\ell}) +  \binom{\frac{N}{2}}{\ell - 1} \phi_{\ell} \right],\,\, j\neq k\,. \notag 
\end{align}
\end{subequations}
By assumptions in \Cref{thm:partial_multiple}, $\E \widehat{\rS}_{kk}^{(0)}(u) - \E \widehat{\rS}_{kj}^{(0)}(u) = \Omega(1)$. We define
\begin{align}
    \mu_{\rm{C}} \coloneqq \frac{1}{2}\sum_{\ell \in \sL} (\ell - 1)\cdot\left\{ \left[ \binom{\frac{\nu N}{2}}{\ell - 1} + \binom{ \frac{(1 - \nu)N}{2} }{\ell - 1} \right](\psi_{\ell} - \phi_{\ell}) + 2\binom{ \frac{N}{2} }{\ell - 1} \phi_{\ell}\right\}\,. \notag 
\end{align}
After \Cref{alg:binary_block_correction}, if a vertex $u\in \gV_{k}$ is mislabelled, one of the following events must happen
\begin{itemize}
    \item $\widehat{\rS}_{kk}^{(0)}(u) \leq \mu_{\rm{C}}$, meaning that $u$ fails to have enough neighbors in $\widehat{\gV}_{k}^{(0)}$ to be correctly labeled as $k$;
    \item $\widehat{\rS}_{kj}^{(0)}(u) \geq \mu_{\rm{C}}$, for some $j\neq k$, meaing that $u$ survived \Cref{alg:binary_block_correction} without being corrected.
\end{itemize}
By an argument similar to \Cref{lem:mislabel_probability_correction}, we can prove that
\begin{align}
    \widehat{\rho}_{1}^{(0)} = \P \left( \widehat{\rS}_{kk}^{(0)}(u) \leq \mu_{\rm{C}}\right) \leq \rho\,, \quad \widehat{\rho}_{2}^{(0)} = \P \left( \widehat{\rS}_{kj}^{(0)}(u) \geq \mu_{\rm{C}}\right) \leq \rho\,, \notag 
\end{align}
where $\rho = \exp\left( - \const_{\sL}(2, \nu)\cdot \mathrm{SNR}_{\sL}(2) \right)$ and  
\begin{align}
    \const_{\sL}(2, \nu)\coloneqq\frac{[(\nu)^{\LM-1} - (1 - \nu)^{\LM-1}]^2 }{8(\LM - 1)^2}, \quad \mathrm{SNR}_{\sL}(2)\coloneqq \frac{ \left[\sum_{\ell \in \sL} (\ell - 1)\cdot (a_{\ell} - b_{\ell})2^{-\ell + 1} \right]^2 }{\sum_{\ell \in \sL} (\ell - 1)\left((a_{\ell} - b_{\ell}) 2^{-\ell + 1} + b_{\ell} \right)}\,. \notag 
\end{align}
As a result, the probability that either of those events happened is bounded by $\rho$. The number of mislabeled vertices in $\gV_1$ after \Cref{alg:multi_block_correction} is at most
\begin{align*}
    \widehat{\rR}_{k}^{(0)} = \sum_{v\in \gV_{k}\setminus \widehat{\gV}_{k}}\Gamma_{v}\, + \sum_{v \in \gV_{k} \cap \widehat{\gV}_{k}}\Lambda_{v}\,,
\end{align*}
where $\Gamma_{v}$ (resp. $\Lambda_{v}$) are i.i.d indicator random variables with mean $\widehat{\rho}_{1}^{(0)}$ (resp. $\widehat{\rho}_{2}^{(0)}$). Then 
\begin{align*}
   \E \widehat{\rR}_{k}^{(0)} \leq  \frac{N}{2} \widehat{\rho}_{1}^{(0)} + \frac{(1 - \nu)N}{2} \widehat{\rho}_{2}^{(0)} =(1 - \nu/2)N\rho\,. 
\end{align*}
where $\nu$ is the correctness after \Cref{alg:multi_block_spectral_partition}. Let $t_{k}\coloneqq (1 + \nu/2)N\rho$, then by Chernoff \Cref{lem:Chernoff},
\begin{align*}
   \P \left( \widehat{\rR}_{k}^{(0)} \geq N\rho \right) = \P \left[ \widehat{\rR}_{k}^{(0)} - (1 - \nu/2)N\rho \geq t_{k} \right] \leq \P \left( \widehat{\rR}_{k}^{(0)} - \E \widehat{\rR}_{k}^{(0)} \geq t_{k} \right) \leq e^{-c t_{k}} = O(e^{-N\rho})\,,
\end{align*}
which means that with probability $1- O(e^{- N\rho})$, the fraction of mislabeled vertices in $\gV_{k}$ is smaller than $2\rho$, i.e., the correctness is at least $\gamma \coloneqq \max\{\nu,\, 1 - 2\rho\}$. 

\section{Technical Lemmas}

\begin{lemma}\label{lem:binomial_convex}
    The function $f: \N^{+} \mapsto \N^{+}$, defined by $f(N) = \binom{N}{\ell}$, is convex for any $\ell \geq 2$. In particular, if $N_{1} + N_{2} + \cdots + N_K = N$ for some $K\in \N^{+}$, then $f(N_{1}) + f(N_{2}) + \cdots + f(N_k) \leq f(N)$.
\end{lemma}
\begin{proof}
    By induction, it suffices to check when $K=2$, which can be verified by direction calculation $\binom{N_{1}}{\ell}+\binom{N_{2}}{\ell}\leq \binom{N_{1}+N_{2}}{\ell}$.
\end{proof}

\begin{lemma}[Markov's inequality, {\cite[Proposition $1.2.4$]{Vershynin2018HighDP}}]\label{lem:Markov}
    For any non-negative random variable $\rX$ and $t>0$, we have 
    \begin{align*}
        \P(\rX > t) \leq \E(\rX)/t.
    \end{align*}
\end{lemma}

\begin{lemma}[Hoeffding's inequality, {\cite[Theorem $2.2.6$]{Vershynin2018HighDP}}]\label{lem:Hoeffding}
    Let $\rX_1,\dots, \rX_{N}$ be independent  random variables with $\rX_i \in [a_i, b_i]$, then for any, $t \geq 0$, we have
    \begin{align*}
        \P\Bigg( \bigg|\sum_{i=1}^{N}(\rX_i - \E \rX_i ) \bigg| \geq t \Bigg) \leq 2\exp \Bigg( -\frac{2t^{2} }{ \sum_{i=1}^{N}(b_i - a_i )^{2} } \Bigg)\,.
    \end{align*}
\end{lemma}

\begin{lemma}[Chernoff's inequality, {\cite[Theorem $2.3.1$]{Vershynin2018HighDP}}]\label{lem:Chernoff}
    Let $\rX_i$ be independent Bernoulli random variables with parameters $p_i$. Consider their sum $\rS_{N} = \sum_{i=1}^{N}\rX_i$ and denote its mean by $\mu = \E \rS_{N}$. Then for any $t > \mu$, 
    \begin{align*}
        \P \big( \rS_{N} \geq t \big) \leq e^{-\mu} \left( \frac{e \mu}{t} \right)^{t}\,.
    \end{align*}
\end{lemma}

\begin{lemma}[Bernstein's inequality, {\cite[Theorem $2.8.4$]{Vershynin2018HighDP}}]\label{lem:Bernstein}
    Let $\rX_1,\dots, \rX_{N}$ be independent mean-zero random variables such that $|\rX_i|\leq K$ for all $i$. Let $\sigma^2 = \sum_{i=1}^{N}\E \rX_i^2$. Then for every $t \geq 0$,
    \begin{align*}
        \P \Bigg( \Big|\sum_{i=1}^{N} \rX_i \Big| \geq t \Bigg) \leq 2 \exp \Bigg( - \frac{t^2/2}{\sigma^2 + Kt/3} \Bigg)\,.
    \end{align*}
\end{lemma}

\begin{lemma}[Bennett's inequality, {\cite[Theorem $2.9.2$]{Vershynin2018HighDP} }]\label{lem:Bennett}
    Let $\rX_1,\dots, \rX_{N}$ be independent random variables. Assume that $|\rX_i - \E \rX_i| \leq K$ almost surely for every $i$. Then for any $t>0$, we have
    \begin{align*}
        \P \Bigg( \sum_{i=1}^{N} (\rX_i - \E \rX_i) \geq t \Bigg) \leq \exp \Bigg( - \frac{\sigma^2}{K^2} \cdot h \bigg( \frac{Kt}{\sigma^2} \bigg)\Bigg)\,, \notag 
    \end{align*}
    where $\sigma^2 = \sum_{i=1}^{N}\Var(\rX_i)$ is the variance of the sum and $h(u) := (1 + u)\log(1 + u) - u$.
\end{lemma}

\begin{lemma}[Weyl's inequality, \cite{Weyl1912DasAV}]\label{lem:weyl}
Let $\rmA, \rmE \in \R^{m \times n}$ be two real $m\times n$ matrices, then $|\sigma_i(\rmA + \rmE) - \sigma_i(\rmA)| \leq \|\rmE\|$ for every $1 \leq i \leq \min\{ m, n\}$. Furthermore, if $m = n$ and $\rmA, \rmE \in \R^{n \times n}$ are real symmetric, then $|\lambda_i(\rmA + \rmE) - \lambda_i(\rmA)| \leq \|\rmE\|$ for all $1 \leq i \leq n$.
\end{lemma}

\begin{lemma}[Davis-Kahan's $\sin \Theta$ Theorem, {\cite[Theorem 2.7]{Chen2021SpectralMF}, \cite{Davis1970TheRO, Yu2015UsefulVO}} ]\label{lem:Davis_Kahan_sin}
Let $\overline{\rmM}$ and $\rmM = \overline{\rmM} + \rmE$ be two real symmetric $N \times N$ matrices, with eigenvalue decompositions given respectively by
\begin{subequations}
    \begin{align*}
        \overline{\rmM} =&\, \sum\limits_{i=1}^{N}\overline{\lambda}_i \overline{\rvu}_{i} {\overline{\rvu}_{i}}^{\sT} = 
            \begin{bmatrix}
                \overline{\rmU} & \overline{\rmU}_{\perp}
            \end{bmatrix}
            \begin{bmatrix}
                \overline{\Lambda} & \bzero\\
                \bzero & \overline{\Lambda}_{\perp}
            \end{bmatrix}
            \begin{bmatrix}
                {\overline{\rmU}}^{\sT}\\
                {\overline{\rmU}}^{\sT}_{\perp}
            \end{bmatrix}\,,\\
            \rmM =&\, \sum\limits_{i=1}^{N}\lambda_i \rvu_i \rvu_i^{\sT} = 
            \begin{bmatrix}
                \rmU & \rmU_{\perp}
            \end{bmatrix}
            \begin{bmatrix}
                \Lambda & \bzero\\
                \bzero & \Lambda_{\perp}
            \end{bmatrix}
            \begin{bmatrix}
               \rmU^{\sT}\\
                \rmU_{\perp}^{\sT}
            \end{bmatrix}\,.
        \end{align*}
    \end{subequations}
    Here, $\{ \overline{\lambda}_{i}\}_{i=1}^{N}$(resp. $\{\lambda_{i}\}_{i=1}^{N}$) stand for the eigenvalues of $\overline{\rmM}$(resp. $\rmM$), and $\overline{\rvu}_i$(resp. $\rvu_i$) denotes the eigenvector associated $\overline{\lambda}_{i}$(resp. $\lambda_i$). Additionally, for some fixed integer $r\in[N]$, we denote
    \begin{align}
        \overline{\Lambda}\coloneqq \mathrm{diag}\{\overline{\lambda}_1, \dots, \overline{\lambda}_r\}, \quad \overline{\Lambda}_{\perp}\coloneqq&\, \mathrm{diag}\{\overline{\lambda}_{r+1}, \dots, \overline{\lambda}_{N}\}, \notag \\
        \overline{\rmU}\coloneqq  [\overline{\rvu}_1, \dots, \overline{\rvu}_r] \in \R^{N \times r}, \quad \overline{\rmU}_{\perp}\coloneqq&\, [\overline{\rvu}_{r+1}, \dots, \overline{\rvu}_{N}] \in \R^{N \times (N - r)}. \notag 
    \end{align}
    The matrices $\Lambda$, $\Lambda_{\perp}$, $\rmU$, $\rmU_{\perp}$ are defined analogously. Assume that
    \begin{align*}
        \mathrm{eigenvalues}(\overline{\Lambda}) \subseteq [\alpha, \beta]\,,\quad \mathrm{eigenvalues}(\Lambda_{\perp}) \subseteq (-\infty, \alpha - \Delta] \cup [\beta + \Delta, \infty),\quad \alpha, \beta \in \R\,, \Delta >0\,,
    \end{align*}
    and the projection matrices are given by $\rmP_{\rmU} \coloneqq \rmU\rmU^{\sT}$, $\rmP_{\overline{\rmU}} \coloneqq \overline{\rmU} \,\overline{\rmU}^{\sT}$, then one has $\|\rmP_{\rmU} - \rmP_{\overline{\rmU}}\| \leq (2\|\rmE\|/\Delta)$. In particular, suppose that $|\overline{\lambda}_1| \geq |\overline{\lambda}_2| \geq \cdots \geq |\overline{\lambda}_r| \geq |\overline{\lambda}_{r+1}| \geq \cdots |\overline{\lambda}_{N}|$ (resp. $|\lambda_1| \geq \cdots \geq |\lambda_{N}|$). If $\|\rmE\|\leq (1 - 1/\sqrt{2})(|\overline{\lambda}|_{r} - |\overline{\lambda}|_{r+1})$, then one has
    \begin{equation*}
         \|\rmP_{\rmU} - \rmP_{\overline{\rmU}}\| \leq \frac{2\|\rmE\| }{ |\overline{\lambda}_{r}| - |\overline{\lambda}_{r+1}|} \,.
    \end{equation*}
\end{lemma}

\begin{lemma}[Wedin's $\sin \Theta$ Theorem, {\cite[Theorem 2.9]{Chen2021SpectralMF}, \cite{Wedin1972PerturbationBO, Yu2015UsefulVO}}]\label{lem:Wedin_sin}
    Let $\overline{\rmM}$ and $\rmM = \overline{\rmM} + \rmE$ be two $N_{1} \times N_{2}$ real matrices and $N_{1}\geq N_{2}$, with \emph{SVD}s given respectively by 
    \begin{subequations}
    \begin{align}
        \overline{\rmM} =&\, \sum\limits_{i=1}^{N_{1}}\overline{\sigma}_i \overline{\rvu}_{i} {\overline{\rvv}_{i}}^{\sT} = 
            \begin{bmatrix}
                \overline{\rmU} & \overline{\rmU}_{\perp}
            \end{bmatrix}
            \begin{bmatrix}
                \overline{\Sigma} & \bzero & \bzero\\
                \bzero & \overline{\Sigma}_{\perp} & \bzero
            \end{bmatrix}
            \begin{bmatrix}
                {\overline{\rmV}}^{\sT}\\
                \overline{\rmV}^{\sT}_{\perp}
            \end{bmatrix} \notag \\
            \rmM =&\, \sum\limits_{i=1}^{N_{1}}\sigma_i \rvu_i \rvv_i^\sT = 
            \begin{bmatrix}
                \rmU & \rmU_{\perp}
            \end{bmatrix}
            \begin{bmatrix}
                \Sigma & \bzero & \bzero\\
                \bzero & \Sigma_{\perp} & \bzero
            \end{bmatrix}
            \begin{bmatrix}
               \rmV^{\sT}\\
                \rmV_{\perp}^{\sT}
            \end{bmatrix} .\notag
        \end{align}
    \end{subequations}
    Here, $\overline{\sigma}_1 \geq \dots \geq \overline{\sigma}_{N_{1}}$ (resp. $\sigma_1 \geq \dots \geq \sigma_{N_{1}}$) stand for the singular values of $\overline{\rmM}$(resp. $\rmM$), $\overline{\rvu}_i$(resp. $\rvu_i$) denotes the left singular vector associated with the singular value $\overline{\sigma}_i$(resp. $\sigma_i$), and $\overline{\rvv}_i$(resp. $\rvv_i$) denotes the right singular vector associated with the singular value $\overline{\sigma}_i$(resp. $\sigma_i$). In addition, for any fixed integer $r\in[N_{1}]$, we denote 
    \begin{align*}
        \Sigma\coloneqq \mathrm{diag}\{\sigma_1, \dots, \sigma_r\}, \quad \Sigma_{\perp}\coloneqq&\, \mathrm{diag}\{\sigma_{r+1}, \dots, \sigma_{N_{1}}\},\\
        \rmU\coloneqq [\rvu_1, \dots, \rvu_r] \in \R^{N_{1} \times r}, \quad \rmU_{\perp}\coloneqq&\, [\rvu_{r+1}, \dots, \rvu_{N_{1}}] \in \R^{N_{1} \times (N_{1} - r)},\\
        \rmV\coloneqq [\rvv_1, \dots, \rvv_r] \in \R^{N_{2} \times r}, \quad \rmV_{\perp}\coloneqq&\, [\rvv_{r+1}, \dots, \rvv_{N_{2}}] \in \R^{N_{2} \times (N_{2} - r)}.
    \end{align*}
    The matrices $\overline{\Sigma}$, $\overline{\Sigma}_{\perp}$, $\overline{\rmU}$, $\overline{\rmU}_{\perp}$, $\overline{\rmV}$, $\overline{\rmV}_{\perp}$ are defined analogously. If $\rmE = \rmM - \overline{\rmM}$ satisfies $\|\rmE\|\leq \overline{\sigma}_r - \overline{\sigma}_{r+1}$, then with the projection matrices $\rmP_{\rmU} \coloneqq \rmU\rmU^{\sT}$, one has
    \begin{equation*}
        \max \left\{\|\rmP_{\rmU} - \rmP_{\overline{\rmU}}\|, \|\rmP_{\rmV} - \rmP_{\overline{\rmV}}\| \right\} \leq \frac{\sqrt{2} \max \left\{\|\rmE^\sT \overline{\rmU}\|, \|\rmE \overline{\rmV}\| \right\} }{ \overline{\sigma}_r - \overline{\sigma}_{r+1} - \|\rmE\|} \,.
    \end{equation*}
    In particular, if $\|\rmE\|\leq (1 - 1/\sqrt{2})(\overline{\sigma}_{r} - \overline{\sigma}_{r+1})$, then one has
    \begin{equation*}
        \max \left\{\|\rmP_{\rmU} - \rmP_{\overline{\rmU}}\|, \|\rmP_{\rmV} - \rmP_{\overline{\rmV}}\| \right\} \leq \frac{\sqrt{2} \|\rmE\| }{ \overline{\sigma}_{r} - \overline{\sigma}_{r+1}} \,.
    \end{equation*}
\end{lemma}
\end{document}